\documentclass[10pt,a4paper]{article}
\usepackage[utf8]{inputenc}
\usepackage[T1]{fontenc}
\usepackage{geometry}
\usepackage{amsthm}
\usepackage[french,english]{babel}
\usepackage{amssymb}
\usepackage{lmodern}
\usepackage{amsmath,amsfonts}
\usepackage{mathrsfs} 
\usepackage{dsfont}
\usepackage{stmaryrd}
\DeclareMathOperator{\dive}{div} 
\usepackage{graphicx}
\usepackage{fullpage}
\usepackage[nottoc,notlot,notlof]{tocbibind}
\usepackage[colorlinks=true,urlcolor=black,linkcolor=black,citecolor=black]{hyperref}
\numberwithin{equation}{section}
\usepackage{mathtools,array}
\usepackage{verbatimbox}
\usepackage{color}
\usepackage{bm}
\usepackage{tikz-cd}
\usepackage{pst-node}
\usetikzlibrary{matrix}
\usepackage{mathtools}
\usepackage{verbatimbox}
\usepackage{diagbox}
\usepackage{array}
\newcolumntype{C}{>{$\displaystyle} c <{$}}
\usepackage{makecell}
\usepackage{float}
\usepackage{tabu}
\usepackage{cancel}
\usepackage{lscape}
\usepackage[babel=true]{csquotes}
\usepackage{bm}
\usepackage{xcolor}
\makeatletter
\def\env@dmatrix{\hskip -\arraycolsep
	\let\@ifnextchar\new@ifnextchar
	\def\arraystretch{2}%
	\array{*{\c@MaxMatrixCols}{>{\displaystyle}c}}%
}

\makeatother

\begin{document}

	\renewcommand{\thefootnote}{\fnsymbol{footnote}}
	
	\title{Higher Regularity of Weak Limits of Willmore Immersions I}
	\author{Alexis Michelat\footnote{Department of Mathematics, ETH Zentrum, CH-8093 Zürich, Switzerland.}\; and Tristan \selectlanguage{french}Rivière$^{*}$\selectlanguage{english}\setcounter{footnote}{0}}
	\date{\today}
	
	\maketitle
	
	\vspace{-0.5em}
	
	\begin{abstract}
	    We obtain in arbitrary codimension a removability result on the order of singularity of weak limits and bubbles of Willmore immersions measured by the second residue. This permits to reduce significantly the number of possible bubbling scenarii. As a consequence, out of the twelve families of non-planar minimal surfaces in $\mathbb{R}^3$ of total curvature greater than $-12\pi$, only three of them may occur as conformal images of bubbles of Willmore immersions. 
	\end{abstract}

	\tableofcontents
	\vspace{0cm}
	\begin{center}
	\end{center}
	
	\theoremstyle{plain}
	\newtheorem*{theorem*}{Theorem}
	\newtheorem{theorem}{Theorem}[section]
	\newenvironment{theorembis}[1]
	{\renewcommand{\thetheorem}{\ref{#1}$'$}%
		\addtocounter{theorem}{-1}%
		\begin{theorem}}
		{\end{theorem}}
	\renewcommand*{\thetheorem}{\Alph{theorem}}
	\newtheorem{lemme}[theorem]{Lemma}
	\newtheorem*{lemme*}{Lemma}
	\newtheorem{propdef}[theorem]{Definition-Proposition}
	\newtheorem*{propdef*}{Definition-Proposition}
	\newtheorem{prop}[theorem]{Proposition}
	\newtheorem{cor}[theorem]{Corollary}
	\theoremstyle{definition}
	\newtheorem*{definition}{Definition}
	\newtheorem{defi}[theorem]{Definition}
	\newtheorem{rem}[theorem]{Remark}
	\newtheorem*{rem*}{Remark}
	\newtheorem{rems}[theorem]{Remarks}
	\newtheorem{exemple}[theorem]{Example}
	\newtheorem{defi2}{Definition}
	\newtheorem{propdef2}[defi2]{Proposition-Definition}
	\newtheorem{remintro}[defi2]{Remark}
	\newtheorem{remsintro}[defi2]{Remarks}
	\renewcommand\hat[1]{%
		\savestack{\tmpbox}{\stretchto{%
				\scaleto{%
					\scalerel*[\widthof{\ensuremath{#1}}]{\kern-.6pt\bigwedge\kern-.6pt}%
					{\rule[-\textheight/2]{1ex}{\textheight}}%WIDTH-LIMITED BIG WEDGE
				}{\textheight}% 
			}{0.5ex}}%
		\stackon[1pt]{#1}{\tmpbox}
	}
	\parskip 1ex
	\newcommand{\totimes}{\ensuremath{\,\dot{\otimes}\,}}
	\newcommand{\vc}[3]{\overset{#2}{\underset{#3}{#1}}}
	\newcommand{\conv}[1]{\ensuremath{\underset{#1}{\longrightarrow}}}
	\newcommand{\A}{\ensuremath{\vec{A}}}
	\newcommand{\B}{\ensuremath{\vec{B}}}
	\newcommand{\C}{\ensuremath{\mathbb{C}}}
	\newcommand{\D}{\ensuremath{\nabla}}
	\newcommand{\E}{\ensuremath{\vec{E}}}
	\newcommand{\I}{\ensuremath{\mathbb{I}}}
	\newcommand{\Q}{\ensuremath{\vec{Q}}}
	\newcommand{\loc}{\ensuremath{\mathrm{loc}}}
	\newcommand{\z}{\ensuremath{\bar{z}}}
	\newcommand{\hh}{\ensuremath{\mathscr{H}}}
	\newcommand{\h}{\ensuremath{\vec{h}}}
	\newcommand{\vol}{\ensuremath{\mathrm{vol}}}
	\newcommand{\hs}[3]{\ensuremath{\left\Vert #1\right\Vert_{\mathrm{H}^{#2}(#3)}}}
	\newcommand{\R}{\ensuremath{\mathbb{R}}}
	\renewcommand{\P}{\ensuremath{\mathbb{P}}}
	\newcommand{\N}{\ensuremath{\mathbb{N}}}
	\newcommand{\Z}{\ensuremath{\mathbb{Z}}}
	\newcommand{\p}[1]{\ensuremath{\partial_{#1}}}
	\newcommand{\Res}{\ensuremath{\mathrm{Res}}}
	\newcommand{\lp}[2]{\ensuremath{\mathrm{L}^{#1}(#2)}}
	\renewcommand{\wp}[3]{\ensuremath{\left\Vert #1\right\Vert_{\mathrm{W}^{#2}(#3)}}}
	\newcommand{\wpn}[3]{\ensuremath{\Vert #1\Vert_{\mathrm{W}^{#2}(#3)}}}
	\newcommand{\np}[3]{\ensuremath{\left\Vert #1\right\Vert_{\mathrm{L}^{#2}(#3)}}}
	\newcommand{\hardy}[2]{\ensuremath{\left\Vert #1\right\Vert_{\mathscr{H}^{1}(#2)}}}
	\newcommand{\lnp}[3]{\ensuremath{\left| #1\right|_{\mathrm{L}^{#2}(#3)}}}
	\newcommand{\npn}[3]{\ensuremath{\Vert #1\Vert_{\mathrm{L}^{#2}(#3)}}}
	\newcommand{\nc}[3]{\ensuremath{\left\Vert #1\right\Vert_{C^{#2}(#3)}}}
	\renewcommand{\Re}{\ensuremath{\mathrm{Re}\,}}
	\renewcommand{\Im}{\ensuremath{\mathrm{Im}\,}}
	\newcommand{\dist}{\ensuremath{\mathrm{dist}}}
	\newcommand{\diam}{\ensuremath{\mathrm{diam}\,}}
	\newcommand{\leb}{\ensuremath{\mathscr{L}}}
	\newcommand{\supp}{\ensuremath{\mathrm{supp}\,}}
	\renewcommand{\phi}{\ensuremath{\vec{\Phi}}}
	\renewcommand{\H}{\ensuremath{\vec{H}}}
	\renewcommand{\L}{\ensuremath{\vec{L}}}
	\renewcommand{\lg}{\ensuremath{\mathscr{L}_g}}
	\renewcommand{\ker}{\ensuremath{\mathrm{Ker}}}
	\renewcommand{\epsilon}{\ensuremath{\varepsilon}}
	\renewcommand{\bar}{\ensuremath{\overline}}
	\newcommand{\s}[2]{\ensuremath{\langle #1,#2\rangle}}
	\newcommand{\pwedge}[2]{\ensuremath{\,#1\wedge#2\,}}
	\newcommand{\bs}[2]{\ensuremath{\left\langle #1,#2\right\rangle}}
	\newcommand{\scal}[2]{\ensuremath{\langle #1,#2\rangle}}
	\newcommand{\sg}[2]{\ensuremath{\left\langle #1,#2\right\rangle_{\mkern-3mu g}}}
	\newcommand{\n}{\ensuremath{\vec{n}}}
	\newcommand{\ens}[1]{\ensuremath{\left\{ #1\right\}}}
	\newcommand{\lie}[2]{\ensuremath{\left[#1,#2\right]}}
	\newcommand{\g}{\ensuremath{g}}
	\newcommand{\e}{\ensuremath{\vec{e}}}
	\newcommand{\f}{\ensuremath{\vec{f}}}
	\newcommand{\ig}{\ensuremath{|\vec{\mathbb{I}}_{\phi}|}}
	\newcommand{\ik}{\ensuremath{\left|\mathbb{I}_{\phi_k}\right|}}
	\newcommand{\w}{\ensuremath{\vec{w}}}
	\renewcommand{\tilde}{\ensuremath{\widetilde}}
	\newcommand{\vg}{\ensuremath{\mathrm{vol}_g}}
	\newcommand{\im}{\ensuremath{\mathrm{W}^{2,2}_{\iota}(\Sigma,N^n)}}
	\newcommand{\imm}{\ensuremath{\mathrm{W}^{2,2}_{\iota}(\Sigma,\R^3)}}
	\newcommand{\timm}[1]{\ensuremath{\mathrm{W}^{2,2}_{#1}(\Sigma,T\R^3)}}
	\newcommand{\tim}[1]{\ensuremath{\mathrm{W}^{2,2}_{#1}(\Sigma,TN^n)}}
	\renewcommand{\d}[1]{\ensuremath{\partial_{x_{#1}}}}
	\newcommand{\dg}{\ensuremath{\mathrm{div}_{g}}}
	\renewcommand{\Res}{\ensuremath{\mathrm{Res}}}
	\newcommand{\un}[2]{\ensuremath{\bigcup\limits_{#1}^{#2}}}
	\newcommand{\res}{\mathbin{\vrule height 1.6ex depth 0pt width
			0.13ex\vrule height 0.13ex depth 0pt width 1.3ex}}
	\newcommand{\ala}[5]{\ensuremath{e^{-6\lambda}\left(e^{2\lambda_{#1}}\alpha_{#2}^{#3}-\mu\alpha_{#2}^{#1}\right)\left\langle \nabla_{\vec{e}_{#4}}\vec{w},\vec{\mathbb{I}}_{#5}\right\rangle}}
	\setlength\boxtopsep{1pt}
	\setlength\boxbottomsep{1pt}
	\newcommand\norm[1]{%
		\setbox1\hbox{$#1$}%
		\setbox2\hbox{\addvbuffer{\usebox1}}%
		\stretchrel{\lvert}{\usebox2}\stretchrel*{\lvert}{\usebox2}%
	}
	\allowdisplaybreaks
	\newcommand*\mcup{\mathbin{\mathpalette\mcapinn\relax}}
	\newcommand*\mcapinn[2]{\vcenter{\hbox{$\mathsurround=0pt
				\ifx\displaystyle#1\textstyle\else#1\fi\bigcup$}}}
	\def\Xint#1{\mathchoice
		{\XXint\displaystyle\textstyle{#1}}%
		{\XXint\textstyle\scriptstyle{#1}}%
		{\XXint\scriptstyle\scriptscriptstyle{#1}}%
		{\XXint\scriptscriptstyle\scriptscriptstyle{#1}}%
		\!\int}
	\def\XXint#1#2#3{{\setbox0=\hbox{$#1{#2#3}{\int}$ }
			\vcenter{\hbox{$#2#3$ }}\kern-.58\wd0}}
	\def\ddashint{\Xint=}
	\newcommand{\dashint}[1]{\ensuremath{{\Xint-}_{\mkern-10mu #1}}}
	\newcommand\ccancel[1]{\renewcommand\CancelColor{\color{red}}\cancel{#1}}
	\newcommand\colorcancel[2]{\renewcommand\CancelColor{\color{#2}}\cancel{#1}}
	
	\section{Introduction}
	
	The integral of mean-curvature squared (also known as the Willmore energy) is a conformal invariant of submanifolds which originates from elasticity theory (in the work of Germain and Poisson at the dawn of the 19th century \cite{germain}, \cite{poisson}, \cite{dahan}) and appears in several contexts ranging from minimal surfaces theory, general relativity (the Hawking quasi-local mass) to mathematical biology. If $\Sigma$ is a closed Riemann surface, $\phi:\Sigma\rightarrow \R^n$ is an immersion, $g=\phi^{\ast}g_{\R^n}$ is the induced metric, and $\H_g$ is its mean curvature, we have
	\begin{align*}
	W(\phi)=\int_{\Sigma}|\H_g|^2d\vg.
	\end{align*}
	Its critical points are called Willmore surfaces (or Willmore immersions), and the Willmore energy can be seen as a way to measure the complexity of a given immersion. This motivated a lot of work to understand the structure of its minimisers among immersions of a fixed abstract closed surface into Euclidean spaces. Notice that if one extends the definition of the Willmore energy to non-closed surfaces, then minimal surfaces (for which $\H_g=0$ by definition) are obvious absolute minimisers of the Willmore energy. Furthermore, this functional is conformal invariant under conformal transformations which implies that all conformal images of minimal surfaces of $\R^3$, $S^3$ or $\mathbb{H}^3$ in $\R^3$ are Willmore surfaces. In particular, there exists a lot of examples of Willmore surfaces, and one would be interested to study their Moduli Space (it does not restrict to the conformal images of minimal surfaces from a space form, see \emph{e.g.} \cite{pinkall}). 
	
	While the structure of minimisers of the Willmore energy in Euclidean spaces is well-known for spheres (as they just consists in round spheres \cite{willmore1}), it remains elusive for higher genera, with the notable exception of the codimension $1$ genus $1$ case which was completely described as the conformal images of the stereographic projection of the Clifford torus in $S^3$ thanks to the work of Marques and Neves (\cite{marqueswillmore}). Another proof of this result was given recently by Heller-Heller-Ndiaye (\cite{2willmore}).
	
	In this work, we advance the understanding of the moduli space of Willmore immersions by providing a new criterion which permits to show  that a lot of \enquote{bubbles} cannot occur). This is an application of this work and of the classification of Bryant \cite{bryant}  (together with our extension to the branched case \cite{classification}, \cite{sagepaper}). Although we did not give any application here, using the classification of Montiel (\cite{montiels4}), one should be able to rule out additional possibilities of bubbling.

	The main result of the paper is a removability result on the second residue (defined in \cite{beriviere}) of branched Willmore surfaces coming as bubbles of sequences of Willmore immersions. Let us first recall the following result of Rivi\`{e}re and Bernard-Rivi\`{e}re.
	
	\begin{propdef2}[\cite{riviere1}, \cite{beriviere}]
		Let $\phi\in W^{2,2}\cap W^{1,\infty}(D^2)\cap C^{\infty}(D^2\setminus\ens{0})$ be a conformal Willmore immersion of finite total curvature on $D^2$. Then there exists an integer $\theta_0\geq 1$ and $\vec{A}_0\in \C^n\setminus\ens{0}$ such that 
		\begin{align}\label{notalgebraic}
			\left\{\begin{alignedat}{1}
		    &\phi(z)=\Re\left(\vec{A}_0z^{\theta_0}\right)+O\left(|z|^{\theta_0+1}\log|z|\right)\\
			&\p{z}\phi(z)=\frac{\theta_0}{2}\vec{A}_0z^{\theta_0-1}+O\left(|z|^{\theta_0}\log|z|\right),
			\end{alignedat}\right.
   		\end{align}
  		and we say that $\phi$ has a branch point of order $\theta_0\geq 1$ at $z=0$.
		Furthermore, provided the mean curvature $\H$ be not identically zero, there exists an integer $m\leq \theta_0-1$ and $\vec{C}_0\in \C^{n}\setminus\ens{0}$ such that for $\theta_0\geq 2$
		\begin{align}\label{theta2}
			\H=\Re\left(\frac{\vec{C}_0}{z^{m}}\right)+O\left(|z|^{1-m}\log|z|\right),
 		\end{align}
		while for $\theta_0=1$, there exists $\vec{\gamma}_0\in\R^n$ such that 
		\begin{align}\label{theta1}
			\H=\vec{\gamma}_0\log|z|+O(|z|\log|z|).
		\end{align}
		We call $r=\max\ens{m,0}\in \ens{0,\cdots,\theta_0-1}$ the second residue of $\phi$ at the branch point $z=0$. More generally, if $\Sigma$ is a closed Riemann surface, $p_1,\cdots,p_d\in \Sigma$ are fixed distinct points and $\phi:\Sigma\setminus\ens{p_1,\cdots,p_d}\rightarrow \R^n$ is a conformal Willmore immersion of finite total curvature, then we define for all $1\leq j\leq d$ the integers $\theta_0(p_j)\in \N$ to be the order of branch point and $0\leq r(p_j)\leq \theta_0(p_j)-1$ to be the associated residue at $z=0$ of the composition $\phi\circ \psi:D^2\rightarrow \R^n$, for any complex chart $\psi:D^2\rightarrow \Sigma$ such that $\psi(0)=p_i$. This definition does not depend on the chart.
	\end{propdef2}

    \begin{remintro}
    	In algebraic geometry, one would say that an immersion having the expansion \eqref{notalgebraic} admits a branch point of order $\theta_0-1$. However, when $\theta_0=1$, although $\phi$ extends as an immersion at $z=0$, it is not necessarily smooth at $0$. Indeed, the coefficient $\vec{\gamma}_0\in \R^n$ in \eqref{theta1} is not always $0$ (take for example the inversion of the catenoid). Furthermore, with this definition, the multiplicity of the branch point of the compactification of a minimal surface is equal to the multiplicity of the corresponding end. 
    \end{remintro}

    We also recall that under the notations of the previous definition, the first residue $\vec{\gamma}_0(\phi,p_j)\in \R^n$ of an immersion $\phi:\Sigma\setminus\ens{p_1,\cdots,p_d}\rightarrow \R^n$ of finite total curvature is defined for all $1\leq j\leq d$ by 
    \begin{align}\label{residue}
    	\vec{\gamma}_0(\phi,p_j)=\frac{1}{4\pi}\,\Im\int_{\gamma}\partial\H+|\H|^2\partial\phi+2\,g^{-1}\otimes \s{\H}{\h_0}\otimes \bar{\partial}\phi,
    \end{align}
    where $\gamma$ is a smooth closed curve around $p_j$ (and not containing other $p_k$ for $k\neq j$ in the same connected component of the one of $p_j$). These two residues precisely measure the singularity at branched points.

    \begin{theorem*}[Bernard-Rivi\`{e}re \cite{beriviere}]
    	Let $\Sigma$ be a closed Riemann surface, $p_1,\cdots,p_d\in \Sigma$ be fixed distinct points and $\phi:\Sigma\setminus\ens{p_1,\cdots,p_d}\rightarrow \R^n$ be a conformal Willmore immersion of finite total curvature. Then $\phi$ is smooth at $p_j$ if and only if $\vec{\gamma}_0(p_j)=0$ and $r(p_j)=0$. 
    \end{theorem*}

    The main result of this paper is an improvement of the regularity of weak limits of Willmore immersions thanks to a removability result on the second residue. We state it in the simplest setting under which the quantization holds, but see the comments after the statement of the theorem. 
	
	\begin{theorem}\label{ta}
		Let $\Sigma$ be a closed Riemann surface, $\{\phi_k\}_{k\in\N}\subset \mathrm{Imm}(\Sigma,\R^n)$ be a sequence of Willmore immersions and assume that the conformal class of $\{\phi_k\}_{k\in \N}$ stays within a compact subset of the Moduli Space and that
		\begin{align}\label{energybound}
		\sup_{k\in \N}W(\phi_k)<\infty.
		\end{align}
		Let $\phi_{\infty}:\Sigma\rightarrow \R^n$ be the weak sequential limit of $\{\phi_k\}_{k\in \N}$ (up to the composition by suitable chosen sequence of conformal transformations in the target and diffeomorphisms in the domain) and $\vec{\Psi}_i:S^2\rightarrow \R^n$, $\vec{\chi}_j:S^2\rightarrow \R^n$ be the bubbles such that
		\begin{align}\label{energyidentity}
		\lim\limits_{k\rightarrow \infty}W(\phi_k)=W(\phi_{\infty})+\sum_{i=1}^{p}W(\vec{\Psi}_i)+\sum_{j=1}^{q}(W(\vec{\xi}_j)-4\pi\theta_j),
		\end{align}
		where $\theta_0(\vec{\xi}_j,p_j)\geq 1$ is the multiplicity of $\vec{\chi}_j$ at some point $p_j\in \vec{\xi}_j(S^2)\subset \R^n$. Then at every branch point $p$ of $\phi_{\infty}, \vec{\Psi}_i$ or $\vec{\xi_j}$ of multiplicity $\theta_0=\theta_0(p)\geq 2$, the second residue $r(p)$ satisfies the inequality $r(p)\leq \theta_0-2$.
	\end{theorem}

    Recall that for all $p\in \R^n$, the multiplicity of a branched immersion is defined by
    \begin{align*}
    \theta_0(\phi,p)=\lim_{r\rightarrow 0}\frac{\mathrm{Area}(\phi(\Sigma))\cap B_r(p))}{\pi r^2}\in \N.
    \end{align*}
    Furthermore, the Li-Yau inequality (\cite{lieyau}) implies that for all $p\in \R^n$, 
    $
    	W(\phi)\geq 4\pi\,\theta_0(\phi,p),
    $
    so the last component on the right-hand side of \eqref{energyidentity} is non-negative.
    
    We also emphasize that \cite{classification} implies that the first residue of the branched Willmore surfaces $\phi_{\infty},\vec{\Psi}_i$ and $\vec{\chi}_j$ defined in \eqref{residue} all vanish. This condition implies in particular that whenever one of these branched Willmore surfaces is conformally minimal, the flux of the associated minimal surface must vanish. This will be of importance in Theorem \ref{tb} (see also the Figure \ref{figure}).
    
    \begin{remsintro}
    	\begin{enumerate}
    	\item[(1)]
    	Theorem \ref{ta} should also hold in the general case of degenerating conformal class for which the quantization holds, as treated by Laurain-Rivi\`{e}re in \cite{quantamoduli}. Indeed, thanks to another results of these authors (\cite{lauriv1}), under the hypothesis of uniformly bounded total curvature \eqref{energybound} there exists an finite atlas of complex charts $(U_j,\psi_j)_{1\leq j\leq N}$ of $\Sigma$ independent of $k$ such that for all $k\in \N$ and $1\leq j\leq N$, the conformal factor $\lambda_k^j$ of $\phi_k\circ \psi_j^{-1}:\psi(U_j)\subset \C\rightarrow \R^n$ satisfies for some universal constant $C_0=C_0(n)>0$ (independent of $k$ and $j$)
    	\begin{align}\label{controlconf}
    		\npn{\D \lambda_k^j}{2,\infty}{\psi_j(U_j)}\leq C_0\,W(\phi_k),
    	\end{align} 
    	where the conformal factor $\lambda_k^j$ of $\phi_k\circ \psi_j^{-1}$ is explicitly given by
    	\begin{align*}
    		(\phi_k\circ \psi_j^{-1})^{\ast}g_{\R^n}=e^{2\lambda_k^j}|dz|^2,
    	\end{align*}
    	and $|dz|^2$ is the flat metric on the open subset $\psi_j(U_j)\subset \C$. In particular, as the estimates that we obtain are taken in domains shrinking to a point, we can assume that we work in one of the charts $(U_j,\psi_j)$ for some $1\leq j\leq N$ (to have the control \eqref{controlconf} of the $L^{2,\infty}$ norm of the gradient of the conformal factor of $\phi_k$).
    	\item[(2)] More generally, this theorem also holds for a sequence of immersions in the viscosity method (see \cite{eversion}). The details of the proof of this claim can be found in \cite{blow-up2}
    \end{enumerate}
    \end{remsintro}
     
     For conformally minimal Willmore surfaces, the second residue has the following interpretation. Let $\phi:\Sigma\setminus\ens{p_1,\cdots,p_d}\rightarrow \R^3$ be a \emph{non-planar} complete minimal surface with finite total curvature, and fix $1\leq j\leq d$. There exists integers $m_j,k_j\geq 1$ such that for all complex chart $\psi:D^2\rightarrow \Sigma$ such that $\psi(0)=p_j$, there exists $\vec{A}_0,\vec{A}_1\in \C^n\setminus\ens{0}$ such that 
    \begin{align}\label{weierstrass}
    	\phi\circ \psi(z)=\Re\left(\frac{\vec{A}_0}{z^{m_j}}+\frac{\vec{A}_1}{z^{m_j-k_j}}\right)+O(|z|^{k_j-m_j+1}).
    \end{align}
    Then a quick computation shows that the inversion $\vec{\Psi}=\iota\circ \phi$ of $\phi$ has a branch point of order $\theta_0(p_j)=m_j$ at $p_j$ and its second residue satisfies
    \begin{align*}
    	r(\phi,p_j)=\max\ens{m_j-k_j,0},
    \end{align*}
    provided $\vec{A}_0,\vec{A}_1$ be linearly \emph{independent}.
    Using the classification of L\'{o}pez of complete minimal surface with total curvature greater than $-12\pi$ (\cite{lopez}), we can describe the possible bubbles appearing in blow-ups of Willmore immersions.
    	
	\begin{theorem}\label{tb}
		Let $\phi:S^2\setminus\ens{p_1,\cdots,p_d}\rightarrow \R^3$ be a \emph{non-planar} complete minimal surface with finite total curvature arising as a bubble of a sequence of Willmore immersions (or as a conformal image of this bubble). Then  
		\begin{align*}
			\int_{S^2}K_gd\vg\leq -8\pi.
		\end{align*}
		with equality if and only if $\phi$ is either a minimal sphere with exactly one end of multiplicity $5$ whose Gauss map is ramified at its end, or $\phi$ is a minimal sphere with one flat end and one end of multiplicity $3$. 
	\end{theorem}
    Let $\Sigma$ be a closed Riemann surface. Recall that the flux of a complete minimal immersion (of finite total curvature) $\phi:\Sigma\setminus\ens{p_1,\cdots,p_d}\rightarrow \R^n$  at $p_j$ is defined by 
    \begin{align*}
    	\mathrm{Flux}(\phi,p_j)=\frac{1}{4\pi}\Im\int_{\gamma}\partial\phi,
    \end{align*}
    where $\gamma$ is as in \eqref{residue}. Thanks to \cite{classification}, if $\vec{\Psi}:\Sigma\rightarrow \R^n$ is the inversion at $0$ of $\phi$, we have
    \begin{align*}
    	\vec{\gamma}_0(\vec{\Psi},p_j)=\mathrm{Flux}(\phi,p_j)
    \end{align*}
    and for the Willmore immersions of \ref{ta}, as the first residue vanishes, the flux of the associated minimal surfaces also vanishes. 
    
    For the proof of this result and the Weierstrass data of the  surfaces classified by L\'{o}pez (\cite{lopez}), see the Appendix \ref{proofTB}.
    Notice that out of these twelve families, only five do not have flux, and the two others do not satisfy the condition on the second residue of Theorem \ref{ta}. In other words, only L\'{o}pez surfaces of type II, VIII and IX are admissible (see the Figure \ref{figure}). Furthermore, notice that they all have flux or a positive second residue. Although the Catenoid and the Enneper surfaces are well known (as they are the only complete minimal surfaces of total curvature $-4\pi$), the trinoid is a complete minimal surface with three catenoid ends (there exists a whole family of such minimal surfaces with an arbitrary number of catenoid ends constructed by Jorge-Meeks \cite{jorge}). 
    
    We describe in the following table the twelve families of complete minimal surface of $\R^3$ of total curvature larger than $-12\pi$. Each time that they have non-zero flux (using \cite{classification}) or second residue $r(p)=\theta_0(p)-1$ at an end of multiplicity $\theta_0(p)\geq 2$, they cannot occur. We denote by $d\geq 1$ the number of ends, and by $m_1,\cdots,m_d$ and $r_1,\cdots,r_d$ the multiplicities of the ends and associated second residues. 
    {\tabulinesep=0.6mm
    	\begin{figure}[H]
    		\begin{center}
    			\begin{tabu}{|c|c|c|c|c|c|c|}
    				\hline
    				\begin{tabular}{@{}c@{}}Minimal\\ surface \end{tabular} & \begin{tabular}{@{}c@{}}Total \\ curvature\end{tabular}& \begin{tabular}{@{}c@{}}Non-zero \\ flux?\end{tabular} & \begin{tabular}{@{}c@{}}Number \\ of ends\end{tabular} & \begin{tabular}{@{}c@{}}Multiplicities\\ of ends\end{tabular} & \begin{tabular}{@{}c@{}}Second\\ residues \end{tabular} & \begin{tabular}{@{}c@{}}Possible \\ Willmore bubble?\end{tabular}\\
    				\hline
    				\begin{tabular}{@{}c@{}}Catenoid\\  \end{tabular} & $-4\pi$ & Yes & $d=2$ & \begin{tabular}{@{}c@{}}  $m_1=1$\\ $m_2=1$  \end{tabular}  & \begin{tabular}{@{}c@{}} $r_1=0$\\ $r_2=0$ \end{tabular} & No \\
    				\hline
    				\begin{tabular}{@{}c@{}}Enneper \\ surface\end{tabular} & $-4\pi$ & No & $d=1$ & $m_1=3$ & $r_1=2$ & No\\
    				\hline
    				\begin{tabular}{@{}c@{}}Trinoid\\  \end{tabular} & $-8\pi$ & Yes & $d=3$ & \begin{tabular}{@{}c@{}} $m_j=1$\\ $1\leq j\leq 3$  \end{tabular} & \begin{tabular}{@{}c@{}} $r_j=0$\\ $1\leq j\leq 3$  \end{tabular} & No\\
    				\hline
    				\begin{tabular}{@{}c@{}}L\'{o}pez\\ surface I\end{tabular} & $-8\pi$ & No & $d=1$ & $m_1=5$ & $r_1=4$ & No\\
    				\hline
    				\begin{tabular}{@{}c@{}}L\'{o}pez\\ surface II\end{tabular} & $-8\pi$ & No & $d=1$ & $m_1=5$ & $r_1=3$ & Yes\\
    				\hline
    				\begin{tabular}{@{}c@{}}L\'{o}pez\\ surface III\end{tabular} & $-8\pi$ & Yes & $d=2$ & \begin{tabular}{@{}c@{}} $m_1=2$\\ $m_2=2$  \end{tabular} & \begin{tabular}{@{}c@{}} $r_1=1$\\ $r_2=1$  \end{tabular} & No\\
    				\hline
    				\begin{tabular}{@{}c@{}}L\'{o}pez\\ surface IV\end{tabular} & $-8\pi$ & Yes & $d=2$ & \begin{tabular}{@{}c@{}} $m_1=2$\\ $m_2=2$  \end{tabular} & \begin{tabular}{@{}c@{}} $r_1=1$\\ $r_2=1$  \end{tabular} & No\\
    				\hline
    				\begin{tabular}{@{}c@{}}L\'{o}pez\\ surface V\end{tabular} & $-8\pi$ & Yes & $d=2$ & \begin{tabular}{@{}c@{}} $m_1=2$\\ $m_2=2$  \end{tabular} & \begin{tabular}{@{}c@{}} $r_1=0$\\ $r_2=0$  \end{tabular} & No\\
    				\hline
    				\begin{tabular}{@{}c@{}}L\'{o}pez\\ surface VI\end{tabular} & $-8\pi$ & Yes & $d=2$ & \begin{tabular}{@{}c@{}} $m_1=1$\\ $m_2=3$  \end{tabular} & \begin{tabular}{@{}c@{}} $r_1=0$\\ $1\leq r_2\leq 2$  \end{tabular} & No \\
    				\hline
    				\begin{tabular}{@{}c@{}}L\'{o}pez\\ surface VII\end{tabular} & $-8\pi$ & Yes & $d=2$ & \begin{tabular}{@{}c@{}} $m_1=1$\\ $m_2=3$  \end{tabular} & \begin{tabular}{@{}c@{}} $r_1=0$\\ $r_2=2$  \end{tabular} & No\\
    				\hline
    				\begin{tabular}{@{}c@{}}L\'{o}pez\\ surface VIII\end{tabular} & $-8\pi$  & No & $d=2$ & \begin{tabular}{@{}c@{}} $m_1=1$\\ $m_2=3$  \end{tabular} & \begin{tabular}{@{}c@{}} $r_1=0$\\ $r_2=1$  \end{tabular} & Yes \\
    				\hline
    				\begin{tabular}{@{}c@{}}L\'{o}pez\\ surface IX\end{tabular} & $-8\pi$ & No & $d=2$ & \begin{tabular}{@{}c@{}} $m_1=1$\\ $m_2=3$  \end{tabular} & \begin{tabular}{@{}c@{}} $r_1=0$\\ $r_2=1$  \end{tabular} & Yes \\
    				\hline
    			\end{tabu}
    		\end{center}
    		\caption{Geometric properties of complete minimal surface with total curvature greater than $-12\pi$}\label{figure}
    	\end{figure}
    
    We remark that L\'{o}pez surfaces of type VI are indexed by a real parameter $c\in \R\setminus\ens{-1}$, and all have second residue $r_2=2$ at the end of multiplicity $3$, with the exception of the surface corresponding to the parameter $c=0$ for which the second residue is equal to $r_2=1$.

    As we expect weak limits of Willmore immersions to satisfy the Willmore equation in the weak sense, it is likely that one can improve Theorem \ref{ta} to show that the second residue vanishes. The following observation confirms this intuition.
    
    \begin{theorem}\label{td}
    	A branched Willmore immersion of finite total curvature is smooth if and only if it satisfies the Willmore equation everywhere in the distributional sense.
    \end{theorem}

    The meaning of the distributional Willmore equation is not completely straightforward due to the presence of singularities which are not locally integrable, and we refer to Section \ref{sec7} for more details.     
    
    Finally, on the analysis side, the proof builds on the previous work of Rivi\`{e}re (\cite{riviere1}, \cite{rivierecrelle}), Bernard-Rivi\`{e}re (\cite{beriviere}, \cite{quanta}) and Laurain-Rivi\`{e}re (\cite{angular}, \cite{quantamoduli}, \cite{lauriv1}) and on the general philosophy of integration by compensation and geometric analysis on surfaces (including \cite{meyer}, \cite{toro}, \cite{muller}, \cite{helein}). The proof of the main Theorem \ref{ta} relies mostly the improvement of the energy quantization (\cite{quanta}) to obtain a $L^{2,1}$ no-neck energy (Theorem \ref{improvedquanta}), and on the following result, which shows that the \emph{multiplicity} of weakly converging sequence of immersions becomes eventually constant to an integer. This is a significant improvement of the fundamental work of M\"{u}ller-\v{S}ver\'{a}k (\cite{muller}).
    
    \begin{theorem}\label{integer0} Let $n\geq 3$ be a fixed integer. There exists a universal constant $C(n)>0$ with the following property.
    	Let $\{\phi_k\}_{k\in \N}$ be a sequence of smooth conformal immersions from the disk $B(0,1)\subset \C$ into $\R^n$ and $\ens{\rho_k}_{k\in \N}\subset (0,1)$ be such that $\rho_k\conv{k\rightarrow \infty}0$, $\Omega_k=B(0,1)\setminus \bar{B}(0,\rho_k)$ and define for all $0<\alpha<1$ the sub-domain $\Omega_k(\alpha)=B_{\alpha}\setminus \bar{B}_{\alpha^{-1}\rho_k}(0)$. For all $k\in \N$, let
    	\begin{align*}
    	\lambda_k=\log\left(\frac{|\D\phi_k|}{\sqrt{2}}\right)
    	\end{align*}
    	be the conformal factor of $\phi_k$.
        Assume that 
    	\begin{align*}
    	\sup_{k\in \N}\np{\D\lambda_k}{2,\infty}{\Omega_k}<\infty,\qquad \lim_{\alpha\rightarrow 0}\limsup_{k\rightarrow \infty}\int_{\Omega_k(\alpha)}|\D\n_k|^2dx=0
    	\end{align*} 
    	and that there exists a $W^{2,2}_{\mathrm{loc}}(B(0,1)\setminus\ens{0})\cap  C^{\infty}(B(0,1)\setminus\ens{0})$ immersion $\phi_{\infty}$ such that
    	\begin{align*}
    	\log|\D\phi_{\infty}|\in L^{\infty}_{\mathrm{loc}}(B(0,1)\setminus\ens{0})
    	\end{align*} 
    	such that $\phi_k\conv{k\rightarrow  \infty}\phi_{\infty}$ in $C^l_{\mathrm{loc}}(B(0,1)\setminus\ens{0})$ (for all $l\in \N$). Furthermore, there exists an integer $\theta_0\geq 1$,  $\mu_k\in W^{1,(2,1)}(B(0,1))$ such that 
    	\begin{align*}
    	\np{\D\mu_k}{2,1}{B(0,1)}\leq C(n)\int_{\Omega_k}|\D\n_k|^2dx
    	\end{align*}
    	and a harmonic function $\nu_k$ on $\Omega_k$ such that $\nu_k=\lambda_k$ on $\partial B(0,1)$, $\lambda_k=\mu_k+\nu_k$ on $\Omega_k$ and such that for all $0<\alpha<1$ and such that for all $k\in \N$ sufficiently large 
    	\begin{align*}
    	\np{\D(\nu_k-(\theta_0-1)\log|z|)}{2,1}{\Omega_k(\alpha)}\leq C(n)\left(\sqrt{\alpha}\np{\D\lambda_k}{2,\infty}{\Omega_k}+\int_{\Omega_k}|\D\n_k|^2dx\right)
    	\end{align*}
    	for some universal constant $\Gamma_{10}=\Gamma_{10}(n)$. Finally, we have for all $\rho_k\leq r_k\leq 1$ and $k$ large enough
    	\begin{align*}
    	\frac{1}{2\pi}\int_{\partial B_{r _k}}\ast\, d\nu_k=\theta_0-1.
    	\end{align*}
    \end{theorem}

    %We will not discuss further the technical details of the proof of Theorem \ref{ta} to prevent the discussion to become too technical.
    
    \textbf{Added in proof.} Nicolas Marque showed that \emph{simple minimal  bubblings} cannot occur in $\R^3$ (see \cite{marque}). These are the bubblings where at each concentration point a single minimal bubble blows up.
 
    \renewcommand{\thetheorem}{\thesection.\arabic{theorem}}

    \section{Uniform control of the conformal factor in necks}\label{sec2}
    
    For the definitions related to Lorentz spaces, we refer the reader to the Appendix (Section \ref{basiclorentz}).

    In this section we obtain a refinement of Lemma V.$3$ of \cite{quanta}.
    
    \begin{theorem}\label{neckfine}
    	There exists a positive real number $\epsilon_1=\epsilon_1(n)$ with the following property. Let $0<2^6r<R<\infty$ be fixed radii and $\phi:\Omega=B(0,R)\setminus\bar{B}(0,r)\rightarrow \R^n$ be a weak immersion of finite total curvature such that 
    	\begin{align}\label{I0}
    		\np{\D\n}{2,\infty}{\Omega}\leq \epsilon_1(n).
    	\end{align}
    	For all $\left(\dfrac{r}{R}\right)^{\frac{1}{2}}<\alpha<1$, define $\Omega_{\alpha}=B(0,\alpha R) \setminus  \bar{B}(0,\alpha^{-1}r)$. Then there exists a universal constant $\Gamma_0=\Gamma_0(n)$ and $d\in \R$ (depending on $r,R,\phi$ but not on $\alpha$) such that for all $\left(\dfrac{r}{R}\right)^{\frac{1}{3}}<\alpha<\dfrac{1}{4}$, we have
    	\begin{align}\label{I1}
    		\np{\D(\lambda-d\log|z|)}{2,1}{\Omega_{\alpha}}\leq \Gamma_0\left(\sqrt{\alpha}\np{\D\lambda}{2,\infty}{\Omega}+\int_{\Omega}|\D\n|^2dx\right)
    	\end{align} 
    	and for all $r\leq \rho<R$, we have
    	\begin{align}\label{I2}
    		\left|d-\frac{1}{2\pi}\int_{\partial B_{\rho}}\partial_{\nu}\lambda \,d\mathscr{H}^1\right|\leq \Gamma_0\left(\int_{B_{\max\ens{\rho,2r}}\setminus \bar{B}_r(0)}|\D\n|^2dx+\frac{1}{\log\left(\frac{R}{\rho}\right)}\int_{\Omega}|\D\n|^2dx\right) 
    	\end{align}
    	In particular, there exists a universal constant $\Gamma_0'=\Gamma_0(n)$ such that for all $\left(\dfrac{r}{R}\right)^{\frac{1}{3}}<\alpha<\dfrac{1}{4}$, there exists $A_{\alpha}\in \R$ such that
      	\begin{align}\label{I3}
    		\np{\lambda-d\log|z|-A_{\alpha}}{\infty}{\Omega_{\alpha}}\leq \Gamma_0'\left(\sqrt{\alpha}\np{\D\lambda}{2,\infty}{\Omega}+\int_{\Omega}|\D\n|^2dx\right).
    	\end{align}    	
    \end{theorem}
    The proof relies on the strategy developed in \cite{quanta} (and the lemmas from \cite{angular}, \cite{quantamoduli} for the Lemmas \ref{l2linfty} and \ref{l21l2}) and the following two lemmas, which will allow us to move from a $L^{2,\infty}$ bound to a $L^{2,1}$ bound in a quantitative way.
    
    \begin{lemme}\label{l2linfty}
    	Let $u:B_R\setminus \bar{B}_r(0)\rightarrow \R$ be a harmonic function such that for some $\rho_0\in (r,R)$
    	\begin{align*}
    		\int_{\partial B_{\rho_{0}}}\partial_{\nu}u \,d\mathscr{H}^1=0.
          	\end{align*}
    	Then there exists a universal constant $\Gamma_1>0$ (independent of $0<4r<R<\infty$) such that for all $\left(\dfrac{r}{R}\right)^{\frac{1}{2}}<\alpha<\dfrac{1}{2}$, we have
    	\begin{align*}
    		\np{\D u}{2}{B_{\alpha R}\setminus \bar{B}_{\alpha^{-1}r}(0)}\leq 
    		{\Gamma_1}\np{\D u}{2,\infty}{B_{R}\setminus \bar{B}_r(0)}.
    	\end{align*}
    \end{lemme}
    \begin{proof}
    	First, we show that for all $\alpha^{-1}r\leq \rho\leq \alpha R$, and for all $0<\alpha<\dfrac{1}{2}$ we have
    	\begin{align}\label{i1}
    		\np{\D u}{\infty}{\partial B_{\rho}(0)}\leq \frac{4}{\log(2)}\sqrt{\frac{3}{\pi}}\frac{1}{(1-\alpha)\rho}\np{\D u}{2,\infty}{B_{\alpha^{-1}\rho}\setminus \bar{B}_{\alpha \rho}}.
    	\end{align}
    	By a slight abuse of notation, we will write $r$ instead of $\rho$ in the following estimates.
    	
    	As $0<\alpha<\dfrac{1}{2}$, we have for all $x\in \partial B_r(0)$, the inclusion $B(x,(1-\alpha)r)\subset B_{\alpha^{-1}r}\setminus \bar{B}_{\alpha r}(0)$. Therefore, thanks to the mean value property, we have for all $0<\beta<(1-\alpha)r$
    	\begin{align}\label{mean}
    		\D u(x)=\frac{1}{2\pi \beta}\int_{\partial B (x,\beta)}\D u(y)\,d\mathscr{H}^1(y).
    	\end{align}
    	Now, thanks to the co-area formula, we have (if $I_{\alpha}(r)=(\dfrac{(1-\alpha)}{2}r,(1-\alpha)r)$)
    	\begin{align*}
    		&\int_{B(x,(1-\alpha)r)\setminus \bar{B}(x,(1-\alpha)r/2)}|\D u(y)|dy=\int_{\frac{(1-\alpha)r}{2}}^{(1-\alpha)r}\left(\int_{\partial B(x,\beta)}|\D u(y)|\,d\mathscr{H}^1(y)\right)d\beta\\
    		&\geq \inf_{\beta\in I_{\alpha}(r)}\left(\beta \int_{\partial B(x,\beta)}|\D u(y)|d\mathscr{H}^1(y)\right)\int_{\frac{(1-\alpha)r}{2}}^{(1-\alpha)r}\frac{d\beta}{\beta}
    		=\log(2)\inf_{\beta\in I_{\alpha}(r)}\left(\beta \int_{\partial B(x,\beta)}|\D u(y)|d\mathscr{H}^1(y)\right)
    	\end{align*}
    	Therefore, there exists $\beta\in \left(\frac{(1-\alpha)r}{2},(1-\alpha)r\right)$ (notice that this shows that the limiting values $\rho=\alpha^{-1}r$ and $\rho=\alpha R$ are admissible) such that
    	\begin{align*}
    		\beta\int_{\partial B(x,\beta)}|\D u(y)|\,d\mathscr{H}^1(y)\leq \frac{1}{\log(2)} \int_{B(x,(1-\alpha)r)\setminus \bar{B}(x,(1-\alpha)r/2)}|\D u(y)|dy
    	\end{align*}
    	or
    	\begin{align}\label{ineq1}
    		\frac{1}{2\pi \beta}\int_{\partial B(x,\beta)}|\D u(y)|d\mathscr{H}^1(y)\leq \frac{1}{2\pi \log(2)\beta^2}\int_{B(x,(1-\alpha)r)\setminus \bar{B}(x,(1-\alpha)r/2)}|\D u(y)|dy.
    	\end{align}
    	Now, notice that
    	\begin{align}\label{ineq2}
    		\np{\mathrm{1}_{B(x,(1-\alpha)r)\setminus \bar{B}(x,(1-\alpha)r/2)}}{2,1}{\R^2}&=4\int_{0}^{\infty}\left(\leb^2(B(x,(1-\alpha)r)\setminus \bar{B}(x,(1-\alpha)r/2))\cap\ens{x: 1>t}\right)^{\frac{1}{2}}dt\nonumber\\
    		&=2\sqrt{3\pi}(1-\alpha)r.
    	\end{align}
    	Furthermore, as $\beta>\dfrac{(1-\alpha)r}{2}$, we have
    	\begin{align}\label{ineq3}
    		\frac{1}{\beta^2}\leq \frac{4}{(1-\alpha)^2r^2}
    	\end{align}
        Therefore, we have by the mean value property \eqref{mean}, the inequalities \eqref{ineq1}, \eqref{ineq2}, \eqref{ineq3} and the duality $L^{2,1}/L^{2,\infty}$
    	\begin{align*}
    		|\D u(x)|&\leq \frac{1}{2\pi \beta}\int_{\partial B(x,\beta)}|\D u(y)|\,d\mathscr{H}^1(y)\\
    		&\leq \frac{2}{\pi\log(2)(1-\alpha)^2r^2}\np{\mathrm{1}_{B(x,(1-\alpha)r\setminus \bar{B}(x,(1-\alpha)r/2)}}{2,1}{\R^2}\np{\D u}{2,\infty}{B(x,(1-\alpha)r)\setminus \bar{B}(x,(1-\alpha)r/2)}\\
    		&\leq \frac{4}{\log(2)}\sqrt{\frac{3}{\pi}}\frac{1}{(1-\alpha)r}\np{\D u}{2,\infty}{B_{\alpha^{-1}r}\setminus \bar{B}_{\alpha r}(0)}.
    	\end{align*}
    	As $x\in \partial B_{r}(0)$ was arbitrary, this proves the inequality \eqref{i1}. Now, as $u$ is harmonic, there exists $\ens{a_n}_{n\in \Z}\subset \C$ such that
    	\begin{align*}
    		u(\rho,\theta)=a_0+d\log \rho+\sum_{n\in \Z^{\ast}}^{}\left(a_n\rho^n+\bar{a_{-n}}\rho^{-n}\right)e^{in\theta},
    	\end{align*}
    	which implies by the hypothesis that
    	\begin{align}\label{logarithm}
    		0=\int_{\partial B_{\rho_0}}\partial_{\nu}u\,d\mathscr{H}^1=2\pi d
    	\end{align}
    	so that for all $r<\rho<R$
    	\begin{align*}
    		\int_{\partial B_{\rho}}\partial_{\nu}u=0.
    	\end{align*}
    	Therefore, integrating by parts, we find
    	\begin{align}\label{i2}
    		\int_{B_{\alpha R}\setminus \bar{B}_{\alpha^{-1}r}}|\D u(x)|^2dx&=\int_{\partial B_{\alpha R}}\partial_{\nu} u\,u\,d\mathscr{H}^1-\int_{\partial B_{\alpha^{-1}r}}\partial_{\nu}u\,u\,d\mathscr{H}^1\nonumber\\
    		&=\int_{\partial B_{\alpha R}}\partial_{\nu}u(u-\bar{u}_{\alpha R})d\mathscr{H}^1-\int_{B_{\alpha^{-1}r}}\partial_{\nu} u\left(u-\bar{u}_{{\alpha^{-1}r}}\right)d\mathscr{H}^1
    	\end{align}
    	where $\displaystyle \bar{u}_{\rho}=\dashint{\partial B_{\rho}}u\,d\mathscr{H}^1$ is the average of $u$ on $\rho$, for all $r<\rho<R$. 
    	
    	Now, if $C_0=C_0(H^{\frac{1}{2}}(S^1),L^1(S^1))$ is the constant of the injection $H^{\frac{1}{2}}(S^1)\hookrightarrow L^1(S^1)$ (for the norm defined by the $L^2$ norm of the harmonic extension), we get by \eqref{i1} for all $r<\rho<R$
    	\begin{align*}
    		\left|\int_{\partial B_{\rho}}\partial_{\nu}u\left(u-\bar{u}_{\rho}\right)d\mathscr{H}^1\right|&\leq \np{\D u}{\infty}{\partial B_{\rho}}\np{u-\bar{u}_{\rho}}{1}{\partial B_{\rho}}\\
    		&\leq \frac{4}{\log(2)}\sqrt{\frac{3}{\pi}}\frac{1}{(1-\alpha)\rho}\np{\D u}{2,\infty}{B_R\setminus \bar{B}_r(0)}\times C_0\rho\hs{u}{\frac{1}{2}}{\partial B_{\rho}}\\
    		&\leq \frac{4}{\log(2)}\sqrt{\frac{3}{\pi}}\frac{1}{(1-\alpha)}C_0\np{\D u}{2,\infty}{B_R\setminus \bar{B}_r(0)}\np{\D u}{2}{B_{\alpha R}\setminus \bar{B}_{\alpha^{-1}r}(0)}
    	\end{align*}
    	which implies by \eqref{i2}  that
    	\begin{align*}
    		\np{\D u}{2}{B_{\alpha R}\setminus \bar{B}_{\alpha^{-1}r}}\leq \frac{8}{\log(2)}\sqrt{\frac{3}{\pi}}\frac{1}{(1-\alpha)}C_0\np{\D u}{2,\infty}{B_R\setminus \bar{B}_r(0)}
    	\end{align*}
    	and this concludes the proof of the Lemma.
    \end{proof}
In the following Lemma we obtain a slight improvement from \cite{quantamoduli} and generalise it to a $W^{2,1}$ estimate, that will be used in the proof of Theorem \ref{improvedquanta}.
    \begin{lemme}\label{l21l2}
    	Let $0<4r<R<\infty$ be fixed radii, and $u:\Omega=B_R\setminus \bar{B}_r(0)\rightarrow \R$ be a harmonic function such that for some $\rho_0\in (r,R)$
    	\begin{align*}
    	\int_{\partial B_{\rho_0}}\partial_{\nu}u \,d\mathscr{H}^1=0.
    	\end{align*}
    	Then for all $\left(\dfrac{r}{R}\right)^{\frac{1}{2}}<\alpha<1$, we have
    	\begin{align*}
    		&\np{\D u}{2,1}{B_{\alpha R}\setminus \bar{B}_{\alpha^{-1}r}(0)}\leq 32\sqrt{\frac{2}{15}}\frac{\alpha}{1-\alpha} \np{\D u}{2}{B_R\setminus \bar{B}_r(0)},\\
    		&\np{\D^2u}{1}{B_{\alpha R}\setminus \bar{B}_{\alpha^{-1}r}(0)}\leq 32\sqrt{\frac{\pi}{15}}\frac{\alpha}{1-\alpha}\np{\D u}{2}{B_{R}\setminus \bar{B}_r(0)}.
    	\end{align*}
        \end{lemme}
    \begin{proof}
    	As $u$ is harmonic on $B_R\setminus \bar{B}_r(0)$, there exists $\ens{a_n}_{n\in \Z}\subset \C$ and $d\in \R$ such that
    	\begin{align*}
    		u(z)=a_0+d\log|z|+2\,\Re\left(\sum_{n\in \Z}a_nz^n\right).
    	\end{align*}
    	Thanks to \eqref{logarithm}, we deduce that $d=0$. Furthermore, taking polar coordinates $z=\rho e^{i\theta}$, we have the identity
    	\begin{align}\label{gradient}
    		|\D u|^2=4|\p{z}u|^2=4\left|\sum_{n\in \Z^{\ast}}n a_nz^{n-1}\right|^2=4\sum_{n,m\in \Z^{\ast}}nm\,a_n\bar{a}_m\rho^{n+m-2}e^{i(n-m)\theta}.
    	\end{align}
    	This implies by the inequality $0<4r<R<\infty$ that
    	\begin{align}\label{l22}
    		\int_{B_R\setminus \bar{B}_r(0)}|\D u(x)|^2dx=&8\pi\sum_{n\in \Z^{\ast}}^{}\int_{r}^{R}|n|^2|a_n|^2\rho^{2n-1}d\rho    		=8\pi\sum_{n\in \Z^{\ast}}|n|^2\left(\frac{1}{2n}|a_n|^2\left(R^{2n}-r^{2n}\right)\right)\nonumber\\
    		&=4\pi \sum_{n\geq 1}^{}|n||a_n|^2R^{2|n|}\left(1-\left(\frac{r}{R}\right)^{2|n|}\right)+4\pi\sum_{n\leq -1}^{}|n||a_{n}|^2\frac{1}{r^{2|n|}}\left(1-\left(\frac{r}{R}\right)^{2|n|}\right)\nonumber\\
    		&\geq \frac{15\pi}{4} \sum_{n\geq 1}^{}|n|\left(|a_n|^2R^{2|n|}+|a_{-n}|^2\frac{1}{r^{2|n|}}\right)
    	\end{align}
    	\textbf{First $L^{2,1}$ estimate.}
    	Now, we have
    	\begin{align*}
    		\np{1}{2,1}{B_R\setminus \bar{B}_r}=4\sqrt{\pi}\left(R^2-r^2\right)^{\frac{1}{2}}\leq 4\sqrt{\pi}R
    	\end{align*}
    	while for all $m\geq 1$, 
    	\begin{align*}
    		\np{|z|^{m}}{2,1}{B_R\setminus \bar{B}_r(0)}&=4\sqrt{\pi}r^m\left(R^2-r^2\right)^{\frac{1}{2}}+4\sqrt{\pi}\int_{r^m}^{R^m}(R^2-t^{\frac{2}{m}})^{\frac{1}{2}}\,dt
    		\leq 4\sqrt{\pi}r^{m}R+4\sqrt{\pi}\int_{r^{m}}^{R^m}Rdt	=4\sqrt{\pi}R^{m+1}.
    	\end{align*}
    	Likewise, for all $m\geq 2$
    	\begin{align*}
    		\np{\frac{1}{|z|^{m}}}{2,1}{B_R\setminus \bar{B}_r(0)}&\leq 4\sqrt{\pi}\int_{0}^{\frac{1}{r^m}}\left(\frac{1}{t^{\frac{2}{m}}}-r^2\right)^{\frac{1}{2}}dt
    		\leq 4\sqrt{\pi}\int_{0}^{\frac{1}{r^m}}\frac{dt}{t^{\frac{1}{m}}}
    		=4\sqrt{\pi}\frac{m}{m-1}\frac{1}{r^{m-1}}
    		\leq 8\sqrt{\pi}r^{-m+1}.
    	\end{align*}
    	By \eqref{gradient}, we have
    	\begin{align*}
    	|\D u|&\leq 2\sum_{n\in \Z^{\ast}}^{}|n||a_n|\rho^{n-1}    	\end{align*}
    	and the following estimates by Cauchy-Schwarz inequality
    	\begin{align}\label{l3}
    		\np{\D u}{2,1}{B_{\alpha R}\setminus \bar{B}_{\alpha{-1}r}(0)}&\leq 16\sqrt{\pi}\left(\sum_{n\geq 1}^{}|n||a_n|\left(\alpha R\right)^{|n|}+\sum_{n\geq 1}^{}|n||a_{-n}|\left(\frac{\alpha}{r}\right)^{|n|}\right)\nonumber\\
    		%&\leq 16\sqrt{\pi}\left(\sum_{n\geq 1}^{}|n|\alpha^{2|n|}\right)^{\frac{1}{2}}\left(\sum_{n\geq 1}\left(|n||a_n|R^{|n|}+|n||a_{-n}|\frac{1}{r^{|n|}}\right)^2\right)^{\frac{1}{2}}\\
    		&\leq 16\sqrt{\pi}\left(\sum_{n\in \Z^{\ast}}^{}|n|\alpha^{2|n|}\right)^{\frac{1}{2}}\left(\sum_{n\geq 1}^{}|n||a_n|^2R^{2|n|}+|n||a_{-n}|^2\frac{1}{r^{2|n|}}\right)^{\frac{1}{2}}\nonumber\\
    		&=16\sqrt{2\pi}\frac{\alpha}{1-\alpha^2}\left(\sum_{n\geq 1}^{}|n||a_n|^2R^{2|n|}+|n||a_{-n}|^2\frac{1}{r^{2|n|}}\right)^{\frac{1}{2}}.
    	\end{align}
    	Combining \eqref{l22} and \eqref{l3} yields 
    	\begin{align*}
    		\np{\D u}{2,1}{B_{\alpha R}\setminus \bar{B}_{\alpha^{-1}r}(0)}&\leq \frac{16\sqrt{2\pi}\alpha}{1-\alpha}\times \sqrt{\frac{4}{15\pi}}\np{\D u}{2}{B_R\setminus \bar{B}_r(0)}
    		=32\sqrt{\frac{2}{15}}\frac{\alpha}{1-\alpha}\np{\D u}{2}{B_R\setminus \bar{B}_r(0)},
    	\end{align*}
    	which concludes the proof of the first part of the Lemma.
    	
    	\textbf{Second $W^{1,1}$ estimate.} 
    	As $\Delta u=0$, we have $|\D^2 u|=4|\p{z}^2u|$, and 
    	\begin{align*}
    	\p{z}^2u(z)=\sum_{n\in \Z^{\ast}}n(n-1)z^{n-2}.
    	\end{align*}
    	Now, for all $m\in \Z\setminus\ens{-2}$, we have
    	\begin{align*}
    	\np{|z|^m}{1}{B_{\alpha R}\setminus \bar{B}_{\alpha^{-1}r}(0)}=2\pi\int_{\alpha^{-1}r}^{\alpha R}\rho^{m+1}d\rho=\frac{2\pi}{m+2}\left((\alpha R)^{m+2}-(\alpha^{-1}r)^{m+2}\right)
    	\end{align*}    
    	In particular, we have by the triangle inequality and Cauchy-Schwarz inequality
    	\begin{align*}
    	&\np{\p{z}^2u}{1}{B_{\alpha R}\setminus \bar{B}_{\alpha^{-1}r}(0)}\leq 2\pi\sum_{n\in \Z^{\ast}}\frac{|n||n-1|}{n}|a_n|\left(\left(\alpha R\right)^n-\left(\alpha^{-1}r\right)^{n}\right)\\
    	&=2\pi\sum_{n\geq 1}|n-1||a_n|(\alpha R)^{|n|}\left(1-\left(\frac{\alpha^2r}{R}\right)^{|n|}\right)+2\pi\sum_{n\leq -1}|n-1||a_n|\left(\frac{\alpha}{r}\right)^{|n|}\left(1-\left(\frac{\alpha^2r}{R}\right)^{|n|}\right)\\
    	&\leq 2\pi \sum_{n\geq 1}{|n-1|}|a_n|(\alpha R)^{|n|}+\sum_{n\geq -1}{|n-1|}|a_n|\left(\frac{\alpha}{r}\right)^{|n|}\\
    	&\leq 2\pi\left(\sum_{n\in \Z^{\ast}}\frac{|n-1|^2}{|n|}\alpha^{2|n|}\right)^{\frac{1}{2}}\left(\sum_{n\geq 1}|n||a_n|^2R^{2|n|}+\sum_{n\leq -1}|n||a_n|^2\frac{1}{r^{2|n|}}\right)^{\frac{1}{2}}.
    	\end{align*}
    	Now, notice that 
    	\begin{align*}
    	\sum_{n\in \Z{\ast}}\frac{|n-1|^2}{|n|}\alpha^{2|n|}&%=4\alpha^2+\sum_{n\geq 2}\frac{(n-1)^2}{n}\alpha^{2n}+\sum_{n\geq 2}\frac{(n+1)^2}{n}\alpha^{2n}
    	%=4\alpha^2+2\sum_{n\geq 2}\frac{n^2+1}{n}\alpha^{2n}\\
    	=2\sum_{n\geq 1}\frac{n^2+1}{n}\alpha^{2n}=\frac{2\alpha^2}{(1-\alpha^2)^2}+2 \log\left(\frac{1}{1-\alpha^2}\right)\leq \frac{4\alpha^2}{(1-\alpha^2)^2}.
    	\end{align*}
    	Recalling from \eqref{l21l2} that 
    	\begin{align*}
    	\int_{B_{R}\setminus \bar{B}_{\alpha^{-1}r}(0)}|\D u(x)|^2dx\geq \frac{15\pi}{4} \sum_{n\geq 1}^{}|n|\left(|a_n|^2R^{2|n|}+|a_{-n}|^2\frac{1}{r^{2|n|}}\right),
    	\end{align*}
    	we deduce that 
    	\begin{align*}
    	\np{\p{z}^2u}{1}{B_{\alpha R}\setminus \bar{B}_{\alpha^{-1}r}(0)}\leq \frac{4\pi\alpha}{(1-\alpha^2)}\times \sqrt{\frac{4}{15\pi}}\np{\D u}{2}{B_R\setminus \bar{B}_r(0)}=8\sqrt{\frac{\pi}{15}}\frac{\alpha}{1-\alpha^2}\np{\D u}{2}{B_R\setminus \bar{B}_r(0)}
    	\end{align*}
    	which concludes the proof as $|\D^2u|=4|\p{z}^2u|$.
    \end{proof}
%newl1
\begin{rem}
	Notice that 
	$
		\np{\D\log|z|}{2}{B_R\setminus \bar{B}_r(0)}=\sqrt{2\pi}\sqrt{\log\left(\dfrac{R}{r}\right)}
	$
	while
	\begin{align*}
		&\np{\D\log|z|}{2,1}{B_{\alpha R}\setminus \bar{B}_{\alpha^{-1}r}(0)}=4\int_{0}^{\frac{1}{\alpha R}}\left(\leb^2(B_{\alpha R}\setminus \bar{B}_{\alpha^{-1}r}(0))\right)^{\frac{1}{2}}dt+4\int_{\frac{1}{\alpha R}}^{\frac{1}{\alpha^{-1}r}}\left(\leb^2\left(B_{\frac{1}{t}}\setminus \bar{B}_{\alpha^{-1}r}(0)\right)\right)^{\frac{1}{2}}dt\\
		&=\frac{4\sqrt{\pi}}{\alpha R}\left(\alpha^2R^2-\alpha^{-2}r^2\right)^{\frac{1}{2}}+4\sqrt{\pi}\int_{\frac{1}{\alpha R}}^{\frac{1}{\alpha^{-1}r}}\frac{1}{t}\sqrt{1-\frac{r^2t^2}{\alpha^2}}dt
		%&=4\sqrt{\pi}\sqrt{1-\left(\frac{r}{\alpha^2R}\right)^2}+4\sqrt{\pi}\left[\sqrt{1-\frac{r^2t^2}{\alpha^2}}+\log(t)-\log\left(1+\sqrt{1-\frac{r^2t^2}{\alpha^2}}\right)\right]_{\frac{1}{\alpha R}}^{\frac{1}{\alpha^{-1} r}}\\
		=4\sqrt{\pi}\left(\log\left(\frac{\alpha^2R}{r}\right)+\log\left(1+\sqrt{1-\left(\frac{r}{\alpha^2R}\right)^2}\right)\right).
	\end{align*}
	In particular, for  all fixed $0<\alpha<1$, if $\ens{R_k}_{k\in \N},\ens{r_k}_{k\in \N}\subset (0,\infty)$ are sequences chosen such that  $\dfrac{R_k}{r_k}\conv{k\rightarrow \infty} \infty$, we have
	\begin{align*}
		\lim\limits_{k\rightarrow \infty}\frac{\np{\D\log|z|}{2,1}{B_{\alpha R_k}\setminus \bar{B}_{\alpha^{-1}r_k}(0)}}{\np{\D\log|z|}{2}{B_{R_k}\setminus\bar{B}_{r_k}(0)}}=\infty.
	\end{align*}
	If the assumption $4r<R$ does not hold, observe that we get
	 the estimate
	\begin{align*}
		\np{\D u}{2,1}{B_{\alpha R}\setminus \bar{B}_{\alpha^{-1}r}(0)}\leq \frac{8\sqrt{2}}{\sqrt{1-\left(\frac{r}{R}\right)^2}}\frac{\alpha}{1-\alpha^2}\np{\D u}{2}{B_R\setminus \bar{B}_r(0)}.
	\end{align*}
\end{rem}

\begin{prop}\label{l212infty}
		Let $0<2^6r<R<\infty$ be fixed radii, and $u:\Omega=B_R\setminus \bar{B}_r(0)\rightarrow \R$ be a harmonic function such that for some $\rho_0\in (r,R)$
		\begin{align*}
		\int_{\partial B_{\rho_0}}\partial_{\nu}u \,d\mathscr{H}^1=0.
		\end{align*}
		Then for all $\left(\dfrac{r}{R}\right)^{\frac{1}{3}}<\alpha<\dfrac{1}{4}$,
		\begin{align*}
		\np{\D u}{2,1}{B_{\alpha R}\setminus \bar{B}_{\alpha^{-1}r}(0)}\leq 24\,\Gamma_1\sqrt{\alpha} \np{\D u}{2,\infty}{B_R\setminus \bar{B}_r(0)}.
		\end{align*}
\end{prop}
\begin{proof}
	Let $\beta=\sqrt{\alpha}$. Then by Lemma \ref{l21l2}, we have
	\begin{align*}
		\np{\D u}{2,1}{B_{\beta^2R}\setminus \bar{B}_{\beta^{-2}r}(0)}\leq \frac{12\beta}{1-\beta}\np{\D u}{2}{B_{\beta R}\setminus \bar{B}_{\beta^{-1}r}(0)}.
	\end{align*}
	Furthermore, by Lemma \ref{l21l2}, we have
	\begin{align*}
		\np{\D u}{2}{B_{\beta R}\setminus \bar{B}_{\beta^{-1}r}(0)}\leq \Gamma_1\np{\D u}{2,\infty}{B_R\setminus \bar{B}_{r}(0)}
	\end{align*}
	Therefore, as $\beta=\sqrt{\alpha}<1/2$, we find
	\begin{align*}
		\np{\D u}{2,1}{B_{\alpha R}\setminus \bar{B}_{\alpha^{-1}r}(0)}&\leq \frac{12\sqrt{\alpha}}{1-\sqrt{\alpha}}\Gamma_1\np{\D u}{2,\infty}{B_{R}\setminus \bar{B}_r(0)}
		\leq 24\,\Gamma_1\sqrt{\alpha}\np{\D u}{2,\infty}{B_R\setminus 
		\bar{B}_r(0)}
	\end{align*}
	which concludes the proof of the corollary.
\end{proof}

We will also need a quantitative estimate of the Lorentz-Sobolev embedding $W^{1,(2,1)}(\Omega)\rightarrow C^0(\Omega)$. %As the arguments do not use anything specific to the dimension $2$, we will state them for general $n\geq 2$.

\begin{lemme}
	Let $n\geq 2$, $\Omega\subset \R^n$ be a bounded connected open set and $u\in W^{1,(n,1)}(\Omega)$. Then $u\in C^0(\Omega)$ and for all $x,y\in \Omega$ such that $B(x,2|x-y|)\cup B(y,2|x-y|)\subset \Omega$, we have
	\begin{align*}
		|u(x)-u(y)|\leq \frac{2^{n+1}}{\alpha(n)^{\frac{1}{n}}}\np{\D u}{n,1}{\Omega\cap B(x,2|x-y|)}.
	\end{align*}
	Furthermore, if $\Omega$ is a bounded Lipschitz open subset of $\R^n$, then there exists a constant $C_4=C_4(\Omega)$ such that
	\begin{align*}
	\np{u-\bar{u}_{\Omega}}{\infty}{\Omega}\leq C_4\np{\D u}{n,1}{\Omega},
	\end{align*}
	where $\displaystyle \bar{u}_{\Omega}=\dashint{\Omega}u\,d\leb^n$ is the mean of $u$.
\end{lemme}
\begin{proof}
	Let $x\in \Omega$ and $d=\mathrm{dist}(x,\partial \Omega)>0$. For all $0<r<d$, let
	\begin{align*}
		u_{x,r}=\dashint{B(x,r)}u\,d\leb^n=\frac{1}{\alpha(n)r^n}\int_{B(x,r)}u\,d\leb^n.
	\end{align*}
	Then for all $0<r<d$, we have
	\begin{align*}
		u_{x,r}=\frac{1}{\alpha(n)}\int_{B(0,1)}u(x+r(y-x))dy
	\end{align*}
	so that
	\begin{align}\label{mean0}
		\left|\frac{d}{dr}u_{x,r}\right|=\left|\int_{B(0,1)}\D u(x+r(y-x))\cdot (y-x)dy\right|\leq \dashint{B(x,r)}|\D u|d\leb^n.
	\end{align}
	Therefore, we have by Fubini theorem and the duality $L^{n,1}/L^{\frac{n}{n-1},\infty}$ (see the estimate \eqref{lnweak}) for all $0<t\leq d$
	\begin{align*}
		\int_{0}^{t}\left|\frac{d}{dr}u_{x,r}\right|dr&\leq \frac{1}{\alpha(n)}\int_{0}^t\frac{1}{r^n}\int_{B(x,r)}|\D u(y)|d\leb^n(y)dr=\frac{1}{\alpha(n)}\int_{0}^t\int_{B(x,t)}\frac{1}{r^n}|\D u(y)|\mathrm{1}_{\ens{|x-y|<r}}d\leb^n(y)dr\\
		&=\frac{1}{\alpha(n)}\int_{B(x,t)}|\D u(y)\left(\int_{|x-y|}^d\frac{dr}{r^n}\right)d\leb^n(y)
		\leq \frac{1}{(n-1)\alpha(n)}\int_{B(x,t)}\frac{|\D u(y)|}{|x-y|^{n-1}}d\leb^n(y)\\
		&\leq \frac{1}{n^2\alpha(n)}\np{\D u}{n,1}{B(x,t)}\np{\frac{1}{|x-\,\cdot\,|^{n-1}}}{\frac{n}{n-1},\infty}{B(x,t)}
		=\frac{1}{n\alpha(n)^{\frac{1}{n}}}\np{\D u}{n,1}{B(x,t)}
	\end{align*}
	as for all $x\in \R^n$
	\begin{align}\label{weaknorm}
		\np{\frac{1}{|x-\,\cdot\,|^{n-1}}}{\frac{n}{n-1},\infty}{\R^n}=n\alpha(n)^{\frac{n}{n-1}}.
    \end{align}
    Therefore, by the Sobolev embedding $W^{1,1}(\R)\subset C^0(\R)$, the function $(0,d]\rightarrow \R, r\mapsto u_{x,r}$ is continuous, and for all $0<s<t\leq d$, we have
    \begin{align}\label{campanato}
    	\left|u_{x,s}-u_{x,t}\right|\leq \int_{s}^{t}\left|\frac{d}{dr}u_{x,r}\right|dr\leq \frac{1}{n\alpha(n)^{\frac{1}{n}}}\np{\D u}{n,1}{B_t(x)}.
    \end{align}
    Let $\ens{r_n}_{n\in \N}\subset (0,\infty)$ such that $r_n\conv{n\rightarrow \infty}0$. Then \eqref{campanato} implies that
    \begin{align*}
    	|u_{x,r_n}-u_{x,r_m}|\leq \frac{1}{n\alpha(n)^{\frac{1}{n}}}\np{\D u}{n,1}{B(x,\max\ens{r_n,r_m})}\conv{n,m\rightarrow \infty}0
    \end{align*}
    which implies that $\ens{u_{x,r_n}}_{n\in \N}$ is a Cauchy sequence.    Now, recall that by the Lebesgue differentiation theorem, for $\leb^n$ almost all $x\in \Omega$, we have
    \begin{align*}
    u(x)=\lim\limits_{r\rightarrow 0}u_{x,r}.
    \end{align*}
    Therefore, for $\leb^n$ almost all $x\in \Omega$ and for all $0<r<d(x)=\mathrm{dist}(x,\partial \Omega)$, we have
    \begin{align}\label{campanato2}
    	|u(x)-u_{x,r}|\leq \frac{1}{(n-1)\alpha(n)^{\frac{1}{n}}}\np{\D u}{n,1}{B(x,r)}.
    \end{align}
    To prove that $u$ is continuous, let $x,y\in \Omega$ such that \eqref{campanato2} holds for $x$ and $y$ (the proof is an adaptation of the H\"{o}lder continuous embedding of Campanato spaces of the right indices). Furthermore, without loss of generality, we can assume that $x\neq y$, and $2|x-y|<\max\ens{d(x),d(y)}$, so that
    \begin{align*}
    	B(x,2|x-y|)\cup B(y,2|x-y|)\subset \Omega.
    \end{align*}
    Therefore, if $r=|x-y|$ we have
    \begin{align}\label{fubini0}
    	\left|u(x)-u(y)\right|&\leq |u(x)-u_{x,r}|+|u_{x,r}-u_{y,r}|+|u(y)-u_{y,r}|\nonumber\\
    	&\leq \frac{1}{n\alpha(n)^{\frac{1}{n}}}\left(\np{\D u}{n,1}{B(x,|x-y|)}+\np{\D u}{n,1}{B(y,|x-y|)}\right)+|u_{x,r}-u_{y,r}|
    \end{align}
    so we need only estimate $|u_{x,r}-u_{y,r}|$, as
    \begin{align*}
    	\np{\D u}{n,1}{B(x,|x-y|)}+\np{\D u}{n,1}{B(y,|x-y|)}\conv{y\rightarrow x}0.
    \end{align*}
    We have
    \begin{align}\label{fubini1}
    	u_{x,r}-u_{y,r}&=\frac{1}{\alpha(n)r^n}\int_{B(x,r)}u(z_1)d\leb^n(z)_1-\frac{1}{\alpha(n)r^n}\int_{B(y,r)}u(z_2)d\leb^n(z_2)\nonumber\\
    	&=\frac{1}{(\alpha(n)r^n)^2}\int_{B(x,r)\times B(y,r)}(u(z_1)-u(z_2))d\leb^n(z_1)d\leb^n(z_2)\nonumber\\
    	&=\frac{1}{(\alpha(n)r^n)^{2}}\int_{B(x,r)\times B(y,r)}\left(\int_{0}^1\D u(z_2+t(z_1-z_2))\cdot (z_1-z_2)dt\right)d\leb^n(z_1)d\leb^n(z_2)    
    \end{align}
    Furthermore, for all $t\in [0,1]$ and $(z_1,z_2)\in B(x,r)\times B(y,r)$, we have $z_2+t(z_1-z_2)\in B(x,2r)$ and $|z_1-z_2|\leq 2r$. Therefore, Fubini's theorem implies that (by \eqref{lnweak})
    \begin{align}\label{fubini2}
    	&\left|\int_{B(x,r)}\left(\int_{0}^{1}\D u(z_2+t(z_1-z_2))\cdot (z_1-z_2)dt\right)d\leb^n(z_1)\right|\nonumber\\
    	&\leq \int_{0}^{1}\left(\int_{B(x,r)}\frac{|\D u(z_2+t(z_1-z_2))|}{|z_1-z_2|^{n-1}}|z_1-z_2|^nd\leb^n(z_1)\right)dt\nonumber\\
    	&\leq \frac{1}{n}2^nr^n\int_{0}^1\np{\D u(z_2+t(\,\cdot\,-z_2))}{n,1}{B(x,r)}\np{\frac{1}{|\,\cdot\,-z_2|}}{\frac{n}{n-1},\infty}{B(x,r)}dt\nonumber\\
    	&\leq 2^nr^n\alpha(n)^{\frac{n}{n-1}}\int_{0}^{1}\np{\D u}{n,1}{B(x,2r)}dt=2^nr^n\alpha(n)^{\frac{n}{n-1}}\np{\D u}{n,1}{B(x,2r)}.
    \end{align}
    Therefore, by \eqref{fubini1} and \eqref{fubini2}, we find
    \begin{align}\label{fubini3}
    	|u_{x,r}-u_{y,r}|\leq \frac{1}{\alpha(n)r^n}\int_{B(y,r)} \frac{2^n}{\alpha(n)^{\frac{1}{n}}}\np{\D u}{n,1}{x,2|x-y|}d\leb^n(z_2)=\frac{2^n}{\alpha(n)^{\frac{1}{n}}}\np{\D u}{n,1}{B(x,2|x-y|)}.
    \end{align}
    Furthermore, as the argument is symmetric in $x$ and $y$ notice that
    \begin{align*}
    	|u_{x,r}-u_{y,r}|&\leq \frac{2^n}{\alpha(n)^{\frac{1}{n}}}\min\ens{\np{\D u}{n,1}{B(x,2|x-y|)},\np{\D u}{n,1}{B(y,2|x-y|)}}.
    \end{align*}
    Finally, thanks to \eqref{fubini0} and \eqref{fubini3} we get
    \begin{align}\label{ascoli}
    	|u(x)-u(y)|\leq \frac{2^{n+1}}{\alpha(n)^{\frac{1}{n}}}\np{\D u}{n,1}{B(x,2|x-y|)}
    \end{align}
    which implies that $u$ is continuous, with modulus of continuity at $x$
    \begin{align*}
    	r\mapsto \frac{2^{n+1}}{\alpha(n)^{\frac{1}{n}}}\np{\D u}{n,1}{\Omega\,\cap\, B(x,2r)}.
    \end{align*}  
    Now, for the $L^{\infty}$ bound, first consider the case $\Omega=\R^n$, and let $G:\R^n\times \R^n\rightarrow\R\cup\ens{\infty}$ be the Green's function of the Laplacian on $\R^n$. Then 
    \begin{align*}
    	\D_y G(x,y)=\frac{1}{n\alpha(n)}\frac{1}{|x-y|^{n-1}}\in L^{\frac{n}{n-1},\infty}(\R^n)
    \end{align*}
    and we have for all $x\in \R^n$
    \begin{align*}
    	u(x)=\int_{\R^n}\Delta_y G(x,y)u(y)dy=-\int_{\R^n}\D_y G(x,y)\cdot \D u(y)dy
    \end{align*}
    and \eqref{weaknorm} implies that
	\begin{align}\label{1}
		\np{u}{\infty}{\R^n}&\leq \frac{n-1}{n^2}\np{\D u}{n,1}{\R^n}\np{\D_y G(x,y)}{\frac{n}{n-1},\infty}{\R^n}=\frac{(n-1)}{n^3\alpha(n)}\np{\D u}{n,1}{\R^n}\np{\frac{1}{|x-\,\cdot\,|^{n-1}}}{\frac{n}{n-1},\infty}{\R^n}\nonumber\\
		&=\frac{(n-1)}{n^2\alpha(n)^{\frac{1}{n}}}\np{\D u}{n,1}{\R^n}\leq \frac{1}{n\alpha(n)^{\frac{1}{n}}}\np{\D u}{n,1}{\R^N}
	\end{align}
	Now, (thanks to \cite{brezis} IX.$7$) there exists a linear extension operator
	\[
	P:\bigcup_{1\leq p<\infty}W^{1,p}(\Omega)\rightarrow \bigcup_{1\leq p<\infty}W^{1,p}(\R^n)
	\]
	such that for $1\leq p<\infty$ the restriction $P|W^{1,p}(\Omega)\rightarrow W^{1,p}(\R^n)$ be a continuous linear operator. Then by identifying $W^{1,p}(\Omega)$ with a closed subset of $L^p(\R^n)^{n+1}$, the Stein-Weiss interpolation theorem implies that for all $P$ extends as a continuous linear operator $W^{1,(n,1)}(\Omega)$ into $W^{1,(n,1)}(\R^n)$, as the Sobolev embedding $L^n(\Omega)\hookrightarrow L^q(\Omega)$ for all $1\leq q<\infty$ shows that $\D u\in L^{n,1}(\Omega)$ implies that $u\in L^{n,1}(\Omega)$. Therefore, by \eqref{1}, for all $u\in W^{1,(n,1)}(\Omega)$, we have
	\begin{align}\label{ascoli2}
		\np{u}{\infty}{\Omega}\leq \np{\D Pu}{\infty}{\R^n}&\leq \frac{1}{n\alpha(n)^{\frac{1}{n}}}\np{P u}{n,1}{\R^n}\leq C\left(\np{u}{n,1}{\Omega}+\np{\D u}{n,1}{\Omega}\right)\nonumber\\
		&\leq C'(\np{u}{n}{\Omega}+\np{\D u}{n,1}{\Omega}),
	\end{align}
	where  we have used in the last line the embedding $W^{1,n}(\Omega)\hookrightarrow L^{n,1}(\Omega)$.
	
	Now, \eqref{ascoli2} implies by the classical Poincaré-Wirtinger inequality and the continuous embedding $L^{n,1}(\Omega)\hookrightarrow L^n(\Omega)$
	\begin{align*}
		\np{u-\bar{u}_{\Omega}}{\infty}{\Omega}&\leq  C'(\np{u-\bar{u}_{\Omega}}{n}{\Omega}+\np{\D u}{n,1}{\Omega})
		\leq C'\left(C''\np{\D u}{n}{\Omega}+\np{\D u}{n,1}{\Omega}\right)\\
		&\leq C_4(\Omega)\np{\D u}{n,1}{\Omega}
	\end{align*} 
	and this concludes the proof of the Lemma.
\end{proof}

Now, we will need to refine the $L^{\infty}$ bound to obtain an estimate independent of the conformal class (bounded away from $-\infty$) of flat annuli in $\R^n$.

\begin{prop}\label{conf}
Let $0<2r<R<\infty$ and $\Omega=B_R\setminus \bar{B}_r(0)\subset \R^n$. Then there exists a universal constant $\Gamma_4=\Gamma_4(n)$ such that for all $u\in W^{1,(n,1)}(\Omega)$, we have
\begin{align*}
	\np{u-\bar{u}_{\Omega}}{\infty}{\Omega}\leq \Gamma_4(n)\np{\D u}{n,1}{\Omega}.
\end{align*}
\end{prop}
\begin{proof}
First, observe that the $L^{\infty}$ norm and the $(n,1)$ norm of the gradient $\np{\D\,\cdot\,}{n,1}{\Omega}$ are scaling invariant (see \eqref{case2} for the case $n=2$). Therefore, the constant $C_4(\Omega)$ in Theorem \ref{neckfine} is scaling invariant. In particular, there exists a universal constant $\Gamma_5(n)=C_4(B_{2}\setminus B_1(0))$ for all $0<r<\infty$ and $u\in W^{1,(n,1)}(B_{2r}\setminus\bar{B}_r(0))$, we have
\begin{align}\label{conformal}
	\np{u-\bar{u}_{B_{2r}\setminus \bar{B}_r(0)}}{\infty}{\bar{B}_{2r}\setminus \bar{B}_r(0)}\leq \Gamma_5(n)\np{\D u}{n,1}{B_{2r}\setminus \bar{B}_r(0)}.
\end{align}
Now, as $2r<R$ let $J\in \N$ such that
\begin{align*}
	2^{J}r<R\leq 2^{J+1}r.
\end{align*}
Then we have
\begin{align*}
	\Omega= B_{R}\setminus B_{\frac{R}{2}}(0)\cup \bigcup_{j=0}^{J-1}B_{2^{j+1}r}\setminus \bar{B}_{2^{j}r}(0).
\end{align*}
For the convenience of notation, let us write $\Omega_j=B_{2^{j+1}r}\setminus \bar{B}_{2^{j}r}$ for all $0\leq j\leq J-1$. Thanks to \eqref{conformal} for all $0\leq j\leq J$, we have
\begin{align}\label{pw}
	&\np{u-\bar{u}_j}{\infty}{\Omega_j}\leq \Gamma_5(n)\np{\D u}{n,1}{\Omega_j}\quad \text{where}\;\, \bar{u}_j=\dashint{\Omega_j}u\,d\leb^n\nonumber\\
	&\np{u-\bar{u}_{B_R\setminus B_{R/2}(0)}}{\infty}{B_R\setminus B_{R/2}(0)}\leq \Gamma_5(n)\np{\D u}{n,1}{B_R\setminus B_{R/2}(0)}.
\end{align}
Now define for all $r<t<R$
\begin{align*}
	u_t=\dashint{\partial B_t(0)}u\,d\mathscr{H}^{n-1}.
\end{align*}
For all $r<t<R$, thanks to a similar argument as given in \eqref{mean0}, we have
\begin{align*}
	\left|\frac{d}{dt}u_t\right|\leq \dashint{\partial B_t}|\D u|d\mathscr{H}^{n-1}.
\end{align*}
Furthermore, if $r\leq r_1<R$ is a fixed radius, thanks to the co-area formula, we have for $\leb^1$ almost all $t\in (r_1,R)$
\begin{align*}
	\int_{\partial B_t}|\D u|\,d\mathscr{H}^{n-1}=\frac{d}{dt}\int_{r_1}^{t}\left(\int_{\partial B_s}|\D u|d\mathscr{H}^{n-1}\right)d\leb^1(s)=\frac{d}{dt}\int_{B_t\setminus \bar{B}_{r_1}(0)}|\D u|d\leb^n.
\end{align*}
Therefore, we have 
\begin{align}\label{nmean0}
	\int_{r_1}^{r_2}\left|\frac{d}{dt}u_t\right|dt&\leq \frac{1}{n\alpha(n)}\int_{r_1}^{r_2}\frac{1}{t^{n-1}}\left(\int_{\partial B_t}|\D u|d\mathscr{H}^{n-1}\right)dt\nonumber\\
	&=\frac{1}{n\alpha(n)}\left[\frac{1}{t^{n-1}}\int_{B_t\setminus B_{r_1}(0)}|\D u|d\leb^n\right]_{r_1}^{r_2}+\frac{n-1}{n\alpha(n)}\int_{r_1}^{r_2}\frac{1}{t^n}\left(\int_{B_t\setminus B_{r_1}(0)}|\D u|d\leb^n\right)dt\nonumber\\
	&=\frac{1}{n\alpha(n)}\frac{1}{r_2^{n-1}}\int_{B_{r_2}\setminus \bar{B}_{r_1}(0)}|\D u|d\leb^n+\frac{n-1}{n\alpha(n)}\int_{r_1}^{r_2}\int_{B_{r_2}\setminus \bar{B}_{r_1}(0)}\frac{|\D u(x)|}{t^{n}}\mathrm{1}_{\ens{r_1\leq |x|\leq t}}d\leb^n(x) dt.
\end{align}
Furthermore, observe that
\begin{align}\label{nmean1}
	\int_{B_{r_2}\setminus B_{r_1}(0)}\frac{|\D u(x)|}{|x|^{n-1}}d\leb^n(x)&\leq \frac{1}{n}  \np{\D u}{n,1}{B_{r_2}\setminus \bar{B}_{r_1}(0)}\np{\frac{1}{|x|^{n-1}}}{\frac{n}{n-1},\infty}{B_{r_2}\setminus \bar{B}_{r_1}(0)}\nonumber\\
	&\leq \alpha(n)^{\frac{n}{n-1}}\np{\D u}{n,1}{B_{r_2}\setminus \bar{B}_{r_1}(0)}
\end{align}
while by Fubini's theorem
\begin{align}\label{nmean2}
	\int_{r_1}^{r_2}\int_{B_{r_2}\setminus \bar{B}_{r_1}(0)}\frac{|\D u(x)|}{t^{n}}\mathrm{1}_{\ens{r_1\leq |x|\leq t}}d\leb^n(x) dt&=\int_{B_{r_2}\setminus \bar{B}_{r_1}(0)}|\D u(x)|\left(\int_{|x|}^{r_2}\frac{dt}{t^{n-1}}dt\right)d\leb^n(x)\nonumber\\
	&=\frac{1}{n-1}\int_{B_{r_2}\setminus \bar{B}_{r_1}(0)}|\D u(x)|\left(\frac{1}{|x|^{n-1}}-\frac{1}{r_2^{n-1}}\right).
\end{align}
Finally, we get by \eqref{nmean0}, \eqref{nmean1}, \eqref{nmean2}, \eqref{nmean3} and \eqref{weaknorm}
\begin{align}\label{nmean3}
	\int_{r_1}^{r_2}\left|\frac{d}{dt}u_t\right|dt&\leq \frac{1}{n\alpha(n)}\frac{1}{r_2^{n-1}}\int_{B_{r_2}\setminus \bar{B}_{r_1}(0)}|\D u|d\leb^n+\frac{n-1}{n\alpha(n)}\int_{r_1}^{r_2}\int_{B_{r_2}\setminus \bar{B}_{r_1}(0)}\frac{|\D u(x)|}{t^{n}}\mathrm{1}_{\ens{r_1\leq |x|\leq t}}d\leb^n(x) dt\nonumber\\
	&=\frac{1}{n\alpha(n)}\frac{1}{r_2^{n-1}}\int_{B_{r_2}\setminus \bar{B}_{r_1}(0)}|\D u|d\leb^n+\frac{1}{n\alpha(n)}\int_{B_{r_2}\setminus \bar{B}_{r_1}(0)}|\D u(x)|\left(\frac{1}{|x|^{n-1}}-\frac{1}{r_2^{n-1}}\right)\nonumber\\
	&=\frac{1}{n\alpha(n)}\int_{B_{r_2}\setminus B_{r_1}(0)}\frac{|\D u(x)|}{|x|^{n-1}}d\leb^n(x)
	\leq \frac{1}{n\alpha(n)^{\frac{1}{n}}}\np{\D u}{n,1}{B_{r_2}\setminus \bar{B}_{r_1}(0)}
\end{align}
Therefore, we have for all $r\leq r_1<r_2\leq R$
\begin{align}\label{smean}
	\left|u_{r_1}-u_{r_2}\right|\leq \frac{1}{n\alpha(n)^{\frac{1}{n}}}\np{\D u}{n,1}{B_{r_2}\setminus \bar{B}_{r_1}(0)}.
\end{align}
Furthermore, recalling that $\beta(n)=\mathscr{H}^{n-1}(S^{n-1})=n\alpha(n)$ we obtain for all $r\leq s<t\leq R$,  thanks to \eqref{smean} that
\begin{align*}
	\dashint{B_t\setminus \bar{B}_s(0)}u\,d\leb^n&=\frac{n}{\beta(n)(t^n-s^n)}\int_{s}^t\left(\int_{\partial B_{\rho}}u\,d\mathscr{H}^{n-1}\right)d\rho\\
	&\leq \int_{s}^t\left(\frac{\rho^{n-1}}{t^{n-1}}\int_{\partial B_t}u\,d\mathscr{H}^{n-1}+\beta(n)\frac{\rho^{n-1}}{n\alpha(n)^{\frac{1}{n}}}\np{\D u}{n,1}{B_t\setminus \bar{B}_s(0)}\right)\\
	&=\frac{n}{\beta(n)(t^n-s^n)}\left(\frac{t^n-s^n}{n}\beta(n)\dashint{\partial B_t}u\,d\mathscr{H}^{n-1}+\frac{\beta(n)}{n}\frac{t^n-s^n}{n\alpha(n)^{\frac{1}{n}}}\np{\D u}{n,1}{B_{r_2}\setminus \bar{B}_{r_1}(0)}\right)\\
	&=\dashint{\partial B_t}u\,d\mathscr{H}^{n-1}+\frac{1}{n\alpha(n)^{\frac{1}{n}}}\np{\D u}{n,1}{B_{t}\setminus \bar{B}_s(0)}
\end{align*}
and the reverse inequality (given by \eqref{smean})
\begin{align*}
	\int_{\partial B_{\rho}}u\,d\mathscr{H}^{n-1} \geq \frac{\rho^{n-1}}{t^{n-1}}\int_{\partial B_t}u\,d\mathscr{H}^{n-1}-\frac{\beta(n)}{n\alpha(n)^{\frac{1}{n}}}\rho^{n-1}\np{\D u}{n,1}{B_{t}\setminus \bar{B}_{s}(0)}\quad \text{for all}\;\, s<\rho<t
\end{align*}
shows that for all $r\leq s<t\leq R$
\begin{align*}
	\left|\dashint{B_t\setminus \bar{B}_s(0)}u\,d\leb^n-\dashint{\partial B_t}u\,d\mathscr{H}^{n-1}\right|\leq \frac{1}{n\alpha(n)^{\frac{1}{n}}}\np{\D u}{n,1}{B_t\setminus \bar{B}_s(0)}.
\end{align*}
Therefore, by the triangle inequality we finally obtain that for all $0\leq j\leq J-1$,
\begin{align}\label{pw1}
	\left|u_j-\bar{u}_{\Omega}\right|&=\left|\dashint{B_{2^{j+1}r}\setminus B_{2^jr}(0)}u\,d\leb^n-\dashint{B_R\setminus B_r(0)}u\,d\leb^n\right|\leq \left|\dashint{B_{2^{j+1}r}\setminus B_{2^jr}(0)}u\,d\leb^n-\dashint{\partial B_{2^{j+1}r}}u\,d\mathscr{H}^{n-1}\right|\nonumber\\
	&+\left|\dashint{\partial B_{2^{j+1}r}}u\,d\mathscr{H}^{n-1}-\int_{\partial B_R}u\,d\mathscr{H}^{n-1}\right|+\left|\dashint{B_R\setminus B_r(0)}u\,d\leb^n-\int_{\partial B_R}u\,d\mathscr{H}^{n-1}\right|\nonumber\\
	&\leq \frac{3}{n\alpha(n)^{\frac{1}{n}}}\np{\D u}{n,1}{\Omega},
\end{align}
and likewise,
\begin{align}\label{pw2}
	\left|\bar{u}_{B_R\setminus B_{R/2}(0)}-\bar{u}_{\Omega}\right|\leq \frac{3}{n\alpha(n)^{\frac{1}{n}}}\np{\D u}{n,1}{\Omega}.
\end{align}
Finally, thanks to \eqref{pw}, \eqref{pw1} and \eqref{pw2}, we have
\begin{align*}
	\np{u-\bar{u}_{\Omega}}{\infty}{\Omega}\leq \left(\Gamma_5(n)+\frac{3}{n\alpha(n)^{\frac{1}{n}}}\right)\np{\D u}{n,1}{\Omega}
\end{align*}
and this concludes the proof of the Proposition.
\end{proof}

\textbf{We now come back to the proof of Theorem \ref{neckfine}.}

\begin{proof}(of Theorem \ref{neckfine})
	Thanks to Lemma IV.$1$ \cite{quanta}, there exists a universal constant $C_0=C_0(n)>0$ and an extension $\widetilde{\n}:B(0,R)\rightarrow \mathscr{G}_{n-2}(\R^n)$ of $\n$ such that
	\begin{align*}
	&\widetilde{\n}=\n\quad \text{on}\;\, \Omega=B(0,R)\setminus \bar{B}(0,r)\\
	&\big\Vert\D \widetilde{\n}\big\Vert_{\mathrm{L}^{2,\infty}(B(0,R))}\leq C_0 \np{\D \n}{2,\infty}{\Omega}.
	\end{align*}
	Therefore, by Lemma IV.$3$ of \cite{quanta}, there exists a universal constant $C_1=C_1(n)$ and a moving Coulomb frame $(\e_1,\e_2)\in W^{1,2}(B(0,R),S^{n-1})\times W^{1,2}(B(0,R),S^{n-1})$ such that
	\begin{align}\label{movingframe}
	\left\{
	\begin{alignedat}{1}
	&\widetilde{\n}=\star \left(\e_1\wedge\e_2\right)\quad \dive\left(\e_1\cdot \D\e_2\right)=0\\
	&\np{\D\e_1}{2}{B(0,R)}^2+\np{\D\e_2}{2}{B(0,R)}^2\leq C_1\np{\D \n}{2}{\Omega}^2.
	\end{alignedat}\right.
	\end{align}
	Furthermore, notice that for all $u\in  W^{1,{(2,1)}}_{\mathrm{loc}}(\R^2)$, and for all $\rho>0$, we have
	\begin{align}\label{case2}
	&\np{\D u}{2,1}{B(0,\rho)}=4\int_{0}^{\infty}\left(\leb^2(B(0,r)\cap\ens{x:|\D u(x)|>t})\right)^{\frac{1}{2}}d\leb^{1}(t)\nonumber\\
	&=4\int_{0}^{\infty}\left(\int_{B(0,\rho)}\mathrm{1}_{\ens{x:|\D u(x)|>t}}d\leb^2(x)\right)^{\frac{1}{2}}d\leb^{1}(t)
	=4\int_{0}^{\infty}\left(\int_{B(0,1)}\mathrm{1}_{\ens{y:|\D (u\circ \varphi_{\rho})(y)|>\rho t}}\rho^2 d\leb^2(y)\right)^{\frac{1}{2}}d\leb^1(t)\nonumber\\
	&=4\int_{0}^{\infty}\left(\int_{B(0,1)}\mathrm{1}_{\ens{y:|\D (u\circ \varphi_{\rho})(y)|>s}}\rho^2 d\leb^2(y)\right)^{\frac{1}{2}}\rho^{-1}\, d\leb^1(s)=\np{\D(u\circ \varphi_{\rho})}{2,1}{B(0,1)}.
	\end{align}
	where $\varphi_{\rho}(y)=\rho y$. Now, if $\mu:B(0,R)\rightarrow \R$ is the unique solution of the system
	\begin{align}\label{systmu}
	\left\{\begin{alignedat}{2}
	\Delta \mu&=\D^{\perp}\e_1\cdot \D\e_2\quad&& \text{in}\;\, B_R(0)\\
	\mu&=0\quad &&\text{on}\;\,\partial B_R(0)
	\end{alignedat}\right.
	\end{align}
	then $\widetilde{\mu}=\mu\circ \varphi_R$ solves (with evident notations)
	\begin{align*}
	\left\{\begin{alignedat}{2}
	\Delta \widetilde{\mu}&=\D^{\perp}\widetilde{\e_1}\cdot \D\widetilde{\e_2}\quad&& \text{in}\;\, B_1(0)\\
	\widetilde{\mu}&=0\quad &&\text{on}\;\,	S^1
	\end{alignedat}\right.
	\end{align*}
	Therefore, the improved Wente inequality (\cite{helein}, $3.4.1$) shows that there exists a universal constant $\Gamma_2>0$ such that
	\begin{align}\label{main1}
	\np{\D \mu}{2,1}{B(0,R)}=\np{\D \widetilde{\mu}}{2,1}{B_1(0)}&\leq \Gamma_2\np{\D\widetilde{\e_1}}{2}{B_1(0)}\np{\D\widetilde{\e_2}}{2}{B_1(0)}
	=\Gamma_2\np{\D\e_1}{2}{B_R(0)}\np{\D\e_2}{2}{B_R(0)}\nonumber\\
	&\leq \frac{1}{2}{\Gamma_2}C_1\int_{\Omega}|\D\n|^2dx.
	\end{align}
	Furthermore, notice that we also have the optimal inequality
	\begin{align}\label{main2}
	\np{\D \mu}{2}{B_R(0)}&\leq \frac{1}{4}\sqrt{\frac{3}{\pi}}\np{\D\e_1}{2}{B_R(0)}\np{\D\e_2}{2}{B_R(0)}
	\leq \frac{1}{8}\sqrt{\frac{3}{\pi}}C_1\int_{\Omega}|\D\n|^2dx.
	\end{align}
	Now, let $\upsilon=\lambda-\mu$ on $\Omega=B(0,R)\setminus \bar{B}(0,r)$. Then $\upsilon$ is harmonic on $\Omega$ and $\upsilon=\lambda$ on $\partial B_R(0)$. 
	Then as $\upsilon$ is harmonic, there exists $d\in \R$ and $\ens{a_k}_{k\in \Z}\subset \C$ such that
	\begin{align*}
	\upsilon(\rho,\theta)=a_0+d\,\log\rho+\sum_{k\in \Z^{\ast}}\left(a_k\rho^k+\bar{a_{-k}}\rho^{-k}\right)e^{ik\theta}.
	\end{align*}
	Now, noticing that for all $r<\rho<R$
	\begin{align}\label{log}
	d=\frac{1}{2\pi}\int_{\partial B_{\rho}}\partial_{\nu}\upsilon,
	\end{align}
	this implies that $\upsilon-d\log|z|$ satisfies the hypothesis of Proposition \ref{l212infty}. Therefore, using the identity $\upsilon=\lambda-\mu$, the inequalities \eqref{main2} and  $\np{\,\cdot\,}{2,\infty}{\,\cdot\,}\leq 2\np{\,\cdot\,}{2}{\,\cdot\,}$, we have for all  $\left(\dfrac{r}{R}\right)^{\frac{1}{3}}<\alpha<\dfrac{1}{4}$
	\begin{align}\label{main3}
	&\np{\D(\upsilon-d\log|z|)}{2,1}{B_{\alpha R}\setminus \bar{B}_{\alpha^{-1}r}}\leq 24\,\Gamma_1\sqrt{\alpha}\np{\D(\upsilon-d\log|z|)}{2,\infty}{B_R\setminus \bar{B_r}(0)}\nonumber\\
	&\leq 24\,\Gamma_1\sqrt{\alpha}\left(\np{\D(\lambda-d\log|z|)}{2,\infty}{B_R\setminus \bar{B}_r(0)}+\np{\D \mu}{2,\infty}{B_R\setminus \bar{B}_r(0)}\right)\nonumber\\
	&\leq 24\,\Gamma_1\sqrt{\alpha}\left(\np{\D(\lambda-d\log|z|)}{2,\infty}{\Omega}+2\np{\D\mu}{2}{\Omega}\right)\nonumber\\
	&\leq 24\,\Gamma_1\sqrt{\alpha}\left(\np{\D(\lambda-d\log|z|)}{2,\infty}{\Omega}+\frac{1}{4}\sqrt{\frac{3}{\pi}}C_1\int_{\Omega}|\D\n|^2dx\right).
	\end{align}
	Furthermore, notice that by the co-area formula, for all $s\in (r,R)$ such that $2s<R$, we have
	\begin{align*}
	\int_{B_{2s}\setminus \bar{B}_{s}(0)}|\D \upsilon(x)|dx
	&=\int_{s}^{2s}\left(\rho\int_{\partial B_{\rho}}|\D\upsilon|d\mathscr{H}^1\right)\frac{d\rho}{\rho}
	\geq \log(2)\inf_{s< \rho< 2s}\left(\rho\int_{\partial B_{\rho}}|\D\upsilon|d\mathscr{H}^1\right).
	\end{align*}
	Therefore, there exists $\rho\in (s,2s)$ such that 
	\begin{align*}
	\int_{\partial B_{\rho}}|\D\upsilon|d\mathscr{H}^1&\leq \frac{1}{\log(2)\rho}\int_{B_{2s}\setminus \bar{B}_s(0)}|\D\upsilon(x)|dx
	\leq \frac{1}{\log(2)\rho}\np{1}{2,1}{B_{2s}\setminus \bar{B}_s}\np{\D\upsilon}{2,\infty}{B_{2s}\setminus \bar{B}_s(0)}\\
	&=\frac{1}{\log(2)\rho}4\sqrt{3\pi}s\np{\D\upsilon}{2,\infty}{B_{2s}\setminus \bar{B}_s(0)}
	\leq \frac{4\sqrt{3\pi}}{\log(2)}\left(\np{\D\lambda}{2,\infty}{\Omega}+2\np{\D \mu}{2}{\Omega}\right)\\
	&\leq \frac{4\sqrt{3\pi}}{\log(2)}\left(\np{\D\lambda}{2,\infty}{\Omega}+\frac{1}{4}\sqrt{\frac{3}{\pi}}C_1\int_{\Omega}|\D\n|^2dx\right).
	\end{align*}
	This implies by \eqref{log} that
	\begin{align}\label{main4}
	|d|\leq \frac{2}{\log(2)}\sqrt{\frac{3}{\pi}}\left(\np{\D\lambda}{2,\infty}{\Omega}+\frac{1}{4}\sqrt{\frac{3}{\pi}}C_1\int_{\Omega}|\D\n|^2dx\right).
	\end{align}
	As $\np{\D\log|z|}{2,\infty}{\Omega}= 2\sqrt{\pi}$, by \eqref{main3} and \eqref{main4} there exists a universal constant $\Gamma_3=\Gamma_3(n)$ such that
	\begin{align}\label{main5}
	\np{\D(\upsilon-d\log|z|)}{2,1}{B_{\alpha R}\setminus \bar{B}_{\alpha^{-1}r}}\leq \Gamma_3(n)\sqrt{\alpha}\left(\np{\D\lambda}{2,\infty}{\Omega}+\int_{\Omega}|\D\n|^2dx\right).
	\end{align}
	Finally, putting together \eqref{main1}, \eqref{main5} and recalling that $\lambda=\mu+\upsilon$, we have for all $\left(\dfrac{r}{R}\right)^{\frac{1}{4}}\leq \alpha<\dfrac{1}{4}$
	\begin{align}\label{main6}
	&\np{\D(\lambda-d\log|z|)}{2,1}{B_{\alpha R}\setminus \bar{B}_{\alpha^{-1}r}(0)}\leq \np{\D(\upsilon-d\log|z|)}{2,1}{B_{\alpha R}\setminus \bar{B}_{\alpha^{-1}r}}+\np{\D\mu}{2,1}{B_{\alpha R}\setminus B_{\alpha^{-1}}(0)}\nonumber\\
	&\leq \Gamma_3(n)\sqrt{\alpha}\left(\np{\D\lambda}{2,\infty}{\Omega}+\int_{\Omega}|\D\n|^2dx\right)+\frac{1}{2}\Gamma_2C_1(n)\int_{\Omega}|\D\n|^2dx
	\end{align} 
	Now, we  estimate for $r\leq \rho< R$ the following quantity
	\begin{align*}
	\left|d-\frac{1}{2\pi}\int_{\partial B_{\rho}}\partial_{\nu}\lambda\,d\mathscr{H}^1\right|=\left|\frac{1}{2\pi}\int_{\partial B_{\rho}}\partial_{\nu}\mu\,d\mathscr{H}^1\right|.
	\end{align*}
	We have, recalling that $\mu$ is well defined on $B_R(0)$ and satisfies \eqref{systmu}, we find
	\begin{align}\label{degree}
	0&=\int_{B_R\setminus B_{\rho}}\mu(x)\,\Delta\log\left(\frac{|x|}{R}\right)dx
	=-\log\left(\frac{R}{\rho}\right)\int_{\partial B_{\rho}}\partial_{\nu}\mu\,d\mathscr{H}^1+\int_{B_R\setminus B_{\rho}(0)}\Delta\mu\,\log\left(\frac{|x|}{R}\right)dx\nonumber\\
	&=-\log\left(\frac{R}{\rho}\right)\int_{\partial B_{\rho}}\partial_{\nu}\mu\,d\mathscr{H}^1+\int_{B_R(0)}\Delta\mu\log\left(\frac{|x|}{R}\right)dx-\int_{B_\rho(0)}\left(\D^{\perp}\e_1\cdot \D\e_2\right)\log\left(\frac{|x|}{R}\right)dx.
	\end{align}
	First, the previous estimate \eqref{main1} yields
	\begin{align}\label{main7}
	\left|\int_{B_R(0)}\Delta\mu\log\left(\frac{|x|}{R}\right)dx\right|&=\left|\int_{B_R(0)}\D\mu\cdot \D\log|x| dx\right|
	\leq \frac{1}{2}\np{\D\mu}{2,1}{B_R(0)} \np{\frac{1}{|x|}}{2,\infty}{B_R(0)}\nonumber\\
	&\leq \frac{\sqrt{\pi}}{2}\Gamma_2C_1(n)\int_{\Omega}|\D\n|^2dx.
	\end{align}
	
	Now, using once more Lemma IV.$3$ of \cite{quanta}, we see that exists a Coulomb moving frame $(\vec{f_1},\vec{f}_2)\in W^{1,2}(B_{\rho}(0),S^{n-1})\times W^{1,2}(B_{\rho}(0),S^{n-1})$ such that
	\begin{align*}
	\widetilde{\n}=\star (\vec{f}_1\wedge \vec{f}_2)
	\end{align*}
	and
	\begin{align}\label{key}
	&\np{\D \vec{f}_1}{2}{B(0,\rho)}^2+\np{\D \vec{f}_2}{2}{B(0,\rho)}^2\leq C_2(n)\np{\D \widetilde{\n}}{B_{\rho}}{2}^2
	=C_2(n)\int_{B_r(0)}|\D\widetilde{\n}|^2dx+C_2(n)\int_{B_{\rho}\setminus \bar{B}_r(0)}|\D\n|^2dx\nonumber\\
	&\leq C_1(n)\int_{B_{2r}\setminus \bar{B}_r(0)}|\D\n|^2dx+C_2(n)\int_{B_{\rho}\setminus \bar{B}_r(0)}|\D \n|^2dx
	\leq C_3(n)\int_{B_{\max\ens{\rho,2r}}\setminus \bar{B}_r(0)}|\D\n|^2dx.
	\end{align}
	Now, let $\psi$ be the solution of
	\begin{align*}
	\left\{
	\begin{alignedat}{2}
	\Delta\psi&=\D^{\perp}\vec{f}_1\cdot \D\vec{f}_2\quad&& \text{in}\;\, B_{\rho}(0)\\
	\psi&=0\quad &&\text{on}\;\, \partial B_{\rho}(0).
	\end{alignedat}\right.
	\end{align*}
	As in \eqref{key}, we get
	\begin{align}\label{step00}
	\np{\D\psi}{2,1}{B_\rho(0)}\leq \frac{1}{2}\Gamma_2 C_3(n)\int_{B_{\max\ens{\rho,2r}}\setminus \bar{B}_r(0)}|\D\n|^2dx.
	\end{align}
	Furthermore, we have
	\begin{align}\label{step0}
	\int_{B_{\rho}(0)}\left(\D^{\perp}\e_1\cdot\D\e_2\right)\log\left(\frac{|x|}{R}\right)dx&=\int_{B_{\rho}(0)}\left(\D^{\perp}\vec{f}_1\cdot \D\vec{f}_2\right)\log\left(\frac{|x|}{R}\right)dx
	=\int_{B_{\rho}(0)}\Delta\psi\,\log\left(\frac{|x|}{R}\right)dx\nonumber\nonumber\\
	&=-\log\left(\frac{R}{\rho}\right)\int_{\partial B_{\rho}}\partial_{\nu}\psi\,d\mathscr{H}^1-\int_{B_{\rho}}\D\psi\cdot \D\log|x|dx
	\end{align}
	while by the Cauchy-Schwarz inequality
	\begin{align}\label{step1}
	\left|\int_{\partial B_{\rho}}\partial_{\nu}\psi\,d\mathscr{H}^1\right|&=\left|\int_{B_{\rho}(0)}\Delta\psi\,dx\right|=\left|\int_{B_{\rho}(0)}\D^{\perp}\vec{f}_1\cdot\D\vec{f}_2dx\right|
	\leq \frac{1}{2}C_3(n)\int_{B_{\max\ens{\rho,2r}}\setminus B_r(0)}|\D\n|^2dx.
	\end{align}
	We estimate as previously by \eqref{step00}
	\begin{align}\label{step2}
	\left|\int_{B_{\rho}}\D\psi\cdot \D\log|x|dx\right|\leq \frac{1}{2} \np{\D\psi}{2,1}{B_{\rho}(0)}\np{\frac{1}{|x|}}{2,\infty}{B_\rho(0)}\leq \frac{\sqrt{\pi}}{2}\Gamma_2C_3(n)\int_{B_{\max\ens{\rho,2r}}\setminus\bar{B}_r(0)}|\D\n|^2dx.
	\end{align}
	Therefore, \eqref{step0}, \eqref{step1} and \eqref{step2} yield
	\begin{align}\label{main8}
	\left|\int_{B_{\rho}}\left(\D^{\perp}\e_1\cdot \D\e_2\right)\log\left(\frac{|x|}{R}\right)dx\right|\leq \left(\frac{1}{2}C_3(n)\log\left(\frac{R}{\rho}\right)+\frac{\sqrt{\pi}}{2}\Gamma_2C_3(n)\right)\int_{B_{\max\ens{\rho,2r}}\setminus\bar{B}_r(0)}|\D\n|^2dx.
	\end{align}
	Finally, by \eqref{degree}, \eqref{main7} and \eqref{main8} we obtain for some universal constant $\Gamma_0=\Gamma_0(n)$
	\begin{align}\label{main9}
	\frac{1}{2\pi}\left|\int_{\partial B_{\rho}}\partial_{\nu}\mu\,d\mathscr{H}^1\right|\leq \Gamma_0(n)\left(\int_{B_{\max\ens{\rho,2r}}\setminus \bar{B}_r(0)}|\D\n|^2dx+\frac{1}{\log\left(\frac{R}{\rho}\right)}\int_{\Omega}|\D\n|^2dx\right)
	\end{align}
	which completes the proof of the theorem, up to the $L^{\infty}$ estimate which is a direct consequence of the inequality $4r<R$ and of Proposition \ref{conf}.
\end{proof}

\section{Improved energy quantization}\label{sec3}

In this section, we build on \cite{quanta} to obtain an improved no-neck energy. This estimate is crucial to obtain the first result on the removability of the second residue (see Theorems \ref{onebubble},  \ref{multibubble}). 

\begin{theorem}\label{improvedquanta}
	Under the hypothesis of Theorem \ref{ta}, for all $0<\alpha<1$ let $\Omega_k(\alpha)=B_{\alpha R_k}\setminus\bar{B}_{\alpha^{-1}r_k}(0)$ be a neck domain and $\theta_0\in \Z$ such that (by Theorem \ref{integer1})
	\begin{align}
		\theta_0-1=\lim\limits_{\alpha\rightarrow 0}\lim\limits_{k\rightarrow \infty}\int_{\partial B_{\alpha^{-1}r_k}(0)}\partial_{\nu}\lambda_k\,d\mathscr{H}^1,
	\end{align}
	and define
	\begin{align*}
		\Lambda=\sup_{k\in \N}\left(\np{\D\lambda_k}{2,\infty}{\Omega_k(1)}+\int_{\Omega_k(1)}|\D\n_k|^2dx\right).
	\end{align*}
    Then there exist a universal constant $\Gamma_{5}=\Gamma_5(n)$, and $\alpha_0=\alpha_0(\{\phi_k\}_{k\in \N})>0$ such that for all $0<\alpha<\alpha_0$ and $k\in \N$ large enough, 
	\begin{align}\label{l21n}
		\np{\D\n_k}{2,1}{\Omega_k(\alpha)}&\leq \Gamma_5(n)e^{\Gamma_5(n)\Lambda}\left(1+\np{\D\n_k}{2}{\Omega_k(4\alpha)}\right)\np{\D\n_k}{2}{\Omega_k(4\alpha)}. 
	\end{align}
	In particular, we deduce by the $L^2$ no-neck energy
	\begin{align*}
		\lim\limits_{\alpha\rightarrow 0}\limsup_{k\rightarrow\infty}\np{\D\n_k}{2,1}{\Omega_k(\alpha)}=0.
	\end{align*}
\end{theorem}
\begin{proof}
	\textbf{Step 1: $L^{2,1}$-quantization of the mean curvature.}
	Here, we will prove that 
	\begin{align*}
		\lim\limits_{\alpha\rightarrow 0}\limsup_{k\rightarrow \infty}\left(\np{e^{\lambda_k}\H_k}{2,1}{\Omega_k(\alpha)}+\np{e^{\lambda_k}\D\H_k}{1}{\Omega_k(\alpha)}\right)=0.
	\end{align*}
    This statement is a consequence of the following lemma.
    \begin{theorem}\label{quantamean}
    	There exists constants $R_0(n),\epsilon_3(n)>0$ with the following property. Let $0<100r<R\leq R_0(n)$, and $\phi:B(0,R)\rightarrow \R^n$ be a weak conformal immersion of finite total curvature, such that 
    	\begin{align*}
    		\sup_{r<s<R/2}\int_{B_{2s}\setminus \bar{B}_s(0)}|\D\n|^2dx\leq \epsilon_3(n).
    	\end{align*}
    	Set $\Omega=B_R\setminus \bar{B}_r(0)$, and
    	\begin{align*}
    		\Lambda=\np{\D\lambda}{2,\infty}{\Omega}+\int_{\Omega}|\D\n|^2dx,
    	\end{align*}
    	where $\lambda$ is the conformal parameter of $\phi$. Then there exists a universal constant $C_{4}=C_{4}(n)$ such that for all $\left(\dfrac{4r}{5R}\right)^{\frac{1}{3}}<\alpha<\dfrac{1}{5}$, we have
    	\begin{align}\label{precisedL21}
    		\np{e^{\lambda}\H}{2,1}{\Omega_{\alpha}}+\np{e^{\lambda}\D\H}{1}{\Omega_{\alpha}}\leq C_{4}(n)\left(1+\Lambda\right)\,e^{4\Gamma_1(n)\Lambda}\left(1+\np{\D\n}{2}{\Omega}\right)\np{\D\n}{2}{\Omega}.
    	\end{align}
    \end{theorem}
\begin{proof}	Define for all $\left(\dfrac{r}{R}\right)^{\frac{1}{2}}<\alpha<1$ the open subset $\Omega_{\alpha}=B_{\alpha R}\setminus \bar{B}_{\alpha^{-1}r}$ of $\Omega$.
	We follow step by steps the proof of Lemma VI.$6$ of \cite{quanta}. First, the pointwise estimate on $\D\n$ is identical and we find that there exists $C_0=C_0(n),C_0'=C_0'(n)>0$ such that for all $z\in B_{4R/5}\setminus \bar{B}_{5r/4}(0)$
	\begin{align}\label{localcontrol}
		|\D\n(z)|\leq \frac{C_0}{|z|^2}\int_{B_{2|z|}\setminus \bar{B}_{|z|/2}(0)}|\D\n|^2d\leb^2\leq \frac{C_0'\sqrt{\epsilon_3(n)}}{|z|}
	\end{align}
	so that 
	\begin{align*}
		\np{\D\n}{2,\infty}{\Omega}\leq \sqrt{\pi}C_0'\sqrt{\epsilon_3(n)}
	\end{align*}
	and we can choose  
	$
	\epsilon_3(n)=\dfrac{\epsilon_1(n)^2}{\sqrt{\pi}C_0'(n)}.
	$
	Therefore, thanks to Theorem \ref{neckfine}, there exists $d\in \R$ such that 
	\begin{align*}
		|d|\leq 3(C_1(n)+1)\Lambda.
	\end{align*}
	and for all $\left(\dfrac{5}{4}\right)^{\frac{2}{3}}\left(\dfrac{r}{R}\right)^{\frac{1}{3}}=\left(\dfrac{\frac{5r}{4}}{\frac{4R}{5}}\right)^{\frac{1}{4}}<\dfrac{5\alpha}{4}<\dfrac{1}{4}$, there exists $A_{\alpha}\in \R$ such that 
	\begin{align}\label{localconf}
		\np{\lambda-d\log|z|-A_{\alpha}}{\infty}{\Omega_{\alpha}}\leq \Gamma_0'(n)\sqrt{\frac{5\alpha}{4}}\Lambda+\Gamma_0'(n)\leq \Gamma_1(n)\left(\sqrt{\alpha}\Lambda+\int_{\Omega}|\D\n|^2dx\right).
	\end{align}
    As $\phi$, is Willmore, the following $1$-form is closed :
    \begin{align*}
    	\vec{\alpha}=\Im\left(\partial\H+|\H|^2\partial\phi+g^{-1}\otimes \s{\H}{\h_0}\otimes \bar{\partial}\phi\right).
    \end{align*}
    As $\phi$ is well-defined on $B(0,R)$, the the Poincar\'{e} lemma, implies that there exists $\vec{L}:B(0,R)\rightarrow \R^n$ such that 
    \begin{align}\label{defL}
    	2i\,\partial\vec{L}=\partial\H+|\H|^2\partial\phi+g^{-1}\otimes \s{\H}{\h_0}\otimes \bar{\partial}\phi.
    \end{align}
    Now, introduce for $0<s<R/2$
    \begin{align*}
    	\delta(s)=\left(\frac{1}{s^2}\int_{B_{2s}\setminus \bar{B}_{s/2}(0)}|\D\n|^2dx\right)^{\frac{1}{2}}.
    \end{align*}
    Then we have trivially for all $2r<s<R/2$
    \begin{align}\label{ineqdelta}
    	s\delta(s)\leq \left(\int_{\Omega}|\D\n|^2dx\right)^{\frac{1}{2}}=\np{\D\n}{2}{\Omega}
    \end{align}
    and Fubini's theorem implies that for all $r\leq r_1<r_2\leq R/2$ 
    \begin{align}\label{intdelta}
    	\int_{r_1}^{r_2}s\delta(s)^2ds&=\int_{r_1}^{r_2}\frac{1}{s}\left(\int_{B_{2s}\setminus \bar{B}_{s/2}(0)}|\D\n(x)|^2dx\right)ds=\int_{r_1}^{r_2}\int_{B_{2r_2}\setminus \bar{B}_{r_1/2}(0)}\frac{|\D\n(x)|^2}{s}\bm{1}_{\ens{s/2\leq |x|\leq 2s}}dxds\nonumber\\
    	&=\log(4)\int_{B_{2r_2}\setminus \bar{B}_{r_1}(0)}|\D\n|^2dx.
    \end{align}
    Now, \eqref{localcontrol} shows that 
    \begin{align}\label{localH}
    	\max\ens{e^{\lambda(z)}|\H(z)|,e^{\lambda(z)}|\H_0(z)|}\leq |\D\n(z)|\leq C_0\delta(|z|).
    \end{align}
    Furthermore, the same argument of \cite{quanta} (see \cite{beriviere} for more details) using a Theorem from \cite{standardargument} implies that there exists a constant $C_1=C_1(n)$ such that 
    \begin{align}\label{localdH}
    	e^{\lambda(z)}|\D\H(z)|\leq C_1\frac{\delta(|z|)}{|z|}\qquad \text{for all}\;\, z\in \Omega_{1/2}
    \end{align}
    Therefore, we have thanks to \eqref{defL}, \eqref{localH} and \eqref{localdH}
    \begin{align}\label{localdL}
    	|\D\vec{L}(z)|=2|\partial \vec{L}(z)|\leq e^{-\lambda(z)}\left(C_1\frac{\delta(z)}{|z|}+2C_0\delta(z)^2\right).
    \end{align}
    Now assume for simplicity that $\alpha=1/2$ (then we do not need to use the precised form \eqref{localconf} and we can use instead Lemma V.$3$ from \cite{quanta}).
    Denoting for all $r<s<R$
    \begin{align*}
    \vec{L}_s=\dashint{\partial B_s}\vec{L}\,d\mathscr{H}^1,
    \end{align*}
    we deduce from \eqref{localdH} that for all $z\in \Omega_{1/2}$ (taking $\alpha=1/2$ in \eqref{localconf})
    \begin{align}\label{pointwiseL1}
    	|\vec{L}(z)-\vec{L}_{|z|}|&\leq \int_{\partial B_{|z|}}|\D\vec{L}|d\mathscr{H}^1\leq 2\pi e^{2\Gamma_1\Lambda}e^{-\lambda(|z|)}\left(C_1\delta(|z|)+2C_0|z|\delta(|z|)\cdot \delta(|z|)\right)\nonumber\\
    	&\leq 2\pi e^{2\Gamma_1\Lambda}e^{-\lambda(|z|)}\left(C_1+2C_0\np{\D\n}{2}{\Omega}\right)\delta(|z|).
    \end{align}
    Then we get
    \begin{align}\label{deuxdeux}
    	&\int_{\Omega_{1/2}}|\vec{L}(z)-\vec{L}_{|z|}|^2e^{2\lambda(z)}|dz|^2\leq 2\pi e^{2\Gamma_1\Lambda}\left(C_1+2C_0\np{\D\n}{2}{\Omega}\right)\int_{\Omega}\delta(|z|)^2|dz|^2\nonumber\\
    	&=4\pi^2e^{2\Gamma_1\Lambda}\left(C_1+2C_0\np{\D\n}{2}{\Omega}\right)\int_{2r}^{R/2}s\delta^2(s)ds\nonumber\\
    	&=4\pi^2\log(4)e^{2\Gamma_1\Lambda}\left(C_1+2C_0\np{\D\n}{2}{\Omega}\right)\np{\D\n}{2}{\Omega}.
    \end{align}
    Now, we continue the proof in an exact same way to obtain the pointwise estimate (for some universal constant $C_2=C_2(n)$)
    \begin{align}\label{pointwiseL2}
    	e^{\lambda(z)}|z||\vec{L}_{|z|}|\leq C_2e^{2\Gamma_1\Lambda}\left(1+\np{\D\n}{2}{\Omega}\right)\np{\D\n}{2}{\Omega}.
    \end{align}
    Therefore, we get
    \begin{align}\label{deuxinfini}
    	\np{e^{\lambda(z)}\vec{L}_{|z|}}{2,\infty}{\Omega_{1/2}}\leq 2\sqrt{\pi}C_2e^{2\Gamma_1\Lambda}\left(1+\np{\D\n}{2}{\Omega}\right)\np{\D\n}{2}{\Omega}.
    \end{align}
    Combining \eqref{deuxdeux} and \eqref{deuxinfini} implies as $\np{\,\cdot\,}{2,\infty}{\,\cdot\,}\leq \np{\,\cdot\,}{2}{\,\cdot\,}$ that 
    \begin{align*}
    	\np{e^{\lambda}\vec{L}}{2,\infty}{\Omega_{1/2}}&\leq \left(8\pi^2\log(4)+2\sqrt{\pi}C_2\right)e^{2\Gamma_1\Lambda}\left(1+\np{\D\n}{2}{\Omega}\right)\np{\D\n}{2}{\Omega}\\
    	&=C_3(n)e^{2\Gamma_1\Lambda}\left(1+\np{\D\n}{2}{\Omega}\right)\np{\D\n}{2}{\Omega}.
    \end{align*}
    The estimates \eqref{pointwiseL1}, \eqref{pointwiseL2} and \eqref{ineqdelta} imply that for all $z\in \Omega_{1/2}$ 
    \begin{align}\label{pointwiseL3}
    	e^{\lambda(z)}|\vec{L}(z)|&\leq \left(2\pi\max\ens{C_1(n),2C_0(n)}+C_3(n)\right)e^{2\Gamma_1\Lambda}\left(1+\np{\D\n}{2}{\Omega}\right)\left(\np{\D\n}{2}{\Omega}+|z|\delta(|z|)\right)|z|^{-1}\nonumber\\
    	&\leq 2\left(2\pi\max\ens{C_1(n),2C_0(n)}+C_3(n)\right)e^{2\Gamma_1\Lambda}\left(1+\np{\D\n}{2}{\Omega}\right)\np{\D\n}{2}{\Omega}\,|z|^{-1}\nonumber\\
    	&=C_4(n)e^{2\Gamma_1\Lambda}\left(1+\np{\D\n}{2}{\Omega}\right)\np{\D\n}{2}{\Omega}\frac{1}{|z|}.
    \end{align}
    Now, recall that there exists $S:B(0,R)\rightarrow \R$ and $\vec{R}:B(0,R)\rightarrow\Lambda^2\R^n$ such that 
    \begin{align*}
    	\left\{\begin{alignedat}{1}
    	&\D S=\vec{L}\cdot\D\phi\\
    	&\D\vec{R}=\vec{L}\wedge \D\phi+2\,\H\wedge \D^{\perp}\phi,
    	\end{alignedat}\right.
    \end{align*}
    we trivially obtain from the pointwise inequality \eqref{pointwiseL3}, \eqref{localH} and \eqref{ineqdelta} for all $z\in \Omega_{1/2}$
    \begin{align*}
    	|\D S(z)|\leq 2C_4(n)e^{2\Gamma_1\Lambda}\left(1+\np{\D\n}{2}{\Omega}\right)\np{\D\n}{2}{\Omega}\frac{1}{|z|}
    \end{align*}
    and
    \begin{align*}
    	|\D\vec{R}(z)|&\leq 2C_4(n)e^{2\Gamma_1\Lambda}\left(1+\np{\D\n}{2}{\Omega}\right)\np{\D\n}{2}{\Omega}\frac{1}{|z|}+4C_0(n)\delta(|z|)\\
    	&\leq 2\left(C_4(n)+2C_0(n)\right)e^{2\Gamma_1\Lambda}\left(1+\np{\D\n}{2}{\Omega}\right)\np{\D\n}{2}{\Omega}\frac{1}{|z|}.
    \end{align*}
    Therefore, if $C_5(n)=4\sqrt{\pi}(C_4(n)+2C_0(n))>0$, we deduce that 
    \begin{align}\label{l2inftybis}
    	&\np{\D S}{2,\infty}{\Omega_{1/2}}\leq C_5(n)e^{2\Gamma_1\Lambda}\left(1+\np{\D\n}{2}{\Omega}\right)\np{\D\n}{2}{\Omega}\nonumber\\
    	&\np{\D \vec{R}}{2,\infty}{\Omega_{1/2}}\leq C_5(n) e^{2\Gamma_1\Lambda}\left(1+\np{\D\n}{2}{\Omega}\right)\np{\D\n}{2}{\Omega}.
    \end{align}
    Now, define for all $2r\leq \rho<\dfrac{R}{2}$
    \begin{align*}
    	S_\rho=\dashint{\partial B_\rho(0)}S\,d\mathscr{H}^1,\qquad \vec{R}_{\rho}=\dashint{\partial B_{\rho}(0)}\vec{R}\,d\mathscr{H}^1
    \end{align*}
    following the exact same steps as \cite{quanta}, we find that for some universal constant $C_6=C_6(n)$
    \begin{align}\label{intdelta2}
    	\left|\frac{dS_\rho}{d\rho}\right|+\left|\frac{d\vec{R}_{\rho}}{d\rho}\right|\leq C_6(n)e^{2\Gamma_1\Lambda}\left(1+\np{\D\n}{2}{\Omega}\right)\np{\D\n}{2}{\Omega}\delta(\rho).
    \end{align}
    Therefore, \eqref{intdelta} and \eqref{intdelta2} imply that 
    \begin{align}\label{radialderivative}
    	\int_{2r}^{\frac{R}{2}}\left(\left|\frac{d S_{\rho}}{d\rho}\right|^2+\left|\frac{d\vec{R}_{\rho}}{d\rho}\right|^2\right)\rho\,d\leb^1(\rho)\leq C_6(n)^2e^{4\Gamma_1\Lambda}\left(1+\np{\D\n}{2}{\Omega}\right)^2\np{\D\n}{2}{\Omega}^3.
    \end{align}
    We will now use a precised version of Lemma VI.$2$ of \cite{quanta}  (proved in \cite{angular}, see also \cite{quantamoduli}). 
    \begin{lemme}\label{newl2estimate}
    	There exists a universal constant $R_0>0$ with the following property.
    	Let $0<4r<R<R_0$, $\Omega=B(0,R)\setminus \bar{B}(0,r)\rightarrow \R$, $a,b:\Omega \rightarrow \R$ such that $\D a\in L^{2,\infty}(\Omega)$ and $\D b\in L^2(\Omega)$, and $u:\Omega\rightarrow \R$ be a solution of 
    	\begin{align*}
    		\D \varphi=\D a\cdot \D^{\perp}b\qquad \text{in}\;\ \Omega.
    	\end{align*}
    	Furthermore, define for $r\leq \rho\leq R$
    	\begin{align*}
    		\bar{\varphi}_r=\dashint{\partial B_{\rho}(0)}\varphi\,d\mathscr{H}^1=\frac{1}{2\pi \rho}\int_{\partial B_{\rho}(0)}\varphi\,d\mathscr{H}^1.
    	\end{align*}
    	Then $\D \varphi\in L^2(\Omega)$, and there exists a positive constant $C_0>0$ independent of $0<4r<R$ such that for all $\left(\dfrac{r}{R}\right)^{\frac{1}{2}}<\alpha<\dfrac{1}{2}$
    	\begin{align*}
    		\np{\D \varphi}{2}{B_{\alpha R}\setminus \bar{B}_{\alpha^{-1}r}}\leq C_0\np{a}{2,\infty}{\Omega}\np{\D b}{2}{\Omega}+C_0\np{\D \bar{\varphi}_r}{2}{\Omega}+C_0\np{\D \varphi }{2,\infty}{\Omega}.
    	\end{align*}
    \end{lemme}
    \begin{proof}
    	Let $\tilde{a}:B(0,R)\rightarrow \R$ and $\tilde{b}:B(0,R)\rightarrow \R$ the extensions of $a$ and $b$ given by Theorem \ref{extop}. As $0<4r<R$ and scaling invariance of the $L^{2,\infty}$ and the $L^{2}$ norm of the gradient, we deduce that  there exists a universal constant $C_0>0$ such that 
    	\begin{align*}
    		&\np{\D \tilde{a}}{2,\infty}{B(0,R)}\leq C_0\left(\np{\D a}{2,\infty}{\Omega}+\np{a}{2,\infty}{\Omega}\right)\\
    		&\np{\D \tilde{b}}{2}{B(0,R)}\leq C_0\left(\np{\D b}{2}{\Omega}+\np{b}{2}{\Omega}\right).
    	\end{align*}
    	Thanks to Poincar\'{e}-Wirtinger inequality, and as $\tilde{a}=a$ and $\tilde{b}=b$ on $\Omega$, we deduce that 
    	\begin{align*}
    		\np{\D \tilde{a}}{2,\infty}{B(0,R)}&\leq C_0\left(\np{\D a}{2,\infty}{\Omega}+\np{a-\bar{\tilde{a}}_{B(0,R)}}{2,\infty}{\Omega}\right)
    		=C_0\left(\np{\D a}{2,\infty}{B(0,R)}+\np{\tilde{a}-\bar{\tilde{a}}_{B(0,R)}}{2,\infty}{\Omega}\right)\\
    		&\leq C_0\left(\np{\D a}{2,\infty}{\Omega}+\np{\tilde{a}-\bar{\tilde{a}}_{B(0,R)}}{2,\infty}{B(0,R)}\right)\\
    		&\leq C_0\np{\D a}{2,\infty}{\Omega}+\Gamma_0C_{PW}(L^{2,\infty})R\np{\D \tilde{a}}{2,\infty}{B(0,R)}.
    	\end{align*}
    	Therefore, if $C_0C_{PW}(L^{2,\infty})R_0\leq \dfrac{1}{2}$, we find
    	\begin{align*}
    		\np{\D \tilde{a}}{2,\infty}{B(0,R)}\leq 2C_0\np{\D a}{2,\infty}{\Omega},
    	\end{align*}
    	and likewise, provided $C_0C_{PW}(L^2)R_0\leq \dfrac{1}{2}$, we find
    	\begin{align*}
    		\np{\D \tilde{b}}{2}{B(0,R)}\leq 2C_0\np{\D b}{2}{\Omega}.
    	\end{align*}
    	Now, let $u:B(0,R)\rightarrow \R$ be the solution of 
    	\begin{align*}
    		\left\{\begin{alignedat}{2}
    		\Delta u&=\D \tilde{a}\cdot \D^{\perp}\tilde{b}\qquad &&\text{in}\;\, B(0,R)\\
    		u&=0\qquad &&\text{on}\;\, \partial B(0,R).
    		\end{alignedat}\right.
    	\end{align*}
    	Then the improved Wente inequality of Bethuel (\cite{helein}, $3.3.6$) and the scaling invariance shows that there exists a universal constant $C_1>0$ such that 
    	\begin{align*}
    		\np{\D u}{2}{B(0,R)}\leq C_1\np{\D \tilde{a}}{2,\infty}{B(0,R)}\np{\D \tilde{b}}{2}{B(0,R)}\leq 4
    		C_0^2C_1\np{\D a}{2,\infty}{\Omega}\np{\D b}{2}{\Omega}.
    	\end{align*}
    	Now, let $v=\varphi-u-\bar{(\varphi-u)}_r$. Then $v$ is a harmonic function such that  for all $r<\rho<R$
    	\begin{align*}
    		\int_{\partial B_{\rho}}\partial_{\nu}v\,d \mathscr{H}^1=0.
    	\end{align*}
    	Therefore, Lemma \ref{l2linfty} implies that 
    	\begin{align*}
    		\np{\D v}{2}{B_{\alpha R}\setminus \bar{B}_{\alpha^{-1}r}(0)}\leq \Gamma_1\np{\D v}{2,\infty}{\Omega}
    		&\leq \Gamma_1'\left(\np{\D a}{2,\infty}{\Omega}\np{\D b}{2}{\Omega}+\np{\D \varphi_r}{2}{\Omega}+\np{\D\varphi}{2,\infty}{\Omega}\right)
    	\end{align*}
    	which concludes the proof. 
    \end{proof}
    Now, recall that the following system holds
    \begin{align*}
    	\left\{\begin{alignedat}{1}
    	\Delta S&=-\ast \D\n\cdot \D^{\perp}\vec{R}\\
    	\Delta \vec{R}&=(-1)^n\ast \left(\D\n \res \D^{\perp}\vec{R}\right)+\ast \D\n \cdot \D^{\perp}S.
    	\end{alignedat} \right.
    \end{align*}
    First, thanks to Lemma IV.$1$ of \cite{quanta}, we extend the restriction $\n:B_R\setminus\bar{B}_r(0)\rightarrow \mathscr{G}_{n-2}(\R^n)$ to a map $\tilde{\n}:B(0,R) \rightarrow \mathscr{G}_{n-2}(\R^n)$ such that 
    \begin{align*}
    	\np{\D\tilde{\n}}{2}{B(0,R)}\leq C_0(n)\np{\D\n}{2}{\Omega}.
    \end{align*}
    In particular we have 
    \begin{align}\label{newsystem}
    \left\{\begin{alignedat}{2}
    \Delta S&=-\ast \D\tilde{\n}\cdot \D^{\perp}\vec{R}\qquad&& \text{in}\;\, \Omega\\
    \Delta \vec{R}&=(-1)^n\ast \left(\D\tilde{\n} \res \D^{\perp}\vec{R}\right)+\ast \D\tilde{\n} \cdot \D^{\perp}S\qquad&& \text{in}\;\, \Omega
    \end{alignedat} \right.
    \end{align}
    Therefore, applying the proof of Lemma \ref{newl2estimate} by using the already constructed extension of $\n$, we deduce thanks to \eqref{l2inftybis} and \eqref{radialderivative} that 
    \begin{align*}
    	\np{\D S}{2}{\Omega_{1/4}}+\np{\D\vec{R}}{2}{\Omega_{1/4}}\leq C_7(n)e^{2\Gamma_1\Lambda}\left(1+\np{\D\n}{2}{\Omega}\right)\np{\D\n}{2}{\Omega}.
    \end{align*}
    As in \cite{quanta}, we obtain readily
    \begin{align}\label{lastbootstrap}
    	&\np{\D S}{2}{\Omega_{1/2}}+\np{\D \vec{R}}{2}{\Omega_{1/2}}\leq C_8(n)e^{2\Gamma_1\Lambda}\left(1+\np{\D\n}{2}{\Omega}\right)\np{\D\n}{2}{\Omega}\nonumber\\
    	&\np{S}{\infty}{\Omega_{1/2}}+\np{\vec{R}}{\infty}{\Omega_{1/2}}\leq C_8(n)e^{2\Gamma_1\Lambda}\left(1+\np{\D\n}{2}{\Omega}\right)\np{\D\n}{2}{\Omega}.
    \end{align}
    Now, introduce the following slight variant from a Lemma of \cite{angular}.
    \begin{lemme}\label{hardy}
    	Let $R_0>0$ be the constant of Lemma \ref{newl2estimate}.
    	Let $0<16r<R<R_0$, $\Omega=B(0,R)\setminus \bar{B}(0,r)\rightarrow \R$, $a,b:\Omega \rightarrow \R$ such that $\D a\in L^{2}(\Omega)$ and $\D b\in L^2(\Omega)$, and $\varphi:\Omega\rightarrow \R$ be a solution of 
    	\begin{align*}
    	\Delta \varphi=\D a\cdot \D^{\perp}b\quad \text{in}\;\ \Omega.
    	\end{align*}
    	Assume that $\np{\varphi}{\infty}{\partial\Omega}<\infty$. Then there exists a universal constant $C_1>0$ such that for all $\left(\dfrac{r}{R}\right)^{\frac{1}{2}}<\alpha<\dfrac{1}{4}$, 
    	\begin{align*}
    		\np{\varphi}{\infty}{\Omega}+\np{\D\varphi}{2,1}{B_{\alpha R}\setminus \bar{B}_{\alpha^{-1}r}(0)}+\np{\D^2\varphi}{1}{B_{\alpha R}\setminus\bar{B}_{\alpha^{-1}r}(0)}\leq C_1\left(\np{\D a}{2}{\Omega}\np{\D b}{2}{\Omega}+\np{\varphi}{\infty}{\partial\Omega}\right).
    	\end{align*}
    \end{lemme}
    \begin{proof}
    As in the proof of Lemma \ref{newl2estimate}, introduce extensions $\tilde{a}:B(0,R)\rightarrow\R$ and $\tilde{b}:B(0,R)\rightarrow \R$ of $a$ and $b$, such that 
    \begin{align*}
    	&\npn{\D\tilde{a}}{2}{B(0,R)}\leq 2C_0\np{\D a}{2}{\Omega}\\
    	&\npn{\D\tilde{b}}{2}{B(0,R)}\leq 2C_0\np{\D b}{2}{\Omega}.
    \end{align*}
    Now, let $v:B(0,R)\rightarrow \R$ be the solution of 
    \begin{align*}
    \left\{\begin{alignedat}{2}
    \Delta v&=\D \tilde{a}\cdot \D^{\perp}\tilde{b}\qquad &&\text{in}\;\, B(0,R)\\
    v&=0\qquad &&\text{on}\;\, \partial B(0,R).
    \end{alignedat}\right.
    \end{align*}
    Then the improved Wente inequality and the Coifman-Lions-Meyer-Semmes estimate (\cite{meyer}) shows (by scaling invariance of the different norms considered) that 
    \begin{align}\label{cmls}
    	\np{v}{\infty}{B(0,R)}+\np{\D v}{2,1}{B(0,R)}+\np{\D^2v}{1}{B(0,R)}\leq C_1\np{\D a}{2}{\Omega}\np{\D b}{2}{\Omega}.
    \end{align}
    Now let $u=\varphi-v$. Then $u$ is harmonic, and let $d\in \R$, $\ens{a_n}_{n\in \Z}\subset \C$ such that 
    \begin{align*}
    	u(z)=a_0+d\log|z|+\Re\left(\sum_{n\in \Z^{\ast}}a_nz^n\right).
    \end{align*}
    Then we have by the maximum principle for all $r\leq \rho\leq R$
    \begin{align*}
    	|a_0+d\log\rho|=\left|\frac{1}{2\pi}\int_{0}^{2\pi}u(\rho e^{i\theta})d\theta\right|\leq \np{u}{\infty}{\partial \Omega}.
    \end{align*}
    Therefore, we have
    \begin{align}\label{throwlog0}
    	|d|\log\left(\frac{R}{r}\right)=\left|a_0+d\log R-\left(a_0+d\log r\right)\right|\leq |a_0+d\log R|+|a_0+d\log r|\leq 2\np{u}{\infty}{\partial \Omega}.
    \end{align}
    Now, recall that 
    \begin{align}\label{throwlog1}
    	&\np{\D \log|z|}{2,1}{B_{R}\setminus \bar{B}_{r}(0)}=4\sqrt{\pi}\left(\log\left(\frac{R}{r}\right)+\log\left(1+\sqrt{1-\left(\frac{r}{R}\right)^2}\right)\right)\nonumber\\
    	&\np{\D^2\log|z|}{1}{B_R\setminus \bar{B}_{r}(0)}=4\np{\p{z}^2\log|z|}{1}{B_R\setminus\bar{B}_r(0)}=4\pi\log\left(\frac{R}{r}\right).
    \end{align}
    Therefore, as $R>4r$, \eqref{throwlog0} and \eqref{throwlog1} imply that  
    \begin{align}\label{throwlog2}
    	&\np{\D\left(d\log|z|\right)}{2,1}{B_R\setminus \bar{B}_r(0)}\leq 4\sqrt{\pi}\left(\log\left(\frac{R}{r}\right)+\log(2)\right)|d|\leq 16\sqrt{\pi}\np{u}{\infty}{\partial \Omega}\nonumber\\
    	&\np{\D^2\left(d\log|z|\right)}{1}{B_R\setminus \bar{B}_r(0)}\leq 8\pi\np{u}{\infty}{\partial \Omega}.
    \end{align}
    These estimates \eqref{throwlog0} imply by Lemmas \ref{l21l2} and \ref{hardy} imply that 
    \begin{align}\label{throwlog3}
    	&\np{\D u}{2,1}{B_{\alpha R}\setminus \bar{B}_{\alpha^{-1}r}(0)}+\np{\D^2u}{1}{B_{\alpha R}\setminus \bar{B}_{\alpha^{-1}r}(0)}\leq 64\frac{\sqrt{2}+\sqrt{\pi}}{\sqrt{15}}(2\alpha) \np{\D \left(u-d\log|z|\right)}{2}{B_{R/2}\setminus \bar{B}_{r/2}(0)}\nonumber\\
    	&+24\pi\np{u}{\infty}{\partial \Omega}\nonumber\\
    	&\leq 128\frac{\sqrt{2}+\sqrt{\pi}}{\sqrt{15}}\alpha\np{\D u}{2}{B_{R/2}\setminus\bar{B}_{2r}(0)}+\left(24\pi+256\pi \frac{\sqrt{2}+\sqrt{\pi}}{\sqrt{15}}\alpha\right)\np{u}{\infty}{\partial \Omega}.
    \end{align} 
    Now, recall that the mean value formula and the maximum principle (\cite{hanlin} $1.10$) imply that for all $x\in B_{R}\setminus \bar{B}_r(0)$, and $0<\rho<\dist(x,\partial \Omega)$,
    \begin{align}\label{meanvalue}
    	|\D u(x)|\leq \frac{2}{\rho}\np{u}{\infty}{\partial B_{\rho}(x)}\leq \frac{2}{\rho}\np{u}{\infty}{\partial\Omega}.
    \end{align}
    As 
    \begin{align}\label{throwlog4}
    	\np{\D u}{2}{B_{R/2}^2\setminus \bar{B}_{2r}(0)}^2&=\int_{\partial B_{R/2}(0)} u\,\partial_{\nu}u\,d\mathscr{H}^1-\int_{\partial B_{2r}(0)}u\,\partial_{\nu}u\,d\mathscr{H}^1\nonumber\\
    	&\leq \np{u}{\infty}{\partial \Omega}\left(\int_{\partial B_{R/2}(0)}|\D u|d\mathscr{H}^1+\int_{\partial B_{2r}(0)}|\D u|d\mathscr{H}^1\right)
    \end{align}
    the estimate \eqref{meanvalue} shows that 
    \begin{align}\label{throwlog5}
    	\int_{\partial B_{R/2}(0)}|\D u|d\mathscr{H}^1+\int_{\partial B_{2r}(0)}|\D u|d\mathscr{H}^1&\leq 4\int_{\partial B_{R/2}(0)}\frac{\np{u}{\infty}{\partial \Omega}}{R}d\mathscr{H}^1+\int_{\partial B_{2r}(0)}\frac{\np{u}{\infty}{\partial \Omega}}{r}d\mathscr{H}^1\nonumber\\
    	&=8\pi\np{u}{\infty}{\partial \Omega}. 
    \end{align}
    Therefore, we have by \eqref{throwlog4} and \eqref{throwlog5}
    \begin{align}\label{throwlog6}
    	\np{\D u}{2}{B_{R/2}\setminus \bar{B}_{2r}(0)}\leq 2\sqrt{2\pi}\np{u}{\infty}{\partial \Omega}.
    \end{align}
    Combining \eqref{throwlog3} and \eqref{throwlog6} shows that 
    \begin{align}\label{throwlog7}
    	\np{\D u}{2,1}{B_{\alpha R}\setminus \bar{B}_{\alpha^{-1}r}(0)}+\np{\D^2u}{1}{B_{\alpha R}\setminus \bar{B}_{\alpha^{-1}r}(0)}&\leq \left(256\frac{2\sqrt{\pi}+\pi\sqrt{2}}{\sqrt{15}}\alpha+24\pi+256\pi\frac{\sqrt{2}+\sqrt{\pi}}{\sqrt{15}}\alpha\right)\np{u}{\infty}{\partial \Omega}\nonumber\\
    	&\leq \left(24\pi+512\pi\frac{\sqrt{2}+\sqrt{\pi}}{\sqrt{15}}\alpha\right)\np{u}{\infty}{\partial \Omega}.
    \end{align}
    Combining the maximum principle and inequalities \eqref{cmls}, \eqref{throwlog7} yields the expected estimate.
    \end{proof}
Now, apply Lemma \ref{hardy} to the estimates \eqref{lastbootstrap} shows by using the previous extension $\tilde{\n}$ of $\n$ that 
    \begin{align}\label{endl10}
    	&\np{\D S}{2,1}{\Omega_{1/2}}+\np{\D\vec{R}}{2,1}{\Omega_{1/2}}\leq C_9(n)e^{4\Gamma_1\Lambda}\left(1+\np{\D\n}{2}{\Omega}\right)\np{\D\n}{2}{\Omega}\nonumber\\
    	&\np{\D^2S}{1}{\Omega_{1/2}}+\np{\D^2\vec{R}}{1}{\Omega_{1/2}}\leq C_{9}(n)e^{4\Gamma_1\Lambda}\left(1+\np{\D\n}{2}{\Omega}\right)\np{\D\n}{2}{\Omega}.
    \end{align}
    As
    \begin{align}\label{structure}
    	e^{2\lambda}\H=\frac{1}{4}\D^{\perp}S\cdot \D\phi- \frac{1}{4}\D\vec{R}\res \D^{\perp}\phi,
    \end{align}
    we trivially have
    \begin{align}\label{endl1}
    	\np{e^{\lambda}\H}{2,1}{\Omega_{1/2}}\leq \np{\D S}{2,1}{\Omega_{1/2}}+\np{\D\vec{R}}{2,1}{\Omega_{1/2}}\leq 2C_9(n)e^{4\Gamma_1\Lambda}\left(1+\np{\D\n}{2}{\Omega}\right)\np{\D\n}{2}{\Omega}.
    \end{align}
    Now, \eqref{structure} implies that 
    \begin{align*}
    	2(\p{z}\lambda)e^{2\lambda}\H+e^{2\lambda}\p{z}\H=\frac{1}{4}\D^{\perp}\left(\p{z}S\right)\cdot \D\phi+\frac{1}{4}\D^{\perp}S\cdot \D(\p{z}\phi)-\frac{1}{4}\D(\p{z}\vec{R})\res \D^{\perp}\phi-\frac{1}{4}\D\vec{R}\res \D^{\perp}\left(\p{z}\phi\right),
    \end{align*}
    so that
    \begin{align*}
    	e^{\lambda}\p{z}\H&=-2(\p{z}\lambda)e^{\lambda}\H+\frac{1}{4}\D^{\perp}\left(\p{z}S\right)\cdot e^{-\lambda}\D\phi+\frac{1}{4}\D^{\perp}S\cdot e^{-\lambda}\D(\p{z}\phi)-\frac{1}{4}\D(\p{z}\vec{R})\res e^{-\lambda}\D^{\perp}\phi\\
    	&-\frac{1}{4}\D\vec{R}\res e^{-\lambda}\D^{\perp}\left(\p{z}\phi\right).
    \end{align*}
    As $\D\lambda\in L^{2,\infty}$, $e^{\lambda}\H\in L^{2,1}$ and $e^{-\lambda}\D^2\phi\in L^{2,\infty}$, we deduce by \eqref{endl10} and \eqref{endl1} that 
    \begin{align*}
    	\np{e^{\lambda}\D\H}{1}{\Omega_{1/2}}\leq C_{10}(n)\left(1+\Lambda\right)e^{4\Gamma_1\Lambda}\left(1+\np{\D\n}{2}{\Omega}\right)\np{\D\n}{2}{\Omega},
    \end{align*}
    and this concludes the proof of the Theorem. 
\end{proof}
For all neck region of the form $\Omega_k=B_{R_k}\setminus \bar{B}_{r_k}(0)$, define for all $0<\alpha<1$
\begin{align*}
	\Omega_k(\alpha)=B_{\alpha R_k}\setminus \bar{B}_{\alpha^{-1}r_k}(0),
\end{align*}
the estimate \eqref{precisedL21} implies that 
\begin{align}\label{ineq}
	\np{e^{\lambda_k}\H_k}{2,1}{\Omega_k(\alpha)}+\np{e^{\lambda_k}\D\H_k}{1}{\Omega_k(\alpha)}\leq C_4(n)\left(1+\Lambda\right)e^{4\Gamma_1(n)\Lambda}\left(1+\np{\D\n_k}{2}{\Omega_k(2\alpha)}\right)\np{\D\n_k}{2}{\Omega_k(2\alpha)},
\end{align}
where
\begin{align*}
	\Lambda=\sup_{k\in \N}\left(\np{\D\lambda_k}{2,\infty}{\Omega_k}+\np{\D\n_k}{2}{\Omega_k}\right)<\infty,
\end{align*}
is finite by hypothesis. 
Therefore, the no-neck energy
\begin{align}\label{noneck}
	\lim\limits_{\alpha\rightarrow 0}\limsup_{k\rightarrow \infty}\np{\D\n_k}{2}{\Omega_k(\alpha)}=0
\end{align}
implies by \eqref{ineq} that 
\begin{align*}
	\lim\limits_{\alpha\rightarrow 0}\limsup_{k\rightarrow \infty}\left(\np{e^{\lambda_k}\H_k}{2,1}{\Omega_k(\alpha)}+\np{e^{\lambda_k}\D \H_k}{1}{\Omega_k(\alpha)}\right)=0.
\end{align*}

\textbf{Step 2: $L^{2,1}$-quantization of the Weingarten tensor} The proof relies on an algebraic computation first given in \cite{riviere1} (II.$10$). We will give its easy derivation in codimension $1$.

\textbf{Algebraic identity in codimension 1.} Let $\phi:B(0,1)\rightarrow \R^3$ be a conformal immersion, and $\n:B(0,1)\rightarrow S^2$ be its unit normal. If $e^{\lambda}=\dfrac{1}{\sqrt{2}}|\D\phi|$ is the associate conformal parameter and $\e_j=e^{-\lambda}\p{x_j}\phi$ for $j=1,2$, we have by definition
\begin{align*}
	\n=\e_1\times \e_2,
\end{align*}
where $\times$ is the vector product. Recall the Grassmann identity valid for all $\vec{u},\vec{v},\vec{w}\in \R^3$
\begin{align*}
	\left(\vec{u}\times \vec{v}\right)\times \vec{w}=\s{\vec{u}}{\vec{w}}\vec{v}-\s{\vec{v}}{\vec{w}}\vec{u}.
\end{align*}
Therefore, we deduce that 
\begin{align}\label{prodvect}
\left\{\begin{alignedat}{1}
	\n\times \e_1&=\left(\e_1\times \e_2\right)\times \e_1=\s{\e_1}{\e_1}\e_2-\s{\e_2}{\e_1}\e_1=\e_2\\
	\n\times \e_2&=-\e_1.
	\end{alignedat}\right.
\end{align}
As $|\n|=1$, we have for all $j=1,2$
\begin{align*}
	\p{x_j}\n&=\s{\D_{\p{x_j}}\n}{\e_1}\e_1+\s{\D_{\p{x_j}}\n}{\e_2}\e_2
	=-\I_{1,j}\e_1-\I_{2,j}\e_2.
\end{align*}
This implies that 
\begin{align}\label{cod14}
	\D\n=(-\I_{1,1}\e_1-\I_{1,2}\e_2,-\I_{1,2}\e_1-\I_{2,2}\e_2)
\end{align}
and \eqref{prodvect} with the identity $\vec{u}\times \vec{v}=-(\vec{v}\times \vec{u})$ (valid for all $\vec{u},\vec{v}\in \R^3$) yield
\begin{align*}
	\p{x_j}\n\times \n=-\I_{1,j}\e_1\times \n-\I_{2,j}\e_2\times \n=\I_{1,j}\e_2-\I_{2,j}\e_1.
\end{align*}
Therefore, we deduce that 
\begin{align}\label{cod11}
	\D^{\perp}\n\times \n=\left(\p{x_2}\n\times \n,-\p{x_1}\n\times \n\right)=(-\I_{2,2}\e_1+\I_{1,2}\e_2,-\I_{1,1}\e_2+\I_{1,2}\e_1).
\end{align}
As
\begin{align}\label{cod12}
	e^{\lambda}H=\frac{1}{2}\left(\I_{1,1}+\I_{2,2}\right),
\end{align}
the identities \eqref{cod11} and \eqref{cod12} show that  
\begin{align}\label{cod13}
	\D^{\perp}\n\times \n+2H\D\phi%&=(-\I_{2,2}\e_1+\I_{1,2}\e_2,-\I_{1,1}\e_2+\I_{1,2}\e_1)+\left(\I_{1,1}+\I_{2,2}\right)(\e_1,\e_2)\nonumber\\
	&=(\I_{1,1}\e_1+\I_{1,2}\e_2,\I_{1,2}\e_1+\I_{2,2}\e_2).
\end{align}
         Comparing \eqref{cod13} and \eqref{cod14}, we deduce that 
\begin{align}\label{div0}
	\D\n=\n\times \D^{\perp}\n-2H\,\D\phi.
\end{align}
Taking the divergence of this equation we find
\begin{align*}
	\Delta\n=\D\n\times \D^{\perp}\n-2\,\dive(H\D\phi).
\end{align*}
\textbf{Argument in arbitrary codimension.} Then we can find a trivialisation of $\n$ such that $\n=\n_1\wedge\n_2\wedge \cdots\wedge \n_{n-2}$ satisfying the Coulomb condition
\begin{align}\label{coulomb}
	\dive\left(\D\n_{\beta}\cdot\n_{\gamma}\right)=0\qquad \text{for all}\;\, 1\leq \beta,\gamma\leq n-2.
\end{align}
Furthermore, recall that for all $1\leq \beta\leq n-2$, \cite{riviere1} implies that (using \eqref{coulomb} for the second condition)
\begin{align}\label{algebraic0}
	\D\n_{\beta}&=-\ast\left(\n\wedge\D^{\perp}\n_{\beta}\right)+\sum_{\gamma=1}^{n-2}\s{\D\n_{\beta}}{\n_{\gamma}}\cdot\n_{\gamma}-2\,H_{\beta}\D\phi
\end{align}
Taking the divergence of this equation yields by the Coulomb condition \eqref{coulomb}
\begin{align}\label{jacobian}
	\Delta\n_{\beta}=-\ast\left(\D\n\wedge \D^{\perp}\n_{\beta}\right)+\sum_{\gamma=1}^{n-2}\s{\D\n_{\beta}}{\n_{\gamma}}\cdot \D\n_{\gamma}-2\,\dive\left(H_{\beta}\D\phi\right).
\end{align}
Now, as in \eqref{movingframe} (recall that this comes from Lemma IV.$1$. in \cite{quanta}), construct for small enough $\alpha>0$ and $k$ large enough (thanks to the no-neck property) an extension $\tilde{\n}_k:B(0,\alpha R_k)\rightarrow \mathscr{G}_{n-2}(\R^n)$ of $\n_k:\Omega_k(\alpha)=B(0,\alpha R_k)\setminus \bar{B}(0,\alpha^{-1}r_k)\rightarrow \mathscr{G}_{n-2}(\R^n)$ such that  for some universal constant $C_0=C_0(n)>0$
\begin{align}\label{controlledextension}
\np{\D\tilde{\n}_k}{2}{B(0,\alpha R_k)}\leq C_0\np{\D\n_k}{2}{\Omega_k(\alpha)}.
\end{align}
Furthermore, as in Lemma IV.$1$ of \cite{quanta} (see also \cite{helein} $4.1.3-4.1.7$) we can construct extensions $\tilde{\vec{n}}_{k}^{\beta}$ of $\n_k^{\beta}$ on $B(0,\alpha R_k)$ such that 
\begin{align*}
	\tilde{\n}_k=\tilde{\n}_k^{1}\wedge \cdots\wedge \tilde{\n}_k^{n-2}\qquad \text{on}\;\, B(0,\alpha R_k)
\end{align*}
satisfying the Coulomb condition for all $1\leq \beta,\gamma\leq n-2$
\begin{align}\label{coulomb2}
	\left\{
	\begin{alignedat}{2}
	\dive\left(\D\tilde{\n}_k^{\beta}\cdot \tilde{\n}_k^{\gamma}\right)&=0\qquad &&\text{for all}\;\, B(0,\alpha R_k)\\
	\partial_{\nu}\tilde{\n}_k^{\beta}\cdot \tilde{\n}_k^{\gamma}&=0\qquad&& \text{on}\;\,\text{on}\;\, \partial B(0,\alpha R_k)
	\end{alignedat}\right.
\end{align}
and for all $1\leq \beta\leq n-2$ (by \eqref{controlledextension} for the second inequality)
\begin{align}\label{controlledtrivialisation}
	\np{\D\tilde{\n}_k^{\beta}}{2}{B(0,\alpha R_k)}\leq \tilde{C}_0(n)\np{\D\tilde{\n}_k}{2}{B(0,\alpha R_k)}\leq  C_0'(n)\np{\D\n_k}{2}{\Omega_k(\alpha)}.
\end{align}
Furthermore, using \cite{helein} $4.1.7$, we have the estimate for all $1\leq \beta\leq n-2$
\begin{align}\label{improvedcontrol}
	\np{\D\tilde{\n}_{k}^{\beta}\cdot\tilde{\n}_k^{\gamma}}{2,1}{B(0,\alpha R_k)}\leq C_0''(n)\np{\D\n_k}{2}{\Omega_k(\alpha)}^2.
\end{align}
Let us recall the argument for this crucial step.
By \eqref{coulomb2}, there exists $\vec{A}_{\beta,\gamma}:B(0,\alpha R_k)\rightarrow \R^n$ such that 
\begin{align}\label{algebraic}
	\D^{\perp}\vec{A}_{\beta,\gamma}=\D\tilde{\n}_{k}^{\beta}\cdot \tilde{\n_k}^{\gamma}.
\end{align}
Furthermore, the boundary conditions of \eqref{coulomb2} implies that we can choose $\vec{A}_{\beta\gamma}$ such that $\vec{A}_{\beta,\gamma}=0$ on $\partial B(0,\alpha R_k)$. Therefore, we have 
\begin{align}\label{cmls2}
	\left\{\begin{alignedat}{2}
	\Delta \vec{A}_{\beta,\gamma}&=\D\tilde{\n}_{k}^{\beta}\cdot\D^{\perp}\tilde{\n}_k^{\gamma}\qquad&&\text{in}\;\, B(0,\alpha R_k)\\
	\vec{A}_{\beta,\gamma}&=0\qquad&& \text{on}\;\, \partial B(0,\alpha R_k)
	\end{alignedat} \right.
\end{align}
Therefore, we get by the improved Wente estimate and \eqref{controlledtrivialisation}
\begin{align}\label{algebraic2}
	\np{\D\vec{A}_{\beta,\gamma}}{2,1}{B(0,\alpha R_k)}\leq C_0''\np{\D\tilde{\n}_k^{\beta}}{2}{B(0,\alpha R_k)}\np{\D\tilde{\n}_k^{\gamma}}{2}{B(0,\alpha R_k)}\leq C_0''C_0'(n)\np{\D\n_k}{2}{\Omega_k(\alpha)}^2.
\end{align}
Combining the pointwise identity \eqref{algebraic} with \eqref{algebraic2} yields \eqref{improvedcontrol}.

Now fix some $1\leq \beta\leq n-2$ and let $\vec{u}_k:B(0,\alpha R_k)\rightarrow \R^n$ be the unique solution of
\begin{align}
\left\{\begin{alignedat}{2}
\Delta\vec{u}_k&=-\ast\left(\D\tilde{\n}_k\wedge \D^{\perp}\tilde{\n}_k^{\beta}\right)+\sum_{\gamma=1}^{n-2}\s{\D\tilde{\n}_k^{\beta}}{\tilde{\n}_k^{\gamma}}\cdot\D\tilde{\n}_k^{\beta}\qquad &&\text{in}\;\, B(0,\alpha R_k)\\
\vec{u}_k&=0\qquad &&\text{in}\;\, \partial B(0,\alpha R_k).
\end{alignedat}\right.
\end{align}
Now, thanks to \eqref{coulomb}, we can apply \cite{meyer}, scaling invariance  and \eqref{controlledextension} to find that there exists $C_1=C_1(n)>0$ such that 
\begin{align*}
	\np{\D^2\vec{u}_k}{1}{B(0,\alpha R_k)}\leq C_1\np{\D\tilde{\n}_k}{2}{B(0,\alpha R_k)}
	^2\leq C_0^2C_1\np{\D\n_k}{2}{\Omega_k(\alpha)}^2.
\end{align*}
Furthermore, as $\vec{u}_k=0$ on $\partial B(0,\alpha R_k)$, and scaling invariance of $\np{u_k}{\infty}{B(0,\alpha R_k)}$, $\np{\D \vec{u}_k}{2,1}{B(0,\alpha R_k)}$ and Sobolev embedding, there exists $C_3=C_3(n)>0$ such that  
\begin{align}\label{quanta1}
	\np{\vec{u}_k}{\infty}{B(0,\alpha R_k)}+\np{\D\vec{u}_k}{2,1}{B(0,\alpha R_k)}+\np{\D^2\vec{u}_k}{1}{B(0,\alpha R_k)}\leq C_3\np{\D\n_k}{2}{\Omega_k(\alpha)}^2.
\end{align}
Now, by Theorem \eqref{quantamean} $H_k^{\beta}\D\phi_k\in L^{2,1}(\Omega_k(\alpha))$. Furthermore, as 
\begin{align*}
	\lim_{k\rightarrow \infty}\frac{R_k}{r_k}=0,\qquad \limsup_{k\rightarrow \infty}R_k<\infty,
\end{align*}
there exists by Theorem \ref{extop} an extension $\vec{F}_k:B(0,\alpha R_k)\rightarrow \R^n$ of $H_k^{\beta}\D\phi_k$ such that for all $k$ large enough
\begin{align*}
	\np{\vec{F}_k}{2,1}{B(0,\alpha R_k)}\leq C_4(n)\np{H_{k}^{\beta}\D\phi_k}{2,1}{\Omega_k(\alpha)}
\end{align*}
where $C_4(n)>0$ is independent of $k$ large enough and $0<\alpha<\alpha_0(n)$ fixed (small enough with respect to some $\alpha_0(n)>0$).
Now, let $\vec{v}_k:\Omega_k(\alpha)\rightarrow \R^n$ be the solution of the system
\begin{align*}
	\left\{\begin{alignedat}{2}
	\Delta \vec{v}_k&=-2\dive\left(\vec{F}_k\right)\qquad&& \text{in}\;\, B(0,\alpha R_k)\\
	\vec{v}_k&=0\qquad&& \text{on}\;\, \partial B(0,\alpha R_k).
	\end{alignedat}\right.
\end{align*}
As we trivially have
\begin{align*}
	\wp{\dive(\vec{F}_k)}{-1,(2,1)}{B(0,\alpha R_k)}\leq \np{\vec{F}_k}{2,1}{B(0,\alpha R_k)},
\end{align*}
scaling invariance and standard Calder\'{o}n-Zygmund estimates show that there exits a universal constant $C_5=C_5(n)$ such that 
\begin{align}\label{quanta2}
	\np{\D\vec{v}_k}{2,1}{B(0,\alpha R_k)}&\leq C_5(n)\np{\vec{F}_k}{2,1}{B(0,\alpha R_k)}\leq C_4(n)C_5(n)\np{H_{k}^{\beta}\D\phi_k}{2,1}{\Omega_k(\alpha)}\nonumber\\
	&\leq 2C_4(n)C_5(n)\np{e^{\lambda_k}\H_k}{2,1}{\Omega_k(\alpha)}.
\end{align}
Furthermore, the Sobolev embedding and the explicit representation of $\vec{v}_k$ (see \eqref{nulltrace}) shows that
\begin{align}\label{quanta3}
	\np{\vec{v}_k}{\infty}{B(0,\alpha R_k)}\leq \frac{1}{\sqrt{\pi}}\np{\D\vec{v}_k}{2,1}{B(0,\alpha R_k)}\leq \frac{C_4(n)C_5(n)}{\sqrt{\pi}}\np{e^{\lambda_k}\H_k}{2,1}{\Omega_k(\alpha)}.
\end{align}
Finally, let $\vec{\varphi}_k=\vec{n}_k^{\beta}-\vec{u}_k-\vec{v}_k$. The $\vec{\varphi}_k:\Omega_k(\alpha)\rightarrow \R^n$ is harmonic and
\begin{align*}
	\left\{\begin{alignedat}{2}
	\Delta\vec{\varphi}_k&=0\qquad&& \text{in}\;\, \Omega_k(\alpha)\\
	\varphi_k&=\vec{n}_k^{\beta}\qquad&& \text{on}\;\, \partial B(0,\alpha R_k)\\
	\varphi_k&=\vec{n}_k^{\beta}-\vec{u}_k-\vec{v}_k\qquad&& \text{on}\;\, \partial B(0,\alpha^{-1}r_k).
	\end{alignedat}\right.
\end{align*}
In particular, as $\vec{u}_k,\vec{v}_k,\vec{n}_k^{\beta}\in L^{\infty}(\Omega_k(\alpha))$ (as $|\n_k^{\alpha}|=1$ and using the bounds \eqref{quanta1} and \eqref{quanta4}), if $\vec{d}_k\in \R$ and $\ens{\vec{a}_n}_{n\in \Z}\subset \C^n$ are such that 
\begin{align*}
	\vec{\varphi}_k(z)=\vec{a}_0+\vec{d_k}\log|z|+\Re\left(\sum_{n\in \Z^{\ast}}\vec{a}_nz^n\right),
\end{align*}
then
\begin{align}
	|\vec{d}_k|\leq \frac{\np{\vec{\varphi}_k}{\infty}{\partial \Omega_k(\alpha)}}{\log\left(\frac{\alpha^2R_k}{r_k}\right)}\leq \frac{2}{\log\left(\frac{\alpha^2R_k}{r_k}\right)}\left(1+C_7(n)\np{\D\n_k}{2}{\Omega_k(\alpha)}^2+C_7(n)\np{e^{\lambda_k}\H_k}{2,1}{\Omega_k(\alpha)}\right),
\end{align}
so that by the proof of Lemma \ref{hardy}
\begin{align}\label{quanta0}
	&\np{\D\vec{\varphi}_k}{2,1}{\Omega_k(\alpha/2)}\leq 16\sqrt{\pi}+C_8(n)\left(\left(1+\np{\D\n_k}{2}{\Omega_k(\alpha)}\right)\np{\D\n_k}{2}{\Omega_k(\alpha)}+\np{e^{\lambda_k}\H_k}{2,1}{\Omega_k(\alpha)}\right)\nonumber\\
	&\np{\D^2\vec{\varphi}_k}{1}{\Omega_k(\alpha/2)}\leq 8\pi +C_8(n)\left(\left(1+\np{\D\n_k}{2}{\Omega_k(\alpha)}\right)\np{\D\n_k}{2}{\Omega_k(\alpha)}+\np{e^{\lambda_k}\H_k}{2,1}{\Omega_k(\alpha)}\right).
\end{align}
Finally, we have by \eqref{quanta1}, \eqref{quanta2}, \eqref{quanta0} and Theorem \ref{quantamean} for some $C_9(n)>0$
\begin{align}\label{quantabis}
	\np{\D\n_k^{\beta}}{2,1}{\Omega_k(\alpha/2)}&\leq \np{\D\varphi_k}{2,1}{\Omega_k(\alpha^2)}+\np{\D\vec{u}_k}{2,1}{\Omega_k(\alpha^2)}+\np{\vec{v}_k}{2,1}{\Omega_k(\alpha^2)}\nonumber\\
	&\leq 16\sqrt{\pi}+C_9(n)\left(1+\Lambda\right)\left(1+\np{\D\n_k}{2}{\Omega_k(2\alpha)}\right)\np{\D\n_k}{2}{\Omega_k(2\alpha)}.
\end{align}
Therefore, the no-neck energy yields for all $1\leq \beta\leq n-2$ 
\begin{align}\label{boundedneck}
	\limsup_{\alpha\rightarrow 0}\limsup_{k\rightarrow \infty}\np{\D \n_k^{\beta}}{2,1}{\Omega_k(\alpha)}\leq 16\sqrt{\pi}.
\end{align}
Now, as 
\begin{align}\label{algebraic3}
	|\D\n_k|&=\left|\sum_{\beta=1}^{n-2}\n_{k}\wedge\cdots\wedge \D\n_{k}^{\beta}\wedge\cdots\wedge \n_k^{n-2}\right|\leq \sum_{\beta=1}^{n-2}|\D\n_k^{\beta}|,
\end{align}
we deduce from \eqref{boundedneck} that 
\begin{align}\label{boundedneck2}
	\limsup_{\alpha\rightarrow 0}\limsup_{k\rightarrow \infty}\np{\D\n_{k}}{2,1}{\Omega_k(\alpha)}\leq 16\sqrt{\pi}(n-2)<\infty
\end{align}
Now, define $\bar{\n}_{k}:B(0,\alpha R_k)\setminus \bar{B}_{\alpha^{-1}r_k}(0)$ such that for all $z\in \Omega_k(\alpha)$ such that $|z|=r$
\begin{align*}
	\bar{\n}_{k}(z)=\dashint{\partial B_r(0)}\n_k\,d\mathscr{H}^1.
\end{align*}
We will prove that for certain universal constants $C_{17}(n)$, $\Gamma_2(n)$
\begin{align}\label{lastquanta}
	\np{\D\bar{\n}_{k}}{2,1}{\Omega_k(\alpha)}\leq C_{17}(n)e^{\Gamma_2(n)\Lambda}\left(1+\np{\D\n_k}{2}{\Omega_k(2\alpha)}\right)\np{\D\n_k}{2}{\Omega_k(2\alpha)},
\end{align}
and this will finish the proof of the Theorem by using Lemmas \ref{newl2estimate} and \ref{newl1}. Indeed, notice that the following Lemma imply by \eqref{quanta1} and \eqref{quanta2} that 
\begin{align}\label{quanta10}
	&\np{\D\bar{\vec{u}_k}}{2,1}{B(0,\alpha R_k)}\leq C_{10}(n)\np{\D\n_k}{2}{\Omega_k(\alpha)}^2\nonumber\\
	&\np{\D\bar{\vec{v}_k}}{2,1}{B(0,\alpha R_k)}\leq C_{10}(n)\np{e^{\lambda_k}\H_k}{2,1}{\Omega_k(\alpha)}.
\end{align}
\begin{lemme}\label{lpmean}
	Let $n\geq 2$, $0<r<R<\infty$, $\Omega=B_R\setminus \bar{B}_r(0)\subset \R^n$, $1\leq p<\infty$ and assume that $u\in W^{1,p}(B_R\setminus \bar{B}_r(0))$. Define $\bar{u}:\Omega\rightarrow \R$ to be the radial function such that for all $r<t<R$ if $t=|x|$, then
	\begin{align*}
	    \bar{u}(x)=u_t=\dashint{\partial B_t(0)}u\,d\mathscr{H}^{n-1}=\frac{1}{\beta(n)t^{n-1}}\int_{\partial B_t(0)}u\,d\mathscr{H}^{n-1}.
	\end{align*}
	Then $\bar{u}\in W^{1,p}(\Omega)$ and 
	\begin{align*}
		\np{\D \bar{u}}{p}{\Omega}\leq \np{\D u}{p}{\Omega}.
	\end{align*}
	Furthermore, for all $1<p<\infty$, and $1\leq q\leq \infty$, there exists a constant $C(p,q)$ independent of $0<r<R<\infty$ such that for all $u\in W^{1,(p,q)}(\Omega)$, $\bar{u}\in W^{1,(p,q)}(\Omega)$ and
	\begin{align*}
		\np{\D\bar{u}}{p,q}{\Omega}\leq C(p,q)\np{\D u}{p,q}{\Omega}.
	\end{align*}
\end{lemme}
\begin{proof}
	First, assume that $u\in W^{1,p}(\Omega)$ for some $1\leq p<\infty$.
	Recall that by the proof of Proposition \ref{conf}, we have
	\begin{align}\label{levelset}
		\left|\frac{d}{dt}u_t\right|\leq \dashint{\partial B_t(0)}|\D u|\,\mathscr{H}^{n-1}.
	\end{align}
	Therefore, as $\bar{u}$ is radial, we have by the co-area formula
	\begin{align}\label{mean1}
		\np{\D \bar{u}}{p}{\Omega}^p=\beta(n)\int_{r}^{R}\left|\frac{d}{dt}u_t\right|^pt^{n-1}dt.
	\end{align}
	Furthermore, by H\"{o}lder's inequality and \eqref{levelset} 
	\begin{align}\label{mean2}
		\left|\frac{d}{dt}u_t\right|^p&\leq \frac{1}{\left(\beta(n)t^{n-1}\right)}\left|\int_{\partial B_t(0)}|\D u|d\mathscr{H}^{n-1}\right|^p\leq \frac{1}{(\beta(n)t^{n-1})^p}\int_{\partial B_t(0)}|\D u|^p\,d\mathscr{H}^{n-1}\left(\beta(n)t^{n-1}\right)^{\frac{p}{p'}}\\
		&=\frac{1}{\beta(n)t^{n-1}}\int_{\partial B_t(0)}|\D u|^p\,d\mathscr{H}^{n-1}.
	\end{align}
	Putting together \eqref{mean1} and \eqref{mean2}, we find by a new application of the co-area formula
	\begin{align*}
		\np{\D\bar{u}}{p}{\Omega}^p\leq \int_{r}^R\left(\int_{\partial B_t(0)}|\D u|^pd\mathscr{H}^{n-1}\right)dt=\int_{B_R\setminus\bar{B}_r(0)}|\D u|^pd\leb^n=\np{\D u}{p}{\Omega}^p.
	\end{align*}
	The last statement comes from the Stein-Weiss interpolation theorem (\cite{helein}, $3.3.3$).
\end{proof}
Now, in order to obtain \eqref{lastquanta}, recall the algebraic equation on $\Omega_k(\alpha)$ from \eqref{algebraic0}
\begin{align*}
	\D\n_{k}^{\beta}=-\ast\left(\n_k\wedge \D^{\perp}\n_k^{\beta}\right)+\sum_{\gamma=1}^{n-2}\s{\D\n_k^{\beta}}{\n_k^{\gamma}}-2\,H_k^{\beta}\D\phi_k.
\end{align*}
To simplify notations, let 
\begin{align*}
	\vec{G}_k=\sum_{\beta=1}^{n-2}\s{\D\n_k^{\beta}}{\n_k^{\gamma}}-2\,H_k^{\beta}\D\phi_k.
\end{align*}
Then \eqref{improvedcontrol} implies that 
\begin{align}\label{quanta4}
	\np{\vec{G}_k}{2,1}{\Omega_k(\alpha)}\leq C_{10}(n)\np{\D\n_k}{2}{\Omega_k(\alpha)}^2+4\np{e^{\lambda_k}\H_k}{2,1}{\Omega_k(\alpha)}.
\end{align}
We have
\begin{align}\label{quanta5}
	\left|\frac{d}{dt}{\n}_{k,t}^{\beta}\right|\leq \left|\dashint{\partial B_t(0)}\n_k\wedge \partial_{\tau}\n_k^{\beta}\,d\mathscr{H}^1\right|+\left|\frac{d}{dt}{\vec{G}}_{k,t}\right|=
	\left|\dashint{\partial B_t(0)}\left(\n_k-\bar{\n}_{k,t}\right)\wedge \partial_{\tau}\n_k^{\beta}\,d\mathscr{H}^1\right|+\left|\frac{d}{dt}{\vec{G}}_{k,t}\right|
\end{align}
Furthermore, by \eqref{quanta4} and Lemma \ref{lpmean}, we have (as $\bar{\vec{G}}_k$ is radial)
\begin{align}\label{quanta6}
	\np{\frac{d}{dt}{\vec{G}}_{k,t}}{2,1}{\Omega_k(\alpha)}=\np{\D\bar{\vec{G}}_k}{2,1}{\Omega_k(\alpha)}\leq C_{11}(n)\left(\np{\D\n_k}{2}{\Omega_k(\alpha)}^2+\np{e^{\lambda_k}\H_k}{2,1}{\Omega_k(\alpha)}\right).
\end{align}
Now, the $\epsilon$-regularity (\cite{riviere1} I.$5$) combined with the small $L^2$ norm of $\D\n_k$ in $\Omega_k(2\alpha)$ implies that there exists a universal constant $C_{12}(n)$ such that 
\begin{align*}
	\np{\D\n_k}{\infty}{\partial B_t}\leq \frac{C_{12}(n)}{t}\left(\int_{B_{2t}\setminus \bar{B}_{t/2}(0)}|\D\n_k|^2dx\right)^{\frac{1}{2}}
\end{align*}
so that 
\begin{align*}
	\np{\n_k-\bar{\n}_{k,t}}{\infty}{\partial B_t(0)}\leq \int_{\partial B_t(0)}|\D\n_k|\,d\mathscr{H}^1\leq 2\pi C_{12}(n)\np{\D\n_k}{2}{B_{2t}\setminus \bar{B}_{t/2}(0)}.
\end{align*}
Therefore, 
\begin{align}\label{quanta14}
	\left|\dashint{\partial B_t(0)}\left(\n_k-\bar{\n}_{k,t}\right)\wedge \partial_{\tau}\n_k^{\beta}\,d\mathscr{H}^1\right|\leq 2\pi C_{12}(n)\np{\D\n_k}{2}{\Omega_k(2\alpha)}\dashint{\partial B_t(0)}|\D\n_k^{\beta}|\,d\mathscr{H}^1.
\end{align}
The proof of Lemma \ref{lpmean} now implies by \eqref{quanta14} that 
\begin{align}\label{quanta7}
	\np{\left|\dashint{\partial B_t(0)} \left(\n_k-\bar{\n}_{k,t}\right)\wedge \partial_{\tau}\n_k^{\alpha}\,d\mathscr{H}^1\right|}{2,1}{\Omega_k(\alpha)}\leq C_{13}(n)\np{\D\n_k}{2}{\Omega_k(2\alpha)}\np{\D\n_k^{\beta}}{2,1}{\Omega_k(\alpha)}.
\end{align}
Finally, thanks to \eqref{quanta5}, \eqref{quanta6} and \eqref{quanta7}, we find
\begin{align}\label{quanta8}
	\np{\D\bar{\n}_{k}^{\beta}}{2,1}{\Omega_k(\alpha)}\leq C_{11}(n)\left(\np{\D\n_k}{2}{\Omega_k(\alpha)}^2+\np{e^{\lambda_k}\H_k}{2,1}{\Omega_k(\alpha)}\right)+C_{13}(n)\np{\D\n_k}{2}{\Omega_k(2\alpha)}\np{\D\n_k^{\beta}}{2,1}{\Omega_k(\alpha)}.
\end{align}
Therefore, \eqref{quantabis} and \eqref{quanta8} imply that
\begin{align}\label{quanta9}
	\np{\D\bar{\n}_k^{\beta}}{2,1}{\Omega_k(\alpha)}&\leq C_{11}(n)\left(\np{\D\n_k}{2}{\Omega_k(\alpha)}^2+\np{e^{\lambda_k}\H_k}{2,1}{\Omega_k(\alpha)}\right)\nonumber\\
	&+C_{13}(n)\left(16\sqrt{\pi}+C_9(n)\left(1+\Lambda\right)\left(1+\np{\D\n_k}{2}{\Omega_k(2\alpha)}\right)\np{\D\n_k}{2}{\Omega_k(2\alpha)}\right)\np{\D\n_k}{2}{\Omega_k(2\alpha)}\nonumber\\
	&\leq C_{14}(n)\left(1+\Lambda\right)^2e^{4\Gamma_1(n)\Lambda}\left(1+\np{\D\n_k}{2}{\Omega_k(2\alpha)}\right)\np{\D\n_k}{2}{\Omega_k(2\alpha)}.
\end{align}
Therefore, \eqref{quanta10} and \eqref{quanta9}  imply that 
\begin{align}\label{quanta11}
	\np{\D\bar{\vec{\varphi}_k}}{2,1}{\Omega_k(\alpha)}&\leq \np{\D\bar{\vec{n}_k}}{2,1}{\Omega_k(\alpha)}+\np{\D\bar{\vec{u}_k}}{2,1}{\Omega_k(\alpha)}+\np{\D\bar{\vec{v}}_k}{2}{\Omega_k(\alpha)}\nonumber\\
	&\leq C_{15}(n)e^{\Gamma_2(n)\Lambda}\left(1+\np{\D\n_k}{2}{\Omega_k(2\alpha)}\right)\np{\D\n_k}{2}{\Omega_k(2\alpha)}
\end{align}
We can now use Lemma \ref{l21l2} (or equivalently Proposition \ref{l212infty}) and Lemma \ref{lpmean} to get for all $0<\beta<1$
\begin{align}\label{quanta12}
	\np{\D\left(\vec{\varphi}_k-\bar{\vec{\varphi_k}}\right)}{2,1}{\Omega_k(\beta\alpha)}&\leq 24\beta\np{\D\left(\vec{\varphi_k}-\bar{\vec{\varphi}_k}\right)}{2}{\Omega_k(\alpha)}
	\leq 48\beta\np{\D\vec{\varphi}_k}{2}{\Omega_k(\alpha)}\nonumber\\
	&\leq 48\beta\left(\np{\D\n_k}{2}{\Omega_k(\alpha)}+\np{\D\vec{u}_k}{2}{\Omega_k(\alpha)}+\np{\D\vec{v}_k}{2}{\Omega_k(\alpha)}\right)\nonumber\\
	&\leq C_{16}(n)\beta\, e^{\Gamma_2(n)\Lambda}\left(1+\np{\D\n_k}{2}{\Omega_k(2\alpha)}\right)\np{\D\n_k}{2}{\Omega_k(2\alpha)}.
\end{align}
Therefore, taking $\beta=1/2$ in \eqref{quanta12}, we get by \eqref{quanta11} and \eqref{quanta12} show that 
\begin{align}\label{quanta13}
	\np{\D\vec{\varphi}_k}{2,1}{\Omega_k(\alpha/2)}\leq C_{17}(n)\, e^{\Gamma_2(n)\Lambda}\left(1+\np{\D\n_k}{2}{\Omega_k(2\alpha)}\right)\np{\D\n_k}{2}{\Omega_k(2\alpha)}.
\end{align}
Finally, by \eqref{quanta1}, \eqref{quanta2} and \eqref{quanta3} we obtain the expected estimate for $\vec{n}_k^{\beta}=\vec{u}_k+\vec{v}_k+\vec{\varphi}_k$ on $\Omega_k(\alpha/2)$, and for $\n_k$ by the algebraic inequality \eqref{algebraic3}.
\end{proof}

\begin{rem}
	Observe that for the mean curvature, we have the improved (because of the Sobolev embedding $W^{1,1}(\R^2)\hookrightarrow L^{2,1}(\R^2)$) no-neck energy
	\begin{align*}
		\lim\limits_{\alpha \rightarrow 0}\limsup_{k\rightarrow\infty}\np{e^{\lambda_k}\D\H_k}{1}{\Omega_k(\alpha)}=0
	\end{align*}
	but this is not completely clear if this also holds for $\D^2\n_k$. However, notice that \eqref{cmls2} implies that 
	\begin{align*}
		\np{\D^2\vec{A}_{\beta,\gamma}}{1}{B(0,\alpha R_k)}\leq C(n)\np{\D\n}{2}{\Omega_k(\alpha)}^2
	\end{align*}
	and as $\D^{\perp}\vec{A}_{\beta,\gamma}=\D\tilde{\n}_k^{\beta}\cdot \tilde{\n}_k^{\gamma}$, we deduce that 
	\begin{align*}
		\D^2\tilde{\n}_k^{\beta}\cdot\tilde{\n}_k^{\gamma}+\D\tilde{\n}_{k}^{\beta}\cdot \D\tilde{\n}_{k}^{\gamma}\in L^1(B(0,\alpha R_k)),
	\end{align*}
	and by the Cauchy-Schwarz inequality, this implies that for all $1\leq \beta,\gamma\leq n-2$
	\begin{align*}
		\np{\D^2\tilde{\n}_k^{\beta}\cdot\tilde{\n}_{k}^{\gamma}}{1}{B(0,\alpha R_k)}\leq C'(n)\np{\D\n_k}{2}{\Omega_k(\alpha)}^2.
	\end{align*}
	Therefore, we deduce as $\tilde{\n}_k^{\beta}=\n_k^{\beta}$ on $\Omega_k(\alpha)$ that 
	\begin{align*}
		\lim\limits_{\alpha\rightarrow 0}\limsup_{k\rightarrow \infty}\np{(\D^2{\n}_k)^N}{1}{\Omega_k(\alpha)}=0,
	\end{align*}
	but this is not completely clear how one may obtain the same result for the tangential part of $\D^2\n_k$.
\end{rem}

\section{Singularity removability of weak limits of immersions : case of one bubble}\label{sec4}

\begin{theorem}\label{onebubble}
	Let $\Sigma$ be a closed Riemann surface, $\{\phi_k\}_{k\in \N}\subset \mathrm{Imm}(\Sigma,\R^n)$ be a sequence of Willmore immersions and assume that the conformal class of $\{\phi_k\}_{k\in \N}$ stays within a compact subset of the Moduli Space and that 
	\begin{align*}
		\sup_{k\in\N} W(\phi_k)<\infty.
	\end{align*}
	Assume furthermore, that there is only one bubble $\vec{\Psi}_{\infty}:S^2\rightarrow \R^n\cup\ens{\infty}$ such  that 
	\begin{align*}
		\lim\limits_{k\rightarrow \infty}W(\phi_k)=W(\phi_{\infty})+W(\vec{\Psi}_{\infty}),
	\end{align*}
	where $\phi_{\infty}:\Sigma\rightarrow \R^n$  is the $W^{2,2}$ weak limit of $\{\phi_k\}_{k\in \N}$ in $\Sigma\setminus\ens{p}$, and $p\in \Sigma$ is the unique point of concentration of Willmore energy. Assume now that $\phi_{\infty}$ has a branch point of order $\theta_0\geq 1$ at $p$. Then there exists a conformal transformation $\vec{\Xi}:\R^n\cup\ens{\infty}\rightarrow \R^n\cup\ens{\infty}$ such that $\vec{\Xi}\circ \vec{\Psi}_{\infty}:S^2\rightarrow \R^n$ be a \emph{compact} Willmore sphere with a unique branch point of multiplicity $\theta_0 $, and if $\theta_0\geq 2$, the second residue $\alpha(\phi_{\infty},p)$ of ${\phi}_{\infty}$ satisfies
	\begin{align*}
		\alpha(\phi_{\infty},p)\leq \theta_0-2.
	\end{align*} 	
	Furthermore, the bubble $\vec{\Psi}_{\infty}$ is non-compact, has a non-compact end of multiplicity $\theta_0$ and satisfies
	\begin{align}\label{liyau1}
		\int_{S^2}K_{\vec{\Psi}_{\infty}}d\mathrm{vol}_{g_{\vec{\Psi}_{\infty}}}=-2\pi(\theta_0-1)
	\end{align}
	while
	\begin{align*}
		\int_{\Sigma}K_{\phi_{\infty}}d\mathrm{vol}_{g_{\phi_{\infty}}}=2\pi\chi(\Sigma)+2\pi(\theta_0-1)
	\end{align*}
\end{theorem}
\begin{rem}
	If $n=3$, then $\vec{\Psi}_{\infty}$ is conformally minimal by Bryant's classification (\cite{bryant}). Furthermore, as $\vec{\Psi}_{\infty}$ admits a unique non-compact end of multiplicity $\theta_0$, if $\vec{\chi}_{\infty}$ is a compactification of $\vec{\Psi}_{\infty}$, we have by the Gauss-Bonnet theorem and the Li-Yau inequality
	\begin{align}\label{liyau2}
		W(\vec{\chi}_{\infty})=4\pi \theta_0\qquad \int_{S^2}K_{\vec{\chi}_{\infty}}d\mathrm{vol}_{g_{\vec{\Psi}_{\infty}}}=4\pi+2\pi(\theta_0-1)=2\pi(\theta_0+1)
	\end{align}
	Therefore, the conformal invariance of the Willmore energy and \eqref{liyau1} and \eqref{liyau2} imply that 
	\begin{align*}
		W(\vec{\Psi}_{\infty})=W(\vec{\chi}_{\infty})-\int_{S^2}K_{\vec{\chi}_{\infty}}d\mathrm{vol}_{g_{\vec{\chi}_{\infty}}}+\int_{S^2}K_{\vec{\Psi}_{\infty}}d\mathrm{vol}_{g_{\vec{\Psi}_{\infty}}}
		=4\pi\theta_0-2\pi(\theta_0+1)+2\pi(\theta_0-1)=0.
	\end{align*}
	Therefore, $\vec{\Psi}_{\infty}$ is minimal, with a unique end of \emph{odd} multiplicity $\theta_0\geq 3$ (by the Li-Yau formula).%, and by the Jorge-Meeks formula, we deduce that $\vec{\Psi}_{\infty}$ has a unique end of multiplicity $\theta_0\geq 1$ (in particular, $\theta_0\geq 3$, and $\theta_0$ is \emph{odd}). 
\end{rem}
\begin{proof}
	Let us first prove the assertion on the Gauss curvature. Using the quantization of the total curvature, we deduce thanks of the strong convergence of the rescaled immersions (as in the end of the proof of the main Theorem in \cite{quanta}) that 
	\begin{align}\label{01}
		2\pi\chi(\Sigma)=\int_{\Sigma}K_{{\phi}_{\infty}}d\mathrm{vol}_{g_{\phi_{\infty}}}+\int_{S^2}K_{\vec{\Psi}_{\infty}}d\mathrm{vol}_{g_{\vec{\Psi}_{\infty}}}.
	\end{align}
	Furthermore, as $\phi_{\infty}$ has a unique branched branched point of order $\theta_0-1$, we have by the generalised Gauss-Bonnet theorem
	\begin{align}\label{02}
		\int_{\Sigma}K_{\phi_{\infty}}d\mathrm{vol}_{g_{\phi_{\infty}}}=2\pi\chi(\Sigma)+2\pi(\theta_0-1)=2\pi\,\chi(\Sigma)+2\pi(\theta_0-1).
	\end{align}
	Therefore, \eqref{01} and \eqref{02} imply \eqref{liyau1}.	
	\textbf{Step 0.} Thanks to the analysis of \cite{beriviere}, there exists a chart $z:U\rightarrow D^2\subset\C$ around $p$ such that $z(p)=0$ and $\vec{C}_0\in \C^n$ such that
	\begin{align}\label{expansion}
		\H_{{\phi_{\infty}}}=\Re\left(\frac{\vec{C}_0}{z^{\theta_0-1}}\right)+O(|z|^{2-\theta_0}\log^2|z|).
	\end{align} 
	Here $\alpha(\phi_{\infty},p)\leq \theta_0-2$ if and only $\vec{C}_0=0$ (by definition of the second residue, see \cite{beriviere}), and this is the result that we will prove in this Theorem.	As we can restrict to any small neighbourhood of $p$ for the analysis, all estimates will be taken with respect to a fixed chart for which the expansion \eqref{expansion} holds.
	
	\textbf{Step 1. Estimate of the metric in necks.}
	
	First, define $\Omega_k(\alpha)=B_{\alpha}\setminus \bar{B}_{\alpha^{-1}\rho_k}(0)$ and recall that by Theorem \ref{neckfine}, we have (applying the inequality on $\Omega_{\alpha}$) for all $\alpha^{-1} \rho_k<\rho<\alpha$
	\begin{align*}
	\left|d_k-\frac{1}{2\pi}\int_{\partial B_{\rho}}\partial_{\nu}\lambda\,d\mathscr{H}^1\right|\leq \Gamma_0\left(\int_{B_{\max\ens{\rho,2\alpha^{-1}\rho_k}}\setminus B_{\alpha^{-1}\rho_k}(0)}|\D\n|^2dx+\frac{1}{\log\left(\dfrac{\alpha^2}{\rho_k}\right)}\int_{\Omega_k(\alpha)}|\D\n|^2dx\right).
	\end{align*}
	Now, taking $\rho=\alpha^2$, we get
	\begin{align*}
	\left|d_k-\frac{1}{2\pi}\int_{\partial B_{\alpha^2}}\partial_{\nu}\lambda_k\,d\mathscr{H}^1\right|\leq \Gamma_0\left(\int_{B_{\alpha^2}\setminus B_{\alpha^{-1}\rho_k}}|\D\n_k|^2dx+\frac{1}{\log\left(\frac{\alpha^2}{\rho_k}\right)}\int_{\Omega_k(\alpha)}|\D\n_k|^2dx\right).
	\end{align*}
	Therefore, the no-neck energy (see \cite{quanta})
	\begin{align*}
	\lim\limits_{\alpha\rightarrow 0}\limsup_{k\rightarrow \infty}\int_{\Omega_k(\alpha)}|\D\n_k|^2dx=0
	\end{align*}
	implies that
	\begin{align*}
	\lim\limits_{\alpha\rightarrow 0}\limsup_{k\rightarrow \infty}\left|d_k-\frac{1}{2\pi}\int_{\partial B_{\alpha^2}}\partial_{\nu}\lambda_k\,d\mathscr{H}^1\right|=0.
	\end{align*}
	Furthermore, as $\phi_{\infty}$ has a branch point of order $\theta_0-1\geq 0$ at $z=0$, we have the expansion for some $\beta\in \R$
	\begin{align*}
	\lambda_{\infty}(z)=(\theta_0-1)\log|z|+\beta+O(|z|)
	\end{align*}
	we have by the strong convergence
	\begin{align*}
	\frac{1}{2\pi}\int_{\partial B_{\alpha^2}}\partial_{\nu}\lambda_k\,d\mathscr{H}^1\conv{k\rightarrow \infty}\frac{1}{2\pi}\int_{\partial  B_{\alpha^2}}\partial_{\nu}\lambda_{\infty}\,d\mathscr{H}^1=\theta_0-1+O(\alpha^2)
	\end{align*}
	Finally, this implies that
	\begin{align}\label{limitdk}
	\lim\limits_{\alpha\rightarrow 0}\limsup_{k\rightarrow \infty
	}|d_k-(\theta_0-1)|=0.
	\end{align}
	Now, recalling that $d_k$ is \emph{independent} of $\alpha>0$ (as it corresponds to the coefficient in front of the logarithm of the associated harmonic function $\nu_k$ on $B_1\setminus\bar{B}_{\alpha^{-1}\rho_k}(0)$), we deduce that  \eqref{limitdk} implies that
	\begin{align}\label{necklimit}
		d_k\conv{k\rightarrow \infty}\theta_0-1
	\end{align}
	
	\textbf{Step 2. Taylor expansion of the second fundamental form on the interior boundary of the neck region.}
	
	\begin{lemme}[Bernard-Rivi\`{e}re \cite{beriviere}]\label{expl}
		Let $\phi:D^2\rightarrow \R^n$ be a \emph{true} Willmore disk with a unique branch point at $z=0$ of multiplicity $\theta_0\geq 1$. Then there exists $\vec{L}\in C^{\infty}(D^2\setminus\ens{0},\R^n)$ such that 
		\begin{align*}
		d\vec{L}=\Im\left(\partial\H+|\H|^2\partial\phi+2\,g^{-1}\otimes\s{\H}{\h_0}\otimes \bar{\partial}\phi\right)
		\end{align*}
		and $\vec{C}_0\in \C^n$ such that for all $\epsilon>0$
		\begin{align*}
		&\H=\Re\left(\frac{\vec{C}_0}{z^{\theta_0-1}}\right)+O(|z|^{2-\theta_0-\epsilon}),\qquad
		\vec{L}=\frac{1}{2}\Im\left(\frac{\vec{C}_0}{z^{\theta_0-1}}\right)+O(|z|^{2-\theta_0-\epsilon}).
		\end{align*}
	\end{lemme}
    Thanks to the analysis of \cite{quanta}, if
    \begin{align*}
    	e^{\bar{\lambda_k}}=\dashint{B(0,\rho_k)}e^{\lambda_k(z)}d\leb^2(z),
    \end{align*}
    the sequence of Willmore disks
    \begin{align*}
    \vec{\Psi}_k&:B(0,\alpha^{-1})\rightarrow \R^n\\
    &z\mapsto e^{-\bar{\lambda_k}}\left(\phi_k(\rho_kz)-\phi_k(0)\right)
    \end{align*}
    converges in $C^{l}_{\loc}(\C)$ to a Willmore plane $\vec{\Psi}_{\infty}:\C\rightarrow \R^n$. Furthermore, we have
    \begin{align}\label{invariance}
    \int_{B(0,\alpha^{-1}\rho_k)}|\H_{\phi_k}|^2d\mathrm{vol}_{g_{\phi_k}}=\int_{B(0,\alpha^{-1})}|\H_{\vec{\Psi}_k}|^2d\mathrm{vol}_{g_{\vec{\Psi}_k}}\conv{k\rightarrow \infty}\int_{B(0,\alpha^{-1})}|\H_{\vec{\Psi}_{\infty}}|^2d\mathrm{vol}_{g_{\vec{\Psi}_{\infty}}}
    \end{align}
    Now, by composing $\vec{\Psi}_{\infty}:\C\rightarrow \R^n$ with a stereographic projection $\C\rightarrow S^2$, we extend it as a map $S^2\rightarrow \R^n$ (which may have a branch point at $\infty\in S^2=\C\cup\ens{\infty}$), and thanks to \eqref{invariance} we get
    \begin{align*}
    W(\vec{\Psi}_{\infty})=\limsup_{\alpha\rightarrow 0}\int_{B(0,\alpha^{-1})}|\H_{\vec{\Psi}_{\infty}}|^2d\mathrm{vol}_{g_{\vec{\Psi}_{\infty}}}\leq \sup_{k\in \N}W(\phi_k)<\infty.
    \end{align*}
    In particular, $\vec{\Psi}_{\infty}:S^2\rightarrow \R^n$ is an immersion of finite total curvature and at most one branch point. Now, recall that
    \begin{align}\label{obs1}
    \lim\limits_{k\rightarrow \infty}d_k=\theta_0-1.
    \end{align}
    Furthermore, we also have the inequality
    \begin{align}\label{obs2}
    \left|d_k-\frac{1}{2\pi}\int_{\partial B_{\alpha^{-1}\rho_k}}\partial_{\nu}\lambda_k\,d\mathscr{H}^1\right|&\leq \Gamma_0\left(\int_{B_{2\alpha^{-1}\rho_k}\setminus \bar{B}_{\alpha^{-1}\rho_k}(0)}|\D\n_k|^2+\frac{1}{\log\left(\frac{\alpha^2}{\rho_k}\right)}\int_{\Omega_k(\alpha)}|\D\n_k|^2dx\right).
    \end{align}
    Observing that 
    \begin{align*}
    \lambda_{\psi_k}(z)=\lambda_k(\rho_kz)-\bar{\lambda_k},
    \end{align*}
    we deduce that
    \begin{align}\label{obs3}
    \int_{\partial B_{\alpha^{-1}}}\partial_{\nu}\lambda_{\vec{\Psi}_k}\,d\mathscr{H}^1=\int_{\partial B_{\alpha^{-1}\rho_k}}\partial_{\nu}\lambda_k\,d\mathscr{H}^1.
    \end{align}
    Combining \eqref{obs1}, \eqref{obs2} and \eqref{obs3}, we obtain
    \begin{align*}
    \lim\limits_{\alpha\rightarrow 0}\limsup_{k\rightarrow \infty}\left|\frac{1}{2\pi}\int_{\partial B_{\alpha^{-1}}}\partial_{\nu}\lambda_{\vec{\Psi}_k}\,d\mathscr{H}^1-(\theta_0-1)\right|=0.
    \end{align*}
    Now, let $\iota:\C\setminus\ens{0}\rightarrow \C\setminus\ens{0}$ be the inversion centred at $0$. Then one directly checks that
    \begin{align*}
    \int_{\partial B_{\alpha}}\partial_{\nu}\lambda_{\vec{\Psi}_k\circ \iota}\,d\mathscr{H}^1=-\left(\int_{\partial B_{\alpha^{-1}}}\partial_{\nu}\lambda_{\vec{\Psi}_k}\,d\mathscr{H}^1+2\right).
    \end{align*}
    Therefore, we finally get
    \begin{align}\label{noncompactend}
    \lim\limits_{\alpha\rightarrow 0}\limsup_{k\rightarrow \infty}\left|\frac{1}{2\pi}\int_{\partial B_{\alpha}}\partial_{\nu}\lambda_{ \vec{\Psi}_k\circ \iota}\,d\mathscr{H}^1+(\theta_0+1)\right|=0.
    \end{align}
    Thanks to the strong convergence of $\vec{\Psi}_k\circ \iota\rightarrow \vec{\Psi}_{\infty}\circ \iota$ in $C^l_{\mathrm{loc}}(\C\setminus\ens{0})$ for all $l\in \N$, we have
    \begin{align}\label{eqinv}
    \int_{\partial B_{\alpha}}\partial_{\nu}\lambda_{ \vec{\Psi}_k\circ \iota}\,d\mathscr{H}^1\conv{k\rightarrow \infty}\int_{\partial B_{\alpha}}\partial_{\nu}\lambda_{\vec{\Psi}_{\infty}\circ \iota}\,d\mathscr{H}^1.
    \end{align}
    Therefore, $\vec{\Psi}_{\infty}\circ \iota:\C\setminus\ens{0}\rightarrow \R^n$ has a non-compact end at zero of multiplicity $\theta_0$. Indeed, \eqref{noncompactend} implies that for some $\gamma\in \R$
    \begin{align*}
    e^{\lambda_{\vec{\Psi}_{\infty}\circ \iota}}=\frac{e^{\gamma}}{|z|^{\theta_0+1}}\left(1+o(1)\right)
    \end{align*}
    and the analysis of \cite{beriviere} implies that there exists $\vec{D}_0\in \C^n$ such that 
    \begin{align}\label{metric}
    &e^{\lambda_{\vec{\Psi}_{\infty}\circ \iota}}=\frac{e^{\gamma}}{|z|^{\theta_0+1}}\left(1+O(|z|)\right)\nonumber\\
    &\H_{\vec{\Psi}_{\infty}\circ \iota}=\Re\left(\vec{D}_0z^{\theta_0+1}\right)+O(|z|^{\theta_0+1}\log^2|z|).
    \end{align}
    Notice that \eqref{metric} implies that we have the expansion as $|z|\rightarrow 0$
    \begin{align*}
    g_{\vec{\Psi}_{\infty}\circ \iota}=\frac{e^{\gamma}}{|z|^{\theta_0+1}}\left(1+O(|z|)\right)|dz|^2
    \end{align*}
    so taking $w=\dfrac{1}{z}$, we have the expansion as $|w|\rightarrow \infty$
    \begin{align*}
    g_{\vec{\Psi}_{\infty}}={e^{\gamma}}|w|^{2\theta_0+2}\left(1+O\left(\frac{1}{|w|}\right)\right)\left|-\frac{dw}{w^2}\right|^2=e^{\gamma}|w|^{2\theta_0-2}\left(1+O\left(\frac{1}{|w|}\right)\right)|dw|^2.
    \end{align*}
    Therefore, coming back to the original definition on $\partial B_{\alpha^{-1}}$, we obtain as $|z|\rightarrow \infty$ the expansion
    \begin{align*}
    &e^{\lambda_{\vec{\Psi}_{\infty}}}=e^{\beta}|z|^{\theta_0-1}\left(1+O\left(\frac{1}{|z|}\right)\right)\\
    &\H_{\vec{\Psi}_{\infty}}=\Re\left(\frac{\vec{D}_0}{z^{\theta_0+1}}\right)+O\left(\frac{1}{|z|^{\theta_0+2}}\right).
    \end{align*}
    This implies that for all $z\in \partial B_{\alpha^{-1}}$, we have
    \begin{align}\label{hl1}
    e^{\lambda_{\vec{\Psi}_{\infty}}(z)}\H_{\vec{\Psi}_{\infty}}(z)=O(\alpha^2).
    \end{align}
    Likewise, the analysis of \cite{beriviere} (see also Lemma \ref{expl}) shows that for all $z\in \partial B_{\alpha^{-1}}$
    \begin{align}\label{hl2}
    e^{\lambda_{\vec{\Psi}_{\infty}}(z)}\vec{L}_{\vec{\Psi}_{\infty}}(z)=O(\alpha^2).
    \end{align}
    Now, by definition of $\vec{\Psi}_k$, one sees easily that for all $z\in \bar{B}(0,\alpha^{-1})$, 
    \begin{align}\label{hl3}
    e^{\lambda_{\vec{\Psi}_k}(z)}\left(\vec{H}_k(z)+2i\vec{L}_k(z)\right)=e^{\lambda_k(\rho_kz)}\left(\H_k(\rho_kz)+2i\vec{L}_k(\rho_kz)\right).
    \end{align}
    Finally, \eqref{hl1}, \eqref{hl2} and \eqref{hl3} imply that for all $z\in \partial B_{\alpha^{-1}\rho_k}$
    \begin{align}\label{inneradius}
    \limsup_{k\rightarrow \infty}\left|e^{\lambda_k(z)}\left(\H_k(z)+2i\vec{L}_k(z)\right)\right|=O(\alpha^2).
    \end{align}
    Now, thanks to Lemma \ref{expl}, as ${\phi}_{\infty}$ has a branch point of order $\theta_0\geq 1$ there exists $\gamma\in \R$ and $\vec{C}_0\in \C^n$ such that
    \begin{align}\label{devphinfty}
        &e^{2\lambda_{\phi_{\infty}}}=e^{2\gamma}|z|^{2\theta_0-2}\left(1+O(|z|)\right)\nonumber\\
    	&z^{\theta_0-1}\left(\H_{{\phi}_{\infty}}(z)+2i\vec{L}_{\phi_{\infty}}(z)\right)=\vec{C}_0+O(|z|).
    \end{align}
    Furthermore, for all $k\in \N$ and $0<\alpha<\alpha_0$, there exists by Theorem \ref{neckfine} real numbers $d_k,A_{k}(\alpha)\in \R$ (with $d_k$ independent of $\alpha$) such that 
    \begin{align*}
    	\np{\lambda_k-d_k\log|z|-A_{k}(\alpha)}{\infty}{\Omega_k(\alpha)}\leq \Gamma_0'(n)\left(\sqrt{\alpha}^{\frac{1}{4}}\Lambda+\np{\D\n_k}{2}{\Omega_k(\sqrt{\alpha})}\right).
    \end{align*}
    where
    \begin{align*}
    	\Lambda=\sup_{k\in \N}\left(1+\np{d\lambda_k}{2,\infty}{\Sigma}+W(\phi_k)\right)<\infty
    \end{align*}
    by hypothesis (as the conformal class stays within a compact set of the Moduli Space, $\np{d\lambda_k}{2,\infty}{\Sigma}$ is uniformly bounded). As $\lambda_k$ is continuous on $\Omega_k(\alpha)$ thanks to the estimate \ref{I1} and the no-neck energy, we have for all $z\in \Omega_k(\alpha)$ 
    \begin{align*}
    	e^{\lambda_k(z)}=e^{A_k(\alpha)+o_{k,\alpha}(1)}|z|^{d_k},
    \end{align*}
    where $
    	\lim_{\alpha\rightarrow 0}\limsup_{k\rightarrow\infty}|o_{k,\alpha}(1)|=0.
    $ 
    Therefore, as by \eqref{necklimit} $d_k\conv{k\rightarrow \infty}\theta_0-1$, and using the strong convergence of $\{\phi_k\}_{k\in \N}$ towards $\phi_{\infty}$ on all compact subsets of $B_{\alpha}(0)\setminus\ens{0}$, for all $z\in B_{\alpha}(0)\setminus \ens{0}$, we have
    \begin{align*}
    	\limsup\limits_{k\rightarrow \infty}e^{A_k(\alpha)}\leq \limsup_{k\rightarrow \infty}e^{|o_{k,\alpha}(1)|}\frac{e^{\lambda_k(z)}}{|z|^{d_k}}=e^{|o_{k,\alpha}(1)|}e^{\gamma}\left(1+O(|z|)\right)
    \end{align*}
    and likewise, for all $z\in B_{\alpha}(0)\setminus\ens{0}$
    \begin{align*}
    	e^{-|o_{\alpha}(1)|}e^{\gamma}(1+O(|z|))\leq \liminf_{k\rightarrow \infty}e^{A_k(\alpha)}.
    \end{align*}
    Therefore, taking $|z|\rightarrow 0$, we deduce that 
    \begin{align*}
    	\lim\limits_{\alpha\rightarrow 0}\limsup\limits_{k\rightarrow \infty}e^{A_k(\alpha)}=\lim_{\alpha\rightarrow 0}\liminf\limits_{k\rightarrow \infty}e^{A_k(\alpha)}=e^{\gamma}.
    \end{align*}
    Now, Theorem \ref{integer1} implies that there exists a constant $C>0$ independent of $k\in \N$ and $\alpha>0$ such that for all $k$ large enough and $0<\alpha<\alpha_0$ (for some fixed $\alpha_0$)
    \begin{align}\label{nondegenerate}
    	\frac{1}{C}\leq \frac{e^{\lambda_k(z)}}{|z|^{\theta_0-1}}\leq C\qquad \text{for all}\;\, z\in B_1\setminus \bar{B}_{\rho_k}(0).
    \end{align}
    Then \eqref{inneradius} implies that for all $z\in \partial B_{\alpha^{-1}\rho_k}$
    \begin{align}\label{inestimate}
    	\limsup\limits_{k\rightarrow \infty}\left|z^{\theta_0-1}\left(\H_k(z)+2i\,\vec{L}_k(z)\right)\right|=O(\alpha^2).
    \end{align}

	\textbf{Step 3. Taylor expansion of the second fundamental form on the exterior boundary of the neck regions.}
	
	Here, the strong convergence of $\{\phi_k\}_{k\in \N}$ towards $\phi_{\infty}$ in $C^l_{\mathrm{loc}}(B_{\alpha}(0)\setminus\ens{0})$ (for all $l\in \N$), the expansion \eqref{expansion} and Lemma \ref{expl} imply that for all $z\in B_{\alpha}(0)\setminus\ens{0}$ (regardless if \eqref{nondegenerate} holds or not)
	\begin{align*}
		\lim\limits_{k\rightarrow \infty} z^{\theta_0-1}\left(\H_k(z)+2i\,\vec{L}_k(z)\right)=\vec{C}_0+O(|z|).
	\end{align*}
	
	\textbf{Step 4. $L^{\infty}$ estimate of the second fundamental form in neck regions.}	
	Now remember the system on $\Omega_k(\alpha)$
	\begin{align}\label{system}
	\left\{\begin{alignedat}{1}
	&2i\,\partial \vec{L}_k=\partial\H_k+|\H_k|^2\partial\phi_k+2\,g_k^{-1}\otimes\s{\H_k}{\h^0_k}\otimes \bar{\partial}\phi_k\\
	&\D S_k=\vec{L}_k\cdot\D\phi_k\\
	&\D \vec{R}_k=\vec{L}_k\wedge \D\phi_k+2\H_k\wedge \D^{\perp}\phi_k
	\end{alignedat}\right.
	\end{align}
	Now, let us rewrite \eqref{system} as
	\begin{align*}
	\bar{\partial}\left(\H_k+2i\vec{L}_k\right)=-|\H_k|^2\bar{\partial}\phi_k-2\,g_k^{-1}\otimes \s{\H_k}{\bar{\h_k^0}}\otimes \partial\phi_k
	\end{align*}
	This implies that
	\begin{align*}
	\bar{\partial}\left(z^{\theta_0-1}\left(\H_k+2i\vec{L}_k\right)\right)=-z^{\theta_0-1}|\H_k|^2\bar{\partial}\phi_k-2z^{\theta_0-1}\,g^{-1}\otimes \s{\H_k}{\bar{\h_k^0}}\otimes \partial\phi_k
	\end{align*}
	Now, we write $z^{\theta_0-1}\left(\H_k+2i\vec{L}_k\right)=\vec{u}_k+\vec{v}_k$, where 
	\begin{align*}
		\left\{\begin{alignedat}{2}
		\bar{\partial}\vec{u}_k&=0\qquad &&\text{in}\;\, \Omega_k(\alpha)\\
		\vec{u}_k&=z^{\theta_0-1}\left(\H_k+2i\,\vec{L}_k\right)\qquad && \text{on}\;\, \partial \Omega_k(\alpha)
		\end{alignedat}\right.
	\end{align*}
	and
	\begin{align*}
	\left\{
	\begin{alignedat}{2}
	\bar{\partial}\vec{v}_k&=-z^{\theta_0-1}|\H_k|^2\bar{\partial}\phi_k-2z^{\theta_0-1}\,g^{-1}\otimes \s{\H_k}{\bar{\h^0_k}}\otimes\partial\phi_k\qquad && \text{in}\;\, \Omega_k(\alpha)\\
	\vec{v}_k&=0\qquad && \text{on}\;\, \partial \Omega_k(\alpha).
	\end{alignedat}\right.
	\end{align*}
	Now, we obtain trivially
	\begin{align}\label{inest}
	\np{\bar{\partial}\vec{v}_k}{2,1}{\Omega_k(\alpha)}&\leq \np{e^{\lambda_k}\H_k}{2,1}{\Omega_k(\alpha)}\np{z^{\theta_0-1}\H_k}{\infty}{\Omega_k(\alpha)}+2\np{e^{\lambda_k}\H^0_k}{2,1}{\Omega_k(\alpha)}\np{z^{\theta_0-1}\H_k}{\infty}{\Omega_k(\alpha)}\nonumber\\
	&\leq \np{\D\n_k}{2,1}{\Omega_k(\alpha)}\np{z^{\theta_0-1}\H_k}{\infty}{\Omega_k(\alpha)}.
	\end{align}
	Now, as $\vec{v}_k=0$ on $\partial \Omega_k(\alpha)$, we have for all $z_0\in \Omega_k(\alpha)$
	\begin{align}\label{nulltrace}
		\vec{v}_k(z_0)=-\frac{1}{\pi}\int_{\Omega_k(\alpha)}\frac{\p{\z}\vec{v}_k(z)}{z-z_0}|dz|^2.
	\end{align}
	Therefore, we have by the duality $L^{2,1}/L^{2,\infty}$ and  the inequalities \eqref{inest}, \eqref{nulltrace}
	\begin{align*}
		|\vec{v_k}(z_0)|&\leq \frac{1}{2\pi}\np{\bar{\partial}\vec{v}_k}{2,1}{\Omega_k(\alpha)}\np{\frac{1}{z-z_0}}{2,\infty}{\Omega_k(\alpha)}%\leq \frac{1}{\sqrt{\pi}}\np{\bar{\partial}\vec{v}_k}{2,1}{\Omega_k(\alpha)}
		\leq \frac{1}{\sqrt{\pi}}\np{\D\n_k}{2,1}{\Omega_k(\alpha)}\np{z^{\theta_0-1}\H_k}{\infty}{\Omega_k(\alpha)}.
	\end{align*}
	This implies as $\vec{L}_k$ is real that
	\begin{align}\label{estneck1}
		\np{\vec{v}_k}{2,1}{\Omega_k(\alpha)}&\leq \frac{1}{\sqrt{\pi}}\np{\D\n_k}{2,1}{\Omega_k(\alpha)}\np{z^{\theta_0-1}\H_k}{\infty}{\Omega_k(\alpha)}\nonumber\\
		&\leq \frac{1}{\sqrt{\pi}}\np{\D\n_k}{2,1}{\Omega_k(\alpha)}\np{z^{\theta_0-1}\left(\H_k+2i\vec{L}_k\right)}{\infty}{\Omega_k(\alpha)}.
	\end{align}
	Furthermore, by the maximum principle, we have
	\begin{align}\label{estneck2}
		\np{\vec{u}_k}{\infty}{\Omega_k(\alpha)}\leq \np{z^{\theta_0-1}\left(\H_k+2i\,\vec{L}_k\right)}{\infty}{\partial\Omega_k(\alpha)}.
	\end{align}
	Therefore, recalling that $z^{\theta_0-1}\left(\H_k+2i\,\vec{L}_k\right)=\vec{u}_k+\vec{v}_k$ and by \eqref{estneck1} and \eqref{estneck2}, we find
	\begin{align}\label{linfty}
		&\np{z^{\theta_0-1}\left(\H_k+2i\,\vec{L}_k\right)}{\infty}{\Omega_k(\alpha)}\leq \np{\vec{v}_k}{\infty}{\Omega_k(\alpha)}+\np{u_k}{\infty}{\Omega_k(\alpha)}\nonumber\\
		&\leq \frac{1}{\sqrt{\pi}}\np{\D\n_k}{2,1}{\Omega_k(\alpha)}\np{z^{\theta_0-1}\left(\H_k+2i\,\vec{L}_k\right)}{\infty}{\Omega_k(\alpha)}+\np{z^{\theta_0-1}\left(\H_k+2i\,\vec{L}_k\right)}{\infty}{\partial\Omega_k(\alpha)}.
	\end{align}
	Thanks to Theorem \ref{improvedquanta}, we have
	\begin{align}\label{strong-no-neck}
	\lim\limits_{\alpha\rightarrow 0}\limsup_{k\rightarrow \infty}\np{\D\n_k}{2,1}{\Omega_k(\alpha)}\conv{k\rightarrow \infty}0.
	\end{align}
	Therefore, if $\alpha_0>0$ and $N\in \N$ are such that 
	\begin{align*}
	\forall\, 0<\alpha<\alpha_0,\,\;\forall\, k\geq N,\quad \np{\D\n_k}{2,1}{\Omega_k(\alpha)}<\frac{\sqrt{\pi}}{2}
	\end{align*}
	we find by \textbf{Step 2} and \textbf{Step 3} that there exists a constant $\Gamma_8>0$ (depending only on $\phi_{\infty}$ on not on $k$ and $\alpha>0$ such that)
	\begin{align*}
	\np{z^{\theta_0-1}\left(\H_k+2i\vec{L}_k\right)}{\infty}{\Omega_k(\alpha)}\leq 2\np{z^{\theta_0-1}\left(\H_k+2i\vec{L}_k\right)}{\infty}{\partial \Omega_k(\alpha)}\leq \Gamma_8
	\end{align*}
	for some constant $C_0>0$ depending only on $\phi_{\infty}$ and not on $k$ and $\alpha>0$. Returning to \eqref{estneck1}, we find
	\begin{align}\label{endest1}
	\np{\vec{v}_k}{\infty}{\Omega_k(\alpha)}\leq \frac{\Gamma_8}{\sqrt{\pi}}\np{\D\n_k}{2,1}{\Omega_k(\alpha)}.
	\end{align}
	Now, to be able to show that $\vec{C}_0=0$, we will need to refine these estimates to obtain a pointwise control of $z^{\theta_0-1}\left(\H_k+2i\,\vec{L}_k\right)$ in the neck region.
	
	\textbf{Step 5. Schwarz lemma for degenerating annuli of conformal class bounded away from $-\infty$.}
	
	The following lemma shows the intuitive fact that Schwarz lemma holds approximately for holomorphic functions in annuli of small inner radius.
	
	\begin{lemme}\label{schwarz}
		Let $0<4r<R<\infty$, let $\vec{u}:\Omega=B_R\setminus B_r(0)\rightarrow\C^m$ be a vector-valued holomorphic function and  let $\delta\geq 0$ be such that
		\begin{align*}
		&\np{\vec{u}}{\infty}{\partial B_r}\leq \delta.
		\end{align*}
		Then for all $1\leq j\leq m$, we have
		\begin{align*}
		|u_j(z)|\leq \frac{5}{R}\left(\np{u_j}{\infty}{\partial B_R}+\delta\right)|z|+2\delta.
		\end{align*}
	\end{lemme}
	\begin{proof}
		By considering each component of $\vec{u}=(u_1,\cdots,u_m)$, it suffices to check the case $m=1$. We then let $u:B_R\setminus B_r(0)\rightarrow \C$ be a holomorphic function. Then there exists $\ens{a_n}_{n\in \Z}\subset \C$ such that
		\begin{align*}
		u(z)=\sum_{n\in \Z}^{}a_nz^n.
		\end{align*} 
		Now, write $u=u_1+u_2$, where
		\begin{align*}
		u_1(z)=\sum_{n\geq 1}^{}a_nz^n\quad u_2(z)=\sum_{n\leq 0}^{}a_nz^n.
		\end{align*}
		Then $u_1$ extends to $B_R(0)$ as a holomorphic function (still denoted by $u_1$) such that $u_1(0)=0$. Therefore, thanks to Schwarz's lemma, we deduce that
		\begin{align}\label{classicalschwarz}
		|u_1(z)|\leq \frac{1}{R}\np{u_1}{\infty}{\partial B_R}|z|.
		\end{align}
		Now, we have for all $z\in \partial B_r(0)$ the inequality
		\begin{align*}
		\left|\sum_{n\leq 0}^{}a_nz^n\right|=|u_2(z)|=|u(z)-u_1(z)|\leq |u(z)|+|u_1(z)|\leq \frac{1}{R}\np{u_1}{\infty}{\partial B_R}|z|+\delta=\np{u_1}{\infty}{\partial B_R}\frac{r}{R}+\delta.
		\end{align*}
		In other words
		\begin{align}\label{schwarz1}
			\np{u_2}{\infty}{\partial B_r}\leq \np{u_1}{\infty}{\partial B_R}\frac{r}{R}+\delta
		\end{align}
		Now, for all $m\geq  0$, we have
		\begin{align*}
		2\pi r^{-m}|a_{-m}|&=\left|\int_{0}^{2\pi}u_2(r e^{i\theta})e^{im\theta }d\theta\right|
		\leq\int_{0}^{2\pi}|u_2(\rho e^{i\theta})|d\theta\leq 2\pi \left(\np{u_1}{\infty}{\partial B_R}\frac{r}{R}+\delta\right),
		\end{align*}
		so that for all $m\geq 0$
		\begin{align*}
		|a_{-m}|\leq r^{m}\left(\np{u_1}{\infty}{\partial B_R}\frac{r}{R}+\delta\right).
		\end{align*}
		In particular, for all $z\in B_R\setminus B_r(0)$, we have
		\begin{align*}
		|u_2(z)|&\leq \sum_{n\geq 0}^{}|a_{-n}||z|^{-n}\leq \sum_{n\geq 0}^{}\left(\frac{r}{|z|}\right)^n\left(\np{u_1}{\infty}{\partial B_R}\frac{r}{R}+\delta\right)
		=\frac{|z|}{|z|-r}\left(\np{u_1}{\infty}{\partial B_R}\frac{r}{R}+\delta\right).
		\end{align*}
		Taking $|z|=R$ yields
		\begin{align*}
		|u_2(z)|\leq \frac{R}{R-r}\left(\np{u_1}{\infty}{\partial B_R}\frac{r}{R}+\delta\right).
		\end{align*}
		Therefore, we have
		\begin{align}\label{schwarz1bis}
			\np{u_2}{\infty}{\partial B_R}\leq \frac{1}{1-\frac{r}{R}}\left(\np{u_1}{\infty}{\partial B_R}\frac{r}{R}+\delta\right).
		\end{align}
		The maximum principle and \eqref{schwarz1}, \eqref{schwarz1bis} imply that
		\begin{align*}
			\np{u_2}{\infty}{B_R\setminus \bar{B}_r(0)}\leq \np{u_2}{\infty}{\partial (B_R\setminus \bar{B}_r(0))}\leq \frac{1}{1-\frac{r}{R}}\left(\np{u_1}{\infty}{\partial B_R}\frac{r}{R}+\delta\right). 
		\end{align*}
		In particular, we have for all $z\in B_R\setminus B_r(0)$ as $r\leq |z|$
		\begin{align}\label{u2}
		|u_2(z)|&\leq \frac{1}{1-\frac{r}{R}}\left(\np{u_1}{\infty}{\partial B_R}\frac{r}{R}+\delta\right)
		\leq \frac{1}{R}\frac{1}{1-\frac{r}{R}}\np{u_1}{\infty}{\partial B_R}|z|+\frac{1}{1-\frac{r}{R}}\delta.
		\end{align}
		Finally, we get by \eqref{classicalschwarz} and \eqref{u2}
		\begin{align*}
		|u(z)|\leq \frac{1}{R}\left(1+\frac{1}{1-\frac{r}{R}}\right)\np{u_1}{\infty}{\partial B_R}|z|+\frac{1}{1-\frac{r}{R}}\delta.
		\end{align*}
		Now, we estimate thanks to the first inequality in \eqref{u2}
		\begin{align*}
		\np{u_1}{\infty}{\partial B_R}\leq \np{u}{\infty}{\partial B_R}+\np{u_2}{\infty}{\partial B_R}\leq \np{u}{\infty}{\partial B_R}+\frac{1}{1-\frac{r}{R}}\left(\np{u_1}{\infty}{\partial B_R}\frac{r}{R}+\delta\right)
		\end{align*}
		which yields as $r<\dfrac{1}{4}$ and $R>4r$
		\begin{align*}
		\np{u_1}{\infty}{\partial B_R}\leq \frac{1-\frac{r}{R}}{1-\frac{2r}{R}}\left(\np{u}{\infty}{\partial B_R}+\frac{1}{1-\frac{r}{R}}\delta\right).
		\end{align*}
		Finally, we find
		\begin{align}\label{pre}
		|u(z)|\leq \frac{1}{R}\left(1+\frac{1}{1-\frac{r}{R}}\right)\frac{\left(1-\frac{r}{R}\right)}{1-\frac{2r}{R}}\left(\np{u}{\infty}{\partial B_R}+\frac{1}{1-\frac{r}{R}}\delta\right)|z|+\frac{1}{1-\frac{r}{R}}\delta.
		\end{align}
		which concludes the proof of the Lemma. Indeed, as $0<4r<\min\ens{1,R}<\infty$, we deduce that 
		\begin{align}\label{pre2}
		%&\left(1+\frac{1}{1-\frac{r}{R}}\right)\frac{(1-\frac{r}{R})}{1-\frac{2r}{R}}<\left(1+\frac{4}{3}\right)\times 2=\frac{14}{3}\nonumber\\
		&\left(1+\frac{1}{1-\frac{r}{R}}\right)\frac{(1-\frac{r}{R})}{1-\frac{2r}{R}}\times \left(\frac{1}{1-\frac{r}{R}}\right)=\left(1+\frac{1}{1-\frac{r}{R}}\right)\frac{1}{1-\frac{2r}{R}}
		<\left(1+\frac{4}{3}\right)\times 2=\frac{14}{3}.
		\end{align}
		Now, by \eqref{pre} and \eqref{pre2}, we find
		\begin{align*}
			|u(z)|\leq \frac{14}{3 R}\left(\np{u}{\infty}{\partial B_R}+\delta\right)|z|+\frac{4}{3}\delta.
		\end{align*}
		This concludes the proof of the Lemma.
	\end{proof}
	
	\textbf{Step 6. Conclusion.}
Thanks to \textbf{Step 2} and \textbf{Step 3}, we have the expansions for $k$ large enough
\begin{align}\label{sch1}
	|z^{\theta_0-1}(\H_k+2i\vec{L}_k)|\leq C\alpha^2+o_k(1)\quad \text{on}\;\, \partial B_{\alpha^{-1}\rho_k}
\end{align}
while
\begin{align}\label{sch2}
	z^{\theta_0-1}(\H_k+2i\vec{L}_k)=\vec{C}_0+O(\alpha)+o_k(1)\quad \text{on}\;\, \partial B_{\alpha}.
\end{align}
Here $o_k(1)$ means that $\limsup\limits_{k\rightarrow \infty}|o_k(1)|=0$.
Therefore, as $\bar{\partial}\vec{u}_k=0$ and $\vec{u}_k=z^{\theta_0-1}(\H_k+2i\vec{L}_k)$ on $\partial \Omega_k(\alpha)$, we obtain by \eqref{sch1}, \eqref{sch2} and Lemma \ref{schwarz} (with $\delta=C\alpha^2+o_k(1)$)
\begin{align}\label{endest2}
	|\vec{u}_k(z)|&\leq 5n\frac{1}{\alpha}\left(\np{\vec{u}_k}{\infty}{\partial B_R}+C\alpha^2+o_k(1)\right)|z|+2n\left(C\alpha^2+o_k(1)\right)\nonumber\\
	&\leq 5n\frac{1}{\alpha}\left(|\vec{C}_0|+C'\alpha+o_k(1)\right)|z|+2n\left(C\alpha^2+o_k(1)\right)
\end{align}
Now, by \eqref{endest1} and \eqref{endest2} as                                $z^{\theta_0-1}\left(\H_k+2i\,\vec{L}_k\right)=\vec{u}_k+\vec{v}_k$ there exists a universal constant $C''>0$ such that for all $z\in \Omega_k(\alpha)=B_{\alpha}\setminus B_{\alpha^{-1}\rho_k}(0)$ 
\begin{align*}
	\left|z^{\theta_0-1}\left(\H_k(z)+2i\vec{L}_k(z)\right)\right|\leq C''\frac{1}{\alpha}\left(|\vec{C}_0|+\alpha+o_k(1)\right)|z|+C''\left(\alpha^2+o_k(1)\right)+C''\np{\D\n_k}{2,1}{\Omega_k(\alpha)}.
\end{align*}
Now, thanks to the strong convergence, we have for all $z\in \partial B_{\alpha^2}$ that
\begin{align*}
	z^{\theta_0-1}\left(\H_k(z)+2i\vec{L}_k(z)\right)\conv{k\rightarrow\infty}\vec{C}_0+O(\alpha^2)
\end{align*}
which implies (taking $|z|=\alpha^2$)
\begin{align*}
	|\vec{C}_0|+O(\alpha^2)+o_k(1)&\leq C''\frac{1}{\alpha}\left(|\vec{C}_0|+\alpha+o_k(1)\right)\alpha^2+C''\left(\alpha^2+o_k(1)\right)+C''\np{\D\n_k}{2,1}{\Omega_k(\alpha)}\\
	&=C''\alpha\left(|\vec{C}_0|+\alpha+o_k(1)\right)+C''\left(\alpha^2+o_k(1)\right)+C''\np{\D\n_k}{2,1}{\Omega_k(\alpha)}
\end{align*}
Therefore, we find
\begin{align*}
	|\vec{C}_0|\leq C'''\alpha+C''\limsup_{k\rightarrow \infty}\np{\D\n_k}{2,1}{\Omega_k(\alpha)},
 \end{align*}
where $C'''>0$ is independent of $0<\alpha<1$. Therefore, thanks to \eqref{strong-no-neck} and taking $\alpha\rightarrow 0$, we find $\vec{C}_0=0$ and this concludes the proof of the Theorem.
\end{proof}

In the next proposition, used in the proof of the previous Theorem, we obtain an integrality result for the multiplicity of a sequence of weak immersions from annuli converging strongly outside of the origin.
\begin{theorem}\label{integer1}
	Let $\{\phi_k\}_{k\in \N}$ be a sequence of smooth conformal immersions from the disk $B(0,1)\subset \C$ into $\R^n$, letlet
	\begin{align*}
	e^{\lambda_k}=\frac{1}{\sqrt{2}}|\D\phi_k|
	\end{align*}
	be the conformal factor of $\phi_k$, and $\ens{\rho_k}_{k\in \N}\subset (0,1)$ such that $\rho_k\conv{k\rightarrow \infty}0$, $\Omega_k=B(0,1)\setminus \bar{B}(0,\rho_k)$ and assume that
	\begin{align*}
		\sup_{k\in \N}\int_{B(0,1)\setminus \bar{B}_{\rho_k}(0)}|\D\n_k|^2dx\leq \epsilon_1(n),\qquad \sup_{k\in \N}\np{\D\lambda_k}{2,\infty}{\Omega_k}<\infty
	\end{align*} 
	where $\epsilon_1(n)$ is given by the proof of Theorem \ref{neckfine}. Define for all $0<\alpha<1$ and $k\in\N$ large enough $\Omega_k(\alpha)=B_{\alpha}\setminus \bar{B}_{\alpha^{-1}\rho_k}(0)$, and assume that 
	\begin{align*}
		\lim_{\alpha\rightarrow 0}\limsup_{k\rightarrow \infty}\int_{\Omega_k(\alpha)}|\D\n_k|^2dx=0
	\end{align*}
	and that there exists a $W^{2,2}_{\mathrm{loc}}(B(0,1)\setminus\ens{0})\cap  C^{\infty}(B(0,1)\setminus\ens{0})$ immersion $\phi_{\infty}$ such that
	\begin{align*}
		\log|\D\phi_{\infty}|\in L^{\infty}_{\mathrm{loc}}(B(0,1)\setminus\ens{0})
	\end{align*} 
	such that $\phi_k\conv{k\rightarrow  \infty}\phi_{\infty}$ in $C^l_{\mathrm{loc}}(B(0,1)\setminus\ens{0})$ (for all $l\in \N$). Then there exists an integer $\theta_0\geq 1$,  $\mu_k\in W^{1,(2,1)}(B(0,1))$ such that 
	\begin{align*}
		\np{\D\mu_k}{2,1}{B(0,1)}\leq \frac{1}{2}\Gamma_2(n)C_1(n)\int_{\Omega_k}|\D\n_k|^2dx
	\end{align*}
	and a harmonic function $\nu_k$ on $\Omega_k$ such that $\nu_k=\lambda_k$ on $\partial B(0,1)$, $\lambda_k=\mu_k+\nu_k$ on $\Omega_k$ and such that for all $0<\alpha<1$ and such that for all $k\in \N$ sufficiently large 
	\begin{align*}
		\np{\D(\nu_k-(\theta_0-1)\log|z|)}{2,1}{\Omega_k(\alpha)}\leq \Gamma_{10}\left(\sqrt{\alpha}\np{\D\lambda_k}{2,\infty}{\Omega_k}+\int_{\Omega_k}|\D\n_k|^2dx\right)
	\end{align*}
	for some universal constant $\Gamma_{10}=\Gamma_{10}(n)$. Furthermore, we have for all $\rho_k\leq r_k\leq 1$ and $k$ large enough
	\begin{align*}
		\frac{1}{2\pi}\int_{\partial B_{r _k}}\ast\, d\nu_k=\theta_0-1.
	\end{align*}
\end{theorem}
\begin{proof}
	First, applying Lemma A.$5$ of \cite{rivierecrelle}, we deduce that there exists an integer $\theta_0\geq 1$ and $\vec{A}_0\in \C^{n}\setminus\ens{0}$ such that 
	\begin{align}\label{limitaylor}
		&\phi_{\infty}(z)=\Re\left(\vec{A}_0z^{\theta_0}\right)+o(|z|^{\theta_0})\nonumber\\
		&\p{z}\phi_{\infty}(z)=\frac{\theta_0}{2}\vec{A}_0z^{\theta_0-1}+o(|z|^{\theta_0-1}).
	\end{align}
	As in the proof of Theorem \ref{neckfine}, we introduce an extension of $\tilde{\n_k}:B(0,1)\rightarrow \mathscr{G}_{n-2}(\R^n)$ of $\n_k:\Omega_k=B_1\setminus \bar{B}_{\rho_k}(0)\rightarrow \mathscr{G}_{n-2}(\R^n)$ such that
	\begin{align*}
	\left\{\begin{alignedat}{1}
	&\tilde{\n_k}=\n\qquad\text{on}\;\, \Omega_k=B_1\setminus B_{\rho_k}(0)\\
	&\np{\D\tilde{\n_k}}{2,\infty}{B(0,1)}\leq C_0(n)\np{\D\n_k}{2,\infty}{\Omega_k}.
	\end{alignedat}\right.
	\end{align*}
	Therefore, by Lemma IV.$3$ of \cite{quanta}, there exists a constant $C_1(n)$ and a Coulomb moving frame $(\f_{k,1},\f_{k,2})\in W^{1,2}(B(0,1),S^{n-1})\times W^{1,2}(B(0,1),S^{n-1})$ of $\tilde{\n_k}$ such that
	\begin{align}\label{controlled}
	\left\{\begin{alignedat}{1}
	&\tilde{\n_k}=\star(\f_{k,1}\wedge \f_{k,2})\qquad\dive\left(\f_{k,1}\cdot \D \f_{k,2}\right)=0\\
	&\np{\D \f_{k,1}}{2}{B(0,1)}+\np{\D\f_{k,2}}{2}{B(0,1)}\leq C_1(n)\np{\D\n_k}{2}{\Omega_k}.
	\end{alignedat}\right.
	\end{align}
	Now, define for all $j=1,2$ $\e_{k,j}=e^{-\lambda_k}\p{x_j}\phi_k$. As $\phi_k$ is conformal, $(\e_{k,1},\e_{k,2})$ is a Coulomb frame of $\n_k$ on $\Omega_k$. Furthermore, as $\tilde{\n_k}=\n_k$ on $\Omega_k$, both $(\f_{k,1},\f_{k,2})$ and $(\e_{k,1},\e_{k,2})$ are Coulomb frames of $\n_k$ on $\Omega_k$, so there exists a rotation $e^{i\theta_k}$ such that
	\begin{align}\label{multivalued}
	(\f_{k,1}+i\f_{k,2})=e^{i\theta_k}\left(\e_{k,1}+i\e_{k,2}\right).
	\end{align}
	Now, we let $f_{k,1},f_{k,2}$ be the vector fields such that 
	\begin{align}\label{fkj}
	d\phi_k(f_{k,j})=\f_{k,j}\qquad \text{for all}\;\, j=1,2.
	\end{align} Then observe as $\phi_k$ is conformal that
	\begin{align*}
	\delta_{i,j}=\s{\f_{k,i}}{\f_{k,j}}=\s{d\phi_k(f_{k,i})}{d\phi_k(f_{k,j})}=e^{2\lambda_k}\s{f_{k,i}}{f_{k,j}}
	\end{align*}
	so we have
	\begin{align*}
	\s{f_{k,i}}{f_{k,j}}=e^{-2\lambda_k}\delta_{i,j}.
	\end{align*}
	Likewise, if $(f_{k,1}^{\ast},f_{k,j}^{\ast})$ is the dual framing, we deduce that
	\begin{align}\label{dualframe}
	|f_{k,j}^{\ast}|=e^{\lambda_k}\qquad\text{for all}\;\, j=1,2.
	\end{align}
	Now, let $\mu_k$ the unique solution of 
	\begin{align*}
	\left\{
	\begin{alignedat}{2}
	\Delta\mu_k&=\D^{\perp}\f_{k,1}\cdot \D \f_{k,2}\qquad &&\text{in}\;\, B(0,1)\\
	\mu_k&=0\qquad &&\text{on}\;\, \partial B(0,1).
	\end{alignedat}\right.
	\end{align*}
	Furthermore, introduce the notation $\nu_k=\lambda_k-\mu_k$. Then $\nu_k$ is harmonic, \textbf{Step 1} of the analysis of the proof of Theorem \ref{neckfine} (actually, this is a direct consequence of Theorem \ref{neckfine}) shows that 
	\begin{align*}
	d_k=\frac{1}{2\pi}\int_{\partial B_{\rho_k}}\ast\, d\nu_k\conv{k\rightarrow \infty}\theta_0-1.
	\end{align*}
	As $\f_{k,1}\cdot \partial_{\nu}\f_{k,2}=0$ on $\partial B(0,1)$, we also have
	\begin{align*}
	d\mu_k=\ast(\f_{k,1}\cdot d\f_{k,2})
	\end{align*}
	Then we compute with $\Z_2$ indices for all $j\in \ens{1,2}$
	\begin{align*}
	d\mu_k\wedge f_{k,j}^{\ast}=(\ast d\mu_{k})\wedge (\ast \f_{k,j}^{\ast})=(-1)^j(\f_{k,1}\cdot d\f_{k,2})\wedge \f_{k,{j+1}}.
	\end{align*}
	Likewise, as in \cite{rivierecrelle}, we compute
	\begin{align*}
	df_{k,j}^{\ast}=(-1)^j\left(\f_{k,1}\cdot d\f_{k,2}\right)\wedge f_{k,j+1}^{\ast}.
	\end{align*}
	Therefore, we have
	\begin{align*}
	d\left(e^{-\mu_k}f_{k,j}^{\ast}\right)=0\qquad\text{in}\;\, \Omega_k\qquad \text{for}\;\, j=1,2.
	\end{align*}
	In particular, by Stokes theorem, we have for all $\rho_k\leq r_1<r_2\leq 1$
	\begin{align*}
	0=\int_{B_{r_2}\setminus \bar{B}_{r_1}(0)}d\left(e^{-\mu_k}f_{k,j}^{\ast}\right)=\int_{\partial B_{r_2}}e^{-\mu_k}f_{k,j}^{\ast}-\int_{\partial B_{r_1}}e^{-\mu_k}f_{k,j}^{\ast}.
	\end{align*}
	Therefore, we introduce the constants $c_j\in \R$ defined for all $\rho_k\leq \rho\leq 1$ by
	\begin{align*}
	c_{k,j}=\int_{\partial B_{\rho}}e^{-\mu_k}f_{k,j}^{\ast}.
	\end{align*}
	Now, introduce the complex valued $1$-forms
	\begin{align*}
	f_{k,z}^{\ast}=f_{k,1}^{\ast}+if_{k,2}^{\ast}\qquad f_{k,\z}^{\ast}=f_{k,1}^{\ast}-if_{k,2}^{\ast},
	\end{align*}
	so that 
	\begin{align*}
	f_{k,1}^{\ast}=\frac{1}{2}\left(f_{k,z}^{\ast}+f_{k,\z}^{\ast}\right)\qquad f_{k,2}^{\ast}=\frac{1}{2i}\left(f_{k,z}^{\ast}-f_{k,\z}^{\ast}\right).
	\end{align*}
	Notice also that
	\begin{align*}
	\ast\, df_{k,z}^{\ast}=-if_{k,z}^{\ast}\qquad \text{and}\;\, \ast df_{k,\z}^{\ast}=if_{k,\z}.
	\end{align*}
	Furthermore, if 
	\begin{align}\label{fkz}
	\left\{
	\begin{alignedat}{1}
	&f_{k,z}=\frac{1}{2}\left(f_{k,1}-if_{k,2}\right)\\
	&f_{k,\z}=\frac{1}{2}\left(f_{k,1}+if_{k,2}\right)
	\end{alignedat}\right.
	\end{align}
	then for all smooth function $\varphi:\Omega_k\rightarrow \C$, we have
	\begin{align*}
	d\varphi&=d\varphi\cdot f_{k,1}\,f_{k,1}^{\ast}+d\varphi\cdot f_{k,2}\,f_{k,2}\\
	&=d\varphi \cdot f_{k,z}\, f_{k,z}^{\ast}+d\varphi\cdot f_{k,z} f_{k,\z}^{\ast}.
	\end{align*}
	Now, we introduce the differential form $\alpha\in \Omega^1(\R^2\setminus\ens{0})$
	\begin{align*}
	\alpha&=\frac{1}{2\pi}\ast\, d\log|z|=\frac{1}{2\pi}\ast\, \left(\D\log|z|\cdot f_{k,z}\,f_{k,z}^{\ast}+\D\log|z|\cdot f_{k,\z}\,f_{k,\z}^{\ast}\right)\\
	&=\frac{1}{2\pi i}\D\log|z|\cdot f_{k,z}\,f_{k,z}^{\ast}-\frac{1}{2\pi i}\D\log|z|\cdot f_{k,\z}\,f_{k,\z}^{\ast}.
	\end{align*}
	In particular, notice that 
	\begin{align}\label{d=del+delb}
	\alpha+\frac{1}{2\pi i}d\log|z|=\frac{1}{2\pi i}\D\log|z|\cdot f_{k,z}f_{k,z}^{\ast}
	\end{align}
	As $\log $ is harmonic on $\R^2\setminus\ens{0}$, the differential form $\alpha$ is closed on $\Omega_k$ and we deduce that the $1$-form
	\begin{align*}
	\omega_{k,j}=e^{-\mu_k}f_{k,j}^{\ast}-c_{k,j}\alpha
	\end{align*}
	is also closed. Furthermore, as
	\begin{align*}
	\int_{\partial B_{\rho_k}}\omega_{k,j}=0,
	\end{align*}
	we deduce by Poincar\'{e} lemma that there exists $(\sigma_{k,1},\sigma_{k,2})\in W^{1,2}(\Omega_k,\R^2)$ such that 
	\begin{align*}
	d\sigma_{k,j}=\omega_{k,j}=e^{-\mu_k}f_{k,j}^{\ast}-c_{k,j}\alpha\qquad\text{for}\;\, j=1,2.
	\end{align*}
	Therefore, we deduce if $c_k=c_{k,1}+ic_{k,2}$ and $\sigma_k=\sigma_{k,1}+i\sigma_{k,2}$ that
	\begin{align*}
	d\sigma_k&=e^{-\mu_k}\left(f_{k,1}^{\ast}+if_{k,2}^{\ast}\right)-c_k\alpha\\
	&=\left(e^{-\mu_k}-\frac{c_k}{2\pi i}\D\log|z|\cdot f_{k,z}\right)\left(f_{k,1}^{\ast}+f_{k,2}^{\ast}\right)f_{k,z}^{\ast}+\frac{c_k}{2\pi i}\D\log|z|\cdot f_{k,\z}\,f_{k,\z}^{\ast}. 
	\end{align*} 
	This implies by \eqref{d=del+delb} that
	\begin{align}\label{hol}
	d\left(\sigma_k-\frac{c_k}{2\pi i}\log|z|\right)=\left(e^{-\mu_k}-\frac{c_k}{\pi i}\D\log|z|\cdot f_{k,z}\right)f_{k,z}^{\ast}.
	\end{align}
	Therefore, the function  
	\begin{align*}
	\tau_k=\sigma_k-\frac{c_k}{2\pi i}\log|z|
	\end{align*}
	is holomorphic. Now, let $\left(\frac{\partial}{\partial\tau_{k,1}},\frac{\partial}{\partial\tau_{k,2}}\right)$ be the dual basis of $(\tau_{k,1},\tau_{k,2})$, where $\tau_k=\tau_{k,1}+i\tau_{k,2}$. Then we define 
	\begin{align}\label{varphi}
	\varphi=e^{-\mu_k}-\frac{c_k}{\pi i}\D\log|z|\cdot f_{k,z},
	\end{align}
	and we notice that  \eqref{hol} implies that 
	\begin{align*}
	d(\tau_{k,1}+i\tau_{k,2})&=\left(\Re(\varphi)+i\Im(\varphi)\right)\left(f_{k,1}^{\ast}+if_{k,2}^{\ast}\right)\\
	&=\left(\Re(\varphi)f_{k,1}^{\ast}-\Im(\varphi)f_{k,2}^{\ast}\right)+i\left(\Im(\varphi)f_{k,1}^{\ast}+\Re(\varphi)f_{k,2}^{\ast}\right)
	\end{align*}
	Therefore, we deduce that
	\begin{align*}
	\begin{pmatrix}
	d\tau_{k,1}\\
	d\tau_{k,2}
	\end{pmatrix}
	=\begin{pmatrix}
	\Re(\varphi) &-\Im(\varphi)\\
	\Im(\varphi) & \Re(\varphi)
	\end{pmatrix}
	\begin{pmatrix}
	f_{k,1}^{\ast}\\
	f_{k,2}^{\ast}.
	\end{pmatrix}
	\end{align*}
	This implies that 
	\begin{align}\label{dualbasis}
	\begin{pmatrix}
	\frac{\partial}{\partial \tau_{k,1}}\\
	\frac{\partial}{\partial \tau_{k,2}}
	\end{pmatrix}
	=\frac{1}{|\varphi|^2}\begin{pmatrix}
	\Re(\varphi) & \Im(\varphi)\\
	-\Im(\varphi) & \Re(\varphi)
	\end{pmatrix}
	\begin{pmatrix}
	f_{k,1}\\
	f_{k,2}
	\end{pmatrix}.
	\end{align}
	Now, defining
	\begin{align*}
	\frac{\partial}{\partial\tau_k}=\frac{1}{2}\left(\frac{\partial}{\partial \tau_{k,1}}-i\frac{\partial}{\partial\tau_{k,2}}\right)
	\end{align*}
	we compute thanks to \eqref{fkj} and \eqref{dualbasis}
	\begin{align*}
	\frac{\partial\phi_k}{\partial\tau_k}&=\frac{1}{2|\varphi|^2}d\phi_k\cdot\left(\Re(\varphi)f_{k,1}+\Im(\varphi)f_{k,2}-i\left(\Im(\varphi)f_{k,1}-\Re(\varphi)f_{k,2}\right)\right)\\
	&=\frac{1}{2|\varphi|^2}\left(\Re(\varphi)\f_{k,1}+\Im(\varphi)\f_{k,2}-i\left(\Im(\varphi)\f_{k,1}-\Re(\varphi)\f_{k,2}\right)\right)\\
	&=\frac{1}{2|\varphi|^2}\left(\left(\Re(\varphi)+i\Im(\varphi)\right)\f_{k,1}+\left(\Im(\varphi)-i\Re(\varphi)\right)\f_{k,2}\right)\\
	&=\frac{\varphi}{2|\varphi|^2}\left(\f_{k,1}-i\f_{k,2}\right)
	=\frac{1}{2\bar{\varphi}}\left(\f_{k,1}-i\f_{k,2}\right).
	\end{align*}
	Therefore, we deduce that 
	\begin{align}\label{idfin}
	\frac{e^{\lambda_k}}{2}\left(\e_{k,1}-i\e_{k,2}\right)\p{z}\phi_k=\frac{\partial\phi_k}{\partial\tau_k}\frac{\partial\tau_k}{\partial z}=\frac{\tau_k'(z)}{2\bar{\varphi}}\left(\f_{k,1}-i\f_{k,2}\right)
	\end{align}
	Now, recall by \eqref{multivalued} that there exists a rotation $e^{i\theta_k}$ (beware that the function $\theta_k$ is multi-valued) such that 
	\begin{align*}
	\f_{k,1}+i\f_{k,2}=e^{i\theta_k}\left(\e_{k,1}+i\e_{k,2}\right).
	\end{align*}
	Therefore, \eqref{idfin}, \eqref{idfin} and $\lambda_k=\mu_k+\nu_k$ imply that 
	\begin{align}\label{almosthol}
	e^{\lambda_k}\varphi=e^{\nu_k}+\frac{\bar{c_k}e^{\lambda_k}}{\pi i}\D\log|z|\cdot  f_{k,\z}=\tau'_{k}(z)e^{-i\theta_k}.
	\end{align}
    Recalling that 
    \begin{align*}
    	d_k=\frac{1}{2\pi}\int_{\partial B_{\rho_k}}\ast\, d\nu_k\conv{k\rightarrow \infty}\theta_0-1,
    \end{align*}
    we will now show that $d_k=\theta_0-1$ for $k$ large enough. First, recall so there exists a rotation $e^{i\theta_k}$ such that
    \begin{align}\label{multivalued2}
    (\f_{k,1}+i\f_{k,2})=e^{i\theta_k}\left(\e_{k,1}+i\e_{k,2}\right),
    \end{align}
    and there exists vector fields $f_{k,1},f_{k,2}$ such that 
    \begin{align}\label{fkj2}
    d\phi_k(f_{k,j})=\f_{k,j}\qquad \text{for all}\;\, j=1,2.
    \end{align} 
    To simplify the notations, we will now delete the subscript $k$ in the following formulas. Now, rewrite \eqref{multivalued2} as 
    \begin{align*}
    	\f_{1}+i\f_{2}=e^{i\theta}(\e_{1}+i\e_2)=\cos(\theta)\e_1-\sin(\theta)\e_2+i\left(\sin(\theta)\e_1+\cos(\theta)\e_2\right)
    \end{align*}
    so that 
    \begin{align*}
    	\left\{\begin{alignedat}{1}
    	\f_1&=\cos(\theta)\e_1-\sin(\theta)\e_2\\
    	\f_2&=\sin(\theta)\e_1+\cos(\theta)\e_2
    	\end{alignedat}\right.
    \end{align*}
    Now, write $f_{1}=(f_1^1,f_1^2)$, $f_2=(f_2^1,f_2^2)$, and observe that
    \begin{align*}
    	&d\phi(f_1)=e^{\lambda}f_1^1\e_1+e^{\lambda}f_2^2\e_2=\f_1=\cos(\theta)\e_1-\sin(\theta)\e_2\\
    	&d\phi(f_2)=e^{\lambda}f_2^1\e_1+e^{\lambda}f_2^2\e_2=\f_2=\sin(\theta)\e_1+\cos(\theta)e_2
    \end{align*} 
    implies that 
    \begin{align*}
    \left\{
    \begin{alignedat}{1}
    	f_1&=e^{-\lambda}(\cos(\theta),-\sin(\theta))\\
    	f_2&=e^{-\lambda}(\sin(\theta),\cos(\theta)).
    	\end{alignedat}\right.
    \end{align*}
    Therefore, we deduce that 
    \begin{align}\label{fstar}
    	\left\{\begin{alignedat}{1}
    	&f_1^{\ast}=e^{\lambda}\cos(\theta)dx_1-e^{\lambda}\sin(\theta)dx_2\\
    	&f_2^{\ast}=e^{\lambda}\sin(\theta)dx_1+e^{\lambda}\cos(\theta)dx_2.
    	\end{alignedat}\right.
    \end{align}
    Recalling the definitions (from \eqref{fkz})
    \begin{align*}
    	&c_j=\int_{\partial B_{\rho}}e^{-\mu}f^{\ast}_j\qquad j=1,2,\qquad
    	f_{\z}=\frac{1}{2}(f_{1}+if_2).
    \end{align*}
    Introducing
    \begin{align*}
    	c=-\frac{1}{2\pi i}(c_1-ic_2)
    \end{align*}
    we have for some holomorphic function $\varphi$ on $\Omega_k$ and for all $z\in \Omega_k$ (corresponding to $\tau'$ in the previous notations)
    \begin{align}\label{rast1}
    	e^{\nu}&=\varphi(z)e^{-i\theta}+2c\,e^{\lambda}\D\log|z|\cdot f_{\z}
    \end{align}
    Notice that $e^{i\theta}=\cos(\theta)+i\sin(\theta)$ implies that
    \begin{align}\label{rast2}
    	f_{\z}&=\frac{e^{-\lambda}}{2}\left((\cos(\theta),-\sin(\theta))+i(\sin(\theta),\cos(\theta))\right)\nonumber\\
    	&=\frac{e^{\lambda}}{2}\left(\cos(\theta)+i\sin(\theta),i\cos(\theta)-\sin(\theta)\right)=\frac{e^{-\lambda}}{2}\left(\cos(\theta)+i\sin(\theta),i(\cos(\theta)+i\sin(\theta))\right)\nonumber\\
    	&=\frac{e^{-\lambda+i\theta}}{2}\left(1,i\right),
    \end{align}
    Therefore, recalling the notation $z=x_1+ix_2$, \eqref{rast1} and \eqref{rast2} imply that
    \begin{align}\label{precised}
    	e^{\nu}&=\varphi(z)e^{-i\theta}+ce^{i\theta}\D\log|z|\cdot (1,i)
    	=\varphi(z)e^{-i\theta}+ce^{i\theta}\left(\frac{x_1}{|z|^2},\frac{x_2}{|z|^2}\right)\cdot (1,i)\nonumber\\
    	&=\varphi(z)e^{-i\theta}+ce^{i\theta}\frac{x_1+ix_2}{|z|^2}=\varphi(z)e^{-i\theta}+ce^{i\theta}\frac{z}{|z|^2}\nonumber\\
    	&=\varphi(z)e^{-i\theta}+\frac{ce^{i\theta}}{\z}
    \end{align}
    Now, as the left hand-side of \eqref{precised} is \emph{real}, taking imaginary parts of the right hand-side, we find that 
    \begin{align*}
    	\varphi(z)e^{-i\theta}-\bar{\varphi(z)}e^{i\theta}+\frac{ce^{i\theta}}{\z}-\frac{\bar{c}e^{-i\theta}}{z}=0.
    \end{align*}
    Multiplying this identity by $e^{i\theta}$, we deduce that 
    \begin{align*}
    	e^{2i\theta}\left(-\bar{\varphi(z)}+\frac{c}{\z}\right)+\varphi(z)-\frac{\bar{c}}{z}=0.
    \end{align*}
    This implies that 
    \begin{align*}
    	e^{2i\theta}=\frac{\varphi(z)-\dfrac{\bar{c}}{z}}{\bar{\varphi(z)}-\dfrac{c}{\z}}=\left(\frac{\varphi(z)-\dfrac{\bar{c}}{z}}{\left|\varphi(z)-\dfrac{\bar{c}}{z}\right|}\right)^2.
    \end{align*}
    Finally, as $e^{\nu}>0$, we deduce thanks to \eqref{precised} that 
    \begin{align*}
    	e^{i\theta}=\frac{\varphi(z)-\dfrac{\bar{c}}{z}}{\left|\varphi(z)-\dfrac{\bar{c}}{z}\right|}.
    \end{align*}
    Letting now $\psi$ be the \emph{holomorphic} function such that 
    \begin{align*}
    	\psi(z)=\varphi(z)-\frac{\bar{c}}{z},
    \end{align*}
    we deduce that 
    \begin{align}\label{rotation}
    	e^{i\theta}=\frac{\psi(z)}{|\psi(z)|}.
    \end{align}
    This implies readily that 
    \begin{align}\label{dtheta}
    	d\theta=\Im\left(\frac{\partial\psi}{\psi}\right)=\Im\left(\frac{\psi'(z)}{\psi(z)}dz\right).
    \end{align}
    Indeed, we have formally (in other words, the following expression must be understood as the equality of two multi-valued functions, \textit{i.e.} modulo $2\pi i$)
    \begin{align*}
    	i\theta=\log\left(\frac{\psi(z)}{|\psi(z)|}\right).
    \end{align*}
    Therefore, we have
    \begin{align}\label{hol1}
    	i\partial\theta&=\frac{|\psi(z)|}{\psi(z)}\bigg\{\frac{\psi'(z)}{|\psi(z)|}-\frac{1}{2}\psi(z)\psi'(z)\bar{\psi}(z)|\psi(z)|^{-3}\bigg\}dz=\frac{|\psi(z)|}{\psi(z)}\bigg\{\frac{\psi'(z)}{|\psi(z)|}-\frac{1}{2}\frac{\psi'(z)}{|\psi(z)|}\bigg\}dz\nonumber\\
    	&=\frac{1}{2}\frac{\psi'(z)}{\psi(z)}dz=\frac{1}{2}\frac{\partial\psi}{\psi}.
    \end{align}
    As $\theta$ is real, we deduce that 
    \begin{align}\label{hol2}
    	i\bar{\partial}\theta=\bar{-i\partial\theta}=-\frac{1}{2}\bar{\left(\frac{\partial\psi}{\psi}\right)}.
    \end{align}
    Using that $d=\partial+\bar{\partial}$, we deduce from \eqref{hol1} and \eqref{hol2} that 
    \begin{align*}
    	d\theta=\partial\theta+\bar{\partial}\theta=\frac{1}{2i}\left(\frac{\partial\psi}{\psi}-\bar{\left(\frac{\partial\psi}{\psi}\right)}\right)=\Im\left(\frac{\partial\psi}{\psi}\right).
    \end{align*}
    Finally, we deduce from \eqref{dtheta} that 
    \begin{align*}
    	\int_{\partial B_{\rho}}d\theta\in 2\pi \Z,
    \end{align*}
    Now, a classical computation shows that 
    \begin{align*}
         \ast\, d\nu=d\theta.
    \end{align*}
    This can be directly checked using the Coulomb condition, but as we have already used it to obtain the closeness of $e^{-\mu}f_{1}^{\ast}$ and $e^{-\mu}f_2^{\ast}$, we can also check this property with these $1$-forms. Recall that 
    thanks to \eqref{fstar}
    \begin{align*}
    	\left\{\begin{alignedat}{1}
    	&e^{-\mu}f_1^{\ast}=e^{\nu}\cos(\theta)dx_1-e^{\nu}\sin(\theta)dx_2\\
    	&e^{-\mu}f_2^{\ast}=e^{\nu}\sin(\theta)dx_1+e^{\nu}\cos(\theta)dx_2.
    	\end{alignedat}\right.
    \end{align*}
    Therefore, that $e^{-\mu}f_1^{\ast}$ be exact is equivalent to
    \begin{align*}
    	0=\left(\p{x_2}\nu\right)e^{\nu}\cos(\theta)-\left(\p{x_2}\theta\right)e^{\nu}\sin(\theta)+(\p{x_1}\nu)e^{\nu}\sin(\theta)+(\p{x_1}\theta)e^{\nu}\cos(\theta)
    \end{align*} 
    or (writing scripts for partial derivatives)
    \begin{align}\label{closed1}
    	(\nu_2+\theta_1)\cos(\theta)+\left(\nu_1-\theta_2\right)\sin(\theta)=0.
    \end{align}
    Likewise, the closeness of $e^{-\mu}f_2^{\ast}$ is equivalent to
    \begin{align}\label{closed2}
    	(-\nu_1+\theta_2)\cos(\theta)+(\nu_2+\theta_1)\sin(\theta)=0.
    \end{align} 
    Therefore, \eqref{closed1} and \eqref{closed2} are equivalent to the system
    \begin{align*}
    	\begin{pmatrix}
    	\cos(\theta)& \sin(\theta)\\
    	\sin(\theta)& -\cos(\theta)
    	\end{pmatrix}
    	\begin{pmatrix}
    	\nu_2+\theta_1\\
    	\nu_1-\theta_2
    	\end{pmatrix}=0.
    \end{align*}
    As
    \begin{align*}
    	\det\begin{pmatrix}
    	\cos(\theta)& \sin(\theta)\\
    	\sin(\theta)& -\cos(\theta)
    	\end{pmatrix}=-\cos^2(\theta)-\sin^2(\theta)=-1\neq 0,
    \end{align*}
    we deduce that 
    \begin{align}\label{almostlast}
    	\left\{\begin{alignedat}{1}
    	&\nu_2+\theta_1=0\\
    	&\nu_1-\theta_2=0
    	\end{alignedat}\right.
    \end{align}
    In other words, \eqref{almostlast} is equivalent to $\D\nu=\D^{\perp}\theta$, or
    \begin{align}\label{lastid}
    	\ast\,d\nu=d\theta.
    \end{align}
    Therefore, thanks to \eqref{limitaylor} and \eqref{almostlast}, we deduce that for $k$ large enough
    \begin{align*}
    	\frac{1}{2\pi}\int_{\partial B_{\rho}}\ast\, d\nu_k=\theta_0-1.
    \end{align*}
    This argument concludes the proof of the Proposition.
\end{proof}

This proof can be immediately generalised to domains punctured by several disks.

\begin{theorem}\label{multibubble}
	Let $\{\phi_k\}_{k\in \N}$ be a sequence of smooth conformal immersions from the disk $B(0,1)\subset \C$ into $\R^n$. Let $m\in \N$, and for all $1\leq j\leq m$, let $\{a_k^j\}_{k\in\N}\subset B(0,1)$, $\{\rho_k^j\}_{k\in \N}\subset (0,\infty)$ and define for $0<\alpha<1$ and $k$ large enough
	\begin{align*}
		\Omega_k=B(0,1)\setminus \bigcup_{j=1}^m\bar{B}(a_{k}^j,\rho_k^j),\qquad \Omega_k(\alpha)=B(0,\alpha)\setminus\bigcup_{j=1}^m\bar{B}(a_k^j,\alpha^{-1}\rho_k^j).
	\end{align*}
	Assume that for all $1\leq j\neq j'\leq m$, and all $0<\alpha<1$ we have $B_{\alpha^{-1}\rho_k^j}(a_k^j)\cap B_{\alpha^{-1}\rho_k^{j'}}(a_k^{j'})=\varnothing$ for $k$ large enough, and
	\begin{align*}
		\rho_k^j\conv{k\rightarrow \infty}0,\qquad a_{k}^j\conv{k\rightarrow \infty}0.
	\end{align*}
	Furthermore, assume that 
	\begin{align*}
	\sup_{k\in \N}\int_{\Omega_k}|\D\n_k|^2dx\leq \epsilon_1(n),\qquad \sup_{k\in \N}\np{\D\lambda_k}{2,\infty}{\Omega_k}<\infty
	\end{align*} 
	where $\epsilon_1(n)$ is given by the proof of Theorem \ref{neckfine}. Finally, assume that
	\begin{align*}
		\lim_{\alpha\rightarrow 0}\limsup_{k\rightarrow\infty}\int_{\Omega_k(\alpha)}|\D\n_k|^2dx=0
	\end{align*}
	and that there exists a $W^{2,2}_{\mathrm{loc}}(B(0,1)\setminus\ens{0})\cap  C^{\infty}(B(0,1)\setminus\ens{0})$ immersion $\phi_{\infty}$ such that
	\begin{align*}
	\log|\D\phi_{\infty}|\in L^{\infty}_{\mathrm{loc}}(B(0,1)\setminus\ens{0})
	\end{align*} 
	such that $\phi_k\conv{k\rightarrow  \infty}\phi_{\infty}$ in $C^l_{\mathrm{loc}}(B(0,1)\setminus\ens{0})$. For all $k\in \N$, let
	\begin{align*}
	e^{\lambda_k}=\frac{1}{\sqrt{2}}|\D\phi_k|
	\end{align*}
	be the conformal factor of $\phi_k$. Then there exists a positive integer $\theta_0\geq 1$, and for all $k\in \N$ integers $\theta_k^1,\cdots,\theta_k^m\in \Z$ such that for all $k\in \N$ large enough
	\begin{align*}
		\sum_{j=1}^m\theta_k^j=\theta_0-1,
	\end{align*}
	and for all $k\in \N$, there exists $1/2<\alpha_k<1$ and $A_k\in \R$ such that 
	\begin{align}\label{estinfty}
	\np{\lambda_k-\sum_{j=1}^j\theta_k^j\log|z-a_k^j|-A_k}{\infty}{\Omega_k(\alpha_k)}\leq \Gamma_{11}\left(\np{\D\lambda_k}{2,\infty}{\Omega_k}+\int_{\Omega_k}|\D\n_k|^2dx\right)
	\end{align}
	for some universal constant $\Gamma_{11}=\Gamma_{11}(n)$. Furthermore, we have for all $0< \rho_k\leq 1$ such that 
	\begin{align*}
		\bigcup_{j=1}^m\bar{B}(a_k^j,\rho_k^j)\subset B(0,\rho_k).
	\end{align*}
	and for all $k\in \N$ large enough
	\begin{align*}
	\frac{1}{2\pi}\int_{\partial B_{\rho_k}(0)}\ast\, d\nu_k=\theta_0-1.
	\end{align*}
	Finally, for all $k\in \N$ and $j\in \ens{1,\cdots,m}$, we have
	\begin{align*}
		\frac{1}{2\pi}\int_{\partial B_{\rho_k^j}(a_k^j)}\ast\, d\nu_k=\theta_k^j\in \Z.
	\end{align*}
\end{theorem}

\begin{proof}
	Indeed, the same argument shows that there exists a holomorphic function $\varphi_k$ on $\Omega_k$ and $c_k^1,\cdots,c_k^m\in \C$ such that 
	\begin{align*}
		e^{\nu_k}=\varphi_k(z)e^{-i\theta_k}+\sum_{j=1}^{m}\frac{c_k^je^{i\theta_k}}{\bar{z}-\bar{a_k^j}}
	\end{align*}
	and the same computation shows if
	\begin{align*}
		\psi_k(z)=\varphi_k(z)-\sum_{j=1}^{m}\frac{\bar{c_k^j}}{z-a_k^j}
	\end{align*}
	that 
	\begin{align*}
		e^{i\theta_k}=\frac{\psi_k(z
			)}{|\psi_k(z)|}.
	\end{align*}
	Therefore, we have
	\begin{align*}
		d\theta_k=\Im\left(\frac{\partial\psi_k}{\psi_k}\right)
	\end{align*}
	and for all $1\leq j\leq m$
	\begin{align*}
		\int_{\partial B_{\rho_k^j}(a_k^j)}d\theta_k\in 2\pi \Z.
	\end{align*}
	Furthermore, we have
	\begin{align*}
		\lim\limits_{k\rightarrow \infty}\int_{\partial B(0,1)}d\theta_k=2\pi(\theta_0-1)\geq 0.
	\end{align*} 
	Therefore, we have for $k$ large enough
	\begin{align*}
		\frac{1}{2\pi}\sum_{j=1}^m\int_{\partial B_{\rho_k^j}(a_k^j)}d\theta_k=\theta_0-1.
	\end{align*}
	Then, we deduce by the argument of Lemma V.$3$ of \cite{quanta} that there exists a universal constant $C_{0}=C_{0}(n)$ such that for all $k\in \N$ there exists $1/2<\alpha_k<1$ such that for all $k\in \N$ large  enough
	\begin{align}\label{estint}
		\np{\nu_k-\sum_{j=1}^{m}\theta_k^j\log|z-a_j|-A_k}{\infty}{\Omega_k(\alpha_k)}\leq C_0\left(\np{\D\lambda_k}{2,\infty}{\Omega_k}+\int_{\Omega_k}|\D\n_k|^2dx\right),
	\end{align}
	where for all $k\in \N$ and $1\leq j\leq m$
	\begin{align*}
		\frac{1}{2\pi}\int_{\partial B_{\rho_k^j}(a_k^j)}d\theta_k=\theta_k^j\in \Z.
	\end{align*}
	In particular, as $\mu_k\in L^{\infty}(B(0,1))$ we get the estimate \eqref{estinfty} from \eqref{estint} and $\np{\mu_k}{\infty}{B(0,1)}\leq C_1$ for some universal $C_1=C_1(\Lambda,n)$ (thanks to Wente's estimate), we deduce that  there exists a universal constant $C=C(n,\Lambda)$,
	where
	\begin{align*}
		\Lambda=\sup_{k\in \N}\left(\np{\D\lambda_k}{2,\infty}{\Omega_k}+\int_{\Omega_k}|\D\n_k|^2dx\right)
	\end{align*}
	 such that for all $k$ large enough and $z\in \Omega_k(1/2)$ (notice that $A_k$ is bounded by the same argument in \textbf{Step 2} of the proof of Theorem \ref{onebubble})
	\begin{align}\label{test}
		\frac{1}{C}\leq \frac{e^{\lambda_k(z)}}{\displaystyle\prod_{j=1}^{m}|z-a_k^j|^{\theta_k^j}}\leq C.
	\end{align} 
	This additional remarks completes the proof of the Proposition.
\end{proof}
\begin{rem}
	The integers $\theta_k^j$  \textit{a priori} depend on $k$, but we will see in the case of interest of bubbling of Willmore immersions, they must stabilise for $k$ large enough.
\end{rem}

\section{Singularity removability of weak limits of immersions : general case}\label{sec5}

The analysis of the preceding Section \ref{sec4} works equally well for multiple bubbles, with the exception of Lemma \ref{schwarz} which needs to be replaced by a similar lemma involving a domain punctured by several disks (we state it for $\C$-valued functions instead of $\C^n$-valued maps for simplicity).

\begin{lemme}\label{schwarz2}
	Let $0<r_1,\cdots,r_m<R<\infty$ be fixed radii, $a_1,\cdots,a_m\in B(0,R)$ be such that $\bar{B}(a_j,r_j)\subset B(0,R)$, $\bar{B}(a_j,r_j)\cap \bar{B}(a_k,r_k)=\varnothing$ for all $1\leq j\neq k\leq m$ and
	\begin{align}\label{hyp}
		4r_j<R-|a_j|\quad\text{for all}\;\, 1\leq j\leq m.
	\end{align} 
	Furthermore, define
	\begin{align*}
	\Omega=B(0,R)\setminus\bigcup_{j=1}^m\bar{B}(a_j,r_j)\qquad \Omega'=\bigcap_{j=1}^mB(a_j,R-|a_j|)\setminus \bar{B}(a_j,r_j).
	\end{align*}
	Let ${u}:\Omega\rightarrow \C$ be a holomorphic function and for all $1\leq j\leq m$ let $\delta_j\geq 0$ such that
	\begin{align*}
	\np{u}{\infty}{\partial B_{r_j}(a_j)}\leq \delta_j.
	\end{align*}
	Then we have for all $z\in \Omega'$
	\begin{align}\label{schwarzmultiple}
	|u(z)|&\leq \sum_{j=1}^m\frac{5}{R-|a_j|}\left(\frac{1}{m}\np{u}{\infty}{\partial B_R(0)}+2\delta_j\right)|z-a_j|+4\sum_{j=1}^{m}\delta_j\nonumber\\
	&+\sum_{j=1}^{m}\frac{5}{R-|a_j|}\left(\frac{1}{m}\np{u}{\infty}{\partial B_R(0)}+2\delta_j\right)\max_{1\leq j\neq k\leq m}\dist_{\mathscr{H}}(a_k,\partial B_{r_j}(a_j)),
	\end{align}
	where $\mathrm{dist}_{\mathscr{H}}$ is the Hausdorff distance.
\end{lemme}
\begin{proof}
	Write
	\begin{align}\label{gensch0}
		u=\sum_{j=1}^{m}u_j+v\qquad \text{on}\;\, \Omega=B(0,R)\setminus \bigcup_{j=1}^m\bar{B}(a_j,r_j)
	\end{align}
	where for all $1\leq j\leq m$
	\begin{align*}
		\left\{
		\begin{alignedat}{2}
		\bar{\partial}u_j&=0\qquad &&\text{in}\;\, B(0,R)\setminus \bar{B}(a_j,r_j)\\
		u_j&=\frac{1}{m}u\qquad &&\text{on}\;\, \partial B(0,R)\\
		u_j&=u\qquad &&\text{on}\;\, \partial B(a_j,r_j).
		\end{alignedat}\right.
	\end{align*}
	and
	\begin{align*}
	\left\{\begin{alignedat}{2}
		\bar{\partial}v&=0\qquad &&\text{in}\;\, B(0,R)\setminus\bigcup_{j=1}^m\bar{B}(a_j,r_j)\\
		v&=0\qquad && \text{on}\;\,\partial B(0,R)\\
		v&=-\sum_{k\neq j}^{}u_k\qquad &&\text{on}\;\, \partial B(a_j,r_j).
		\end{alignedat}\right.
	\end{align*}
	Thanks to Lemma \ref{schwarz}, for all $1\leq j\leq m$, we have thanks to the hypothesis \eqref{hyp} for all $z\in \Omega_j=B(a_j,R-|a_j|)\setminus\bar{B}(a_j,r_j)$,
	\begin{align}\label{gensch1}
		|u_j(z)|\leq \frac{5}{R-|a_j|}\left(\np{u_j}{\infty}{\partial B_{R-|a_j|}(a_j)}+\delta_j\right)|z-a_j|+2\delta_j.
	\end{align}
	Furthermore, by the maximum principle, we have
	\begin{align}\label{gensch2}
		\np{u_j}{\infty}{\partial B_{R-|a_j}(a_j)}&\leq \max\ens{\np{u_j}{\infty}{\partial B_R(0)},\np{u_j}{\infty}{\partial B_{r_j}(a_j)}}=\max\ens{\frac{1}{m}\np{u}{\infty}{\partial B_{R}(0)},\delta_j}\nonumber\\
		&\leq \frac{1}{m}\np{u}{\infty}{\partial B_R(0)}+\delta_j.
	\end{align}
	Therefore, \eqref{gensch1} and \eqref{gensch2} imply that for all $z\in \Omega_j$
	\begin{align}\label{gensch3}
		|u_j(z)|\leq \frac{5}{R-|a_j|}\left(\frac{1}{m}\np{u}{\infty}{\partial B_{R}(0)}+2\delta_j\right)|z-a_j|+2\delta_j.
	\end{align}
	Now, for all $k\neq j$ and $z\in \partial B_{r_j}(a_j)$, we have by \eqref{gensch3}
	\begin{align*}
		|u_k(z)|&\leq \frac{5}{R-|a_k|}\left(\frac{1}{m}\np{u}{\infty}{\partial B_{R}(0)}+2\delta_k\right)|z-a_k|+2\delta_k\\
		&\leq \frac{5}{R-|a_k|}\left(\frac{1}{m}\np{u}{\infty}{\partial B_{R}(0)}+2\delta_k\right)\dist_{\mathscr{H}}(a_k,\partial B_{r_j}(a_j))+2\delta_k,
	\end{align*}
	where $\dist_H$ is the Hausdorff distance.
	By definition of $v$, we deduce that for all $1\leq j\leq m$
	\begin{align*}
		\np{v}{\infty}{\partial B_{r_j}(a_j)}\leq \sum_{k\neq j}^{}\frac{5}{R-|a_k|}\left(\frac{1}{m}\np{u}{\infty}{\partial B_{R}(0)}+2\delta_k\right)\dist_{\mathscr{H}}(a_k,\partial B_{r_j}(a_j))+2\sum_{k\neq j}^{}\delta_k.
	\end{align*}
	Therefore, the maximum principle implies as $v=0$ on $\partial B_R(0)$ that
	\begin{align}\label{gensch4}
		\np{v}{\infty}{\Omega}\leq \max_{1\leq j\leq m}\left\{\sum_{k\neq j}^{}\frac{5}{R-|a_k|}\left(\frac{1}{m}\np{u}{\infty}{\partial B_{R}(0)}+2\delta_k\right)\dist_{\mathscr{H}}(a_k,\partial B_{r_j}(a_j))+2\sum_{k\neq j}^{}\delta_k\right\}.
	\end{align}
	Finally, we deduce by \eqref{gensch0}, \eqref{gensch3} and \eqref{gensch4} that for all $\displaystyle z\in \Omega'=\bigcap_{j=1}^m\Omega_j$, we have
	\begin{align*}
		|u(z)|&\leq \sum_{j=1}^{m}\frac{5}{R-|a_j|}\left(\frac{1}{m}\np{u}{\infty}{\partial B_R(0)}+2\delta_j\right)|z-a_j|+2\sum_{j=1}^{m}\delta_j\\
		&+\max_{1\leq j\leq m}\left\{\sum_{k\neq j}^{}\frac{5}{R-|a_k|}\left(\frac{1}{m}\np{u}{\infty}{\partial B_{R}(0)}+2\delta_k\right)\dist_{\mathscr{H}}(a_k,\partial B_{r_j}(a_j))+2\sum_{k\neq j}^{}\delta_k\right\}\\
		&\leq \sum_{j=1}^m\frac{5}{R-|a_j|}\left(\frac{1}{m}\np{u}{\infty}{\partial B_R(0)}+2\delta_j\right)|z-a_j|+4\sum_{j=1}^{m}\delta_j\\
		&+\sum_{j=1}^{m}\frac{5}{R-|a_j|}\left(\frac{1}{m}\np{u}{\infty}{\partial B_R(0)}+2\delta_j\right)\max_{1\leq j\neq k\leq m}\dist_{\mathscr{H}}(a_k,\partial B_{r_j}(a_j)).
	\end{align*}
    This concludes the proof of the Lemma.
\end{proof}

We first add a small lemma showing that the multiplicity of a bubble cannot be equal to $0$. In the case of one bubble, we know that it must be positive. However, for multiple bubbles, some multiplicities may be negative. 

\begin{lemme}
	Under the hypothesis of Theorem \ref{ta}, assume that $n=3,4$. Then the multiplicities $\theta_k^j$ obtained in neck regions by Proposition \ref{ta}  cannot vanish. 
\end{lemme}
\begin{proof}
	Indeed, as each bubble is conformally minimal, each non-compact end admits the following Taylor expansion
	\begin{align*}
		\phi(z)=\Re\left(\frac{\vec{A}_0}{z^m}\right)+O(|z|^{1-m}),
	\end{align*}
	for some $m\geq 1$. In particular, the conformal factor satisfies (up to scaling)
	\begin{align*}
		e^{2\lambda(z)}=\frac{1}{|z|^{2m+2}}\left(1+O(|z|)\right).
	\end{align*}
	Therefore, the metric becomes as $|z|\rightarrow \infty$
	\begin{align*}
		g=|z|^{2m}(1+O(|z|))
	\end{align*}
	with $m\neq 0$. 
\end{proof}

Now recall the main theorem of the paper.

\begin{theorem}
	Let $\Sigma$ be a closed Riemann surface, $\{\phi_k\}_{k\in\N}\subset \mathrm{Imm}(\Sigma,\R^n)$ be a sequence of Willmore immersions and assume that the conformal class of $\{\phi_k\}_{k\in \N}$ stays within a compact subset of the Moduli Space and that
	\begin{align*}
	\sup_{k\in \N}W(\phi_k)<\infty.
	\end{align*}
	Let $\phi_{\infty}:\Sigma\rightarrow \R^n$ be the weak sequential limit of $\{\phi_k\}_{k\in \N}$ and $\vec{\Psi}_i:S^2\rightarrow \R^n$, $\vec{\chi}_j:S^2\rightarrow \R^n$ be the bubbles such that
	\begin{align*}
	\lim\limits_{k\rightarrow \infty}W(\phi_k)=W(\phi_{\infty})+\sum_{i=1}^{p}W(\vec{\Psi}_i)+\sum_{j=1}^{q}(W(\vec{\chi_j})-4\pi\theta_j),
	\end{align*}
	where $\theta_j\in \N$ is the maximal multiplicity of $\vec{\chi}_j$. Then at every branch point $p$ of $\phi_{\infty}, \vec{\Psi}_i$ or $\vec{\chi_j}$ of multiplicity $\theta_0\geq 2$, the second residue $\alpha=\alpha(p)$ satisfies the inequality $\alpha(p)\leq \theta_0-2$.
\end{theorem}
\begin{proof}
	It is enough to treat the case of one bubble domain, and to consider $\phi_k:B(0,1)\rightarrow \R^n$, and to assume that there exists a branched Willmore disk $\phi_{\infty}:B(0,1)\rightarrow \R^n$ with a unique branch point at $z=0$ of multiplicity $\theta_0\geq 2$ such that 
	\begin{align*}
		\phi_{k}\conv{k\rightarrow }\phi_{\infty}\qquad \text{in}\;\, C^l_{\mathrm{loc}}(B(0,1)\setminus\ens{0})\;\, \forall l\in \N.
	\end{align*}
	Here, we define as in \cite{quanta}
	\begin{align*}
		\Omega_k(\alpha)=B_{\alpha}(0)\setminus \bigcup_{j=1}^m\bar{B}_{\alpha^{-1}\rho_k^j}(x_k^j)
	\end{align*}
	Notice that as \enquote{bubbles on bubbles} do not change the analysis as only the behaviour of the metric on $\partial B_{\alpha^{-1}\rho_k^j(x_k^j)}$ matters, so we can suppose that there is no further energy concentration. Now, for all $1\leq j\leq m$ and $0<\alpha<1$, introduce the sequence of Willmore disks
	\begin{align*}
		\vec{\Psi}_k^j:B(0,\alpha^{-1})&\rightarrow \R^n\\
		z&\mapsto e^{-\bar{\lambda_k^j}}\left(\phi_k(\rho_k^jz+x_{k}^j)-\phi_k(x_k^j)\right),
	\end{align*}
	where 
	\begin{align*}
		e^{\bar{\lambda_k^j}}=\dashint{B(x_k^j,\rho_k^j)}e^{\lambda_k(z)}d\mathscr{H}^1(z)=e^{\lambda_k(z_k^j)}
	\end{align*}
	for some $z_{k,j}\in B(x_k^j,\rho_k^j)$. Notice that the conformal factor of $\vec{\Psi}_k^j$ satisfies
	\begin{align}\label{rescaling}
		\lambda_{\vec{\Psi}_k^j}(z)=\lambda_k(\rho_k z+x_k^j)-\lambda_k(z_k^j).
	\end{align}
    Therefore, thanks to the uniform Harnack inequality of \cite{quanta}, we deduce that there exists a Willmore plane $\vec{\Psi}_{\infty}^j:\C\rightarrow \R^n$ such that 
    \begin{align*}
    	\vec{\Psi}_k^j\conv{k\rightarrow \infty}\vec{\Psi}_{\infty}^j\quad\text{in}\;\, C^l_{\mathrm{loc}}(\C)\;\,\text{for all}\;\, l\in \N.
    \end{align*} 
    Furthermore, the strong convergence implies that
    \begin{align*}
    	W(\vec{\Psi}_{\infty}^j)=\lim\limits_{\alpha\rightarrow 0}\lim\limits_{k\rightarrow \infty}\int_{B(0,\alpha^{-1})}|\H_{\vec{\Psi}_k^j}|^2d\mathrm{vol}_{g_{\vec{\Psi}_k^j}}<\infty.
    \end{align*}
    Therefore, after applying a stereographic projection $\C\rightarrow S^2$, $\vec{\Psi}_{\infty}^j$ extends to a branched Willmore sphere $S^2\rightarrow \R^n$, with at most one branch point. In particular, by the analysis of \cite{beriviere}, there exists $\theta_0^j\in \Z$ such that 
\begin{align}\label{asymtaylor}
	e^{\lambda_{\vec{\Psi}^j_{\infty}}(z)}=|z|^{\theta_0^j}\left(1+O\left(\frac{1}{|z|}\right)\right)\qquad \text{as}\;\, |z|\rightarrow \infty.
\end{align}
Furthermore, \cite{beriviere} also implies that there exists an integer $p\in \Z$ and $\vec{\Delta}_0\in \C^n\setminus\ens{0}$ such that
\begin{align*}
	\H_{\vec{\Psi}_{\infty}^j}=\Re\left(\frac{\vec{\Delta}_0}{z^{p}}\right)+O\left(|z|^{p-1}\right).
\end{align*}
Since $e^{\lambda_{\vec{\Psi}_{\infty}^j}}\H_{\vec{\Psi}_{\infty}^j}\in L^2(\C)$, we must have 
\begin{align*}
	p-\theta_0^j\geq 2,
\end{align*}
and we deduce that 
\begin{align*}
	e^{\lambda_{\vec{\Psi}^j_{\infty}}(z)}\H_{\vec{\Psi}_{\infty}^j}=O\left(\frac{1}{|z|^{2}}\right)\qquad \text{as }\;\, |z|\rightarrow \infty.
\end{align*}
Now, thanks to the corresponding version of \eqref{hl3}, we deduce that 
\begin{align}\label{multinecks}
	\limsup_{\alpha\rightarrow \infty}\left|e^{\lambda_k}\left(\H_k+2i\vec{L}_k\right)\right|=O(\alpha^2)\qquad \text{on}\;\, \partial B_{\alpha^{-1}\rho_k^j}(x_k^j).
\end{align}
This estimate is the replacement of \textbf{Step 3} from the proof of Theorem .
Finally, \eqref{rescaling} and \eqref{asymtaylor} imply that
\begin{align}\label{bubblemultiplicity}
	\frac{1}{2\pi}\int_{\partial B_{\alpha^{-1}\rho_k^j}(x_k^j)}\partial_{\nu}\lambda_k\,d\mathscr{H}^1\conv{k\rightarrow \infty}\theta_0^j\in \Z.
\end{align}
Notice that we also have
\begin{align*}
	\sum_{j=1}^m\theta_0^j=\theta_0-1.
\end{align*}
Now, applying Lemma \ref{schwarz2} to $\Omega_k(\alpha)$ shows thanks to \eqref{multinecks} and applying \textit{mutatis mutandis} \textbf{Step 4} and \textbf{Step 5} that $\vec{C}_0=0$, provided we can find a sequence of never-vanishing holomorphic functions $\varphi_k:\Omega_k(\alpha)\rightarrow \C$ such that for all $z\in \Omega_k(\alpha)$
\begin{align}\label{testhol}
	\frac{1}{C}\leq \frac{e^{\lambda_k}(z)}{|\varphi_k(z)|}\leq C
\end{align}
Now, if $(\vec{f}_{k,1},\vec{f}_{k,2})$ are extensions on $B(0,1)$ (see the proof of Theorem \ref{onebubble}) of the moving frame $(\e_{k,1},\e_{k,2})$ (defined on $\Omega_k(1)$) with controlled energy (\text{i.e.} satisfying ), we let 
\begin{align*}
\left\{
\begin{alignedat}{2}
\Delta\mu_k&=\D^{\perp}\f_{k,1}\cdot \D \f_{k,2}\qquad &&\text{in}\;\, B(0,1)\\
\mu_k&=0\qquad &&\text{on}\;\, \partial B(0,1).
\end{alignedat}\right.
\end{align*}
Then the Wente estimate shows that 
\begin{align*}
	\np{\D\mu_k}{2}{B(0,1)}&\leq \frac{1}{4}\sqrt{\frac{3}{\pi}}\np{\D\vec{f}_{k,1}}{2}{B(0,1)}\np{\D\vec{f}_{k,2}}{2}{B(0,1)}\\
	&\leq \frac{1}{8}\sqrt{\frac{3}{\pi}}\left(\np{\D\vec{f}_{k,1}}{2}{B(0,1)}^2+\np{\D\vec{f}_{k,2}}{2}{B(0,1)}^2\right)\\
	&\leq \frac{C_1}{4}\sqrt{\frac{3}{\pi}}\int_{\Omega_k(1)}|\D\n_k|^2dx\\
	\np{\mu_k}{\infty}{B(0,1)}&\leq \frac{1}{2\pi}\np{\D\vec{f}_{k,1}}{2}{B(0,1)}\np{\D\vec{f}_{k,2}}{2}{B(0,1)}\leq \frac{C_1}{4\pi}\int_{\Omega_k(1)}|\D\n_k|^2dx.
\end{align*}
Now, applying Theorem \ref{multibubble} to $\phi_k$ on $\Omega_k$, implies thanks to \eqref{bubblemultiplicity}
that for all $k\in \N$ large enough, we have for some universal constant $C=C(n,\Lambda)$ independent of $k$ and for all $z\in \Omega_k(\alpha)$
\begin{align*}
	\frac{1}{C}\leq \frac{e^{\lambda_k(z)}}{\displaystyle\prod_{j=1}^{m}|z-x_k^j|^{\theta_0^j}}\leq C.
\end{align*}
Taking in \eqref{testhol}
\begin{align*}
	\varphi_k(z)=\prod_{j=1}^{m}\left(z-x_k^j\right)^{\theta_0^j}
\end{align*}
permits to apply to conclude the proof of this step.
Therefore, by the preceding remarks this ends the argument  for $\phi_{\infty}$. For bubbles, by applying a sequence of conformal transforms we can arrange such that they appear as the \enquote{macroscopic surface} so that  other bubbles stay bubbles, while the initial sequence of immersion becomes a bubble. This concludes the proof of the theorem.
\end{proof}

\section{Distributional Willmore equation}\label{sec7}

       In this section, we define precisely what it means to satisfy the Willmore equation in the distributional sense at branched points, as we cannot define it with any test functions because of the high order of singularity allowed by the mean-curvature at branch points. 
       
       \begin{prop} Let $\phi:D^2\setminus\ens{0}\rightarrow \R^n$ be a  conformal Willmore immersion of finite total curvature having a branch point of order $\theta_0\geq 1$ at $z=0$, and assume that $\phi$ is not minimal. Let $m\in \Z$ and $\vec{C}_0\in \C^{n}\setminus\ens{0}$ be such that 
       \begin{align*}
       		\H=\Re\left(\frac{\vec{C}_0}{z^{m}}\right)+O(|z|^{1-m}\log|z|),
       \end{align*}
       and assume that the second residue $r=r(\phi,0)$ satisfies $r=\max\ens{m,0}=m\geq 0$. Define
       \begin{align*}
       \mathscr{D}_m(D^2,\R^n)=\mathscr{D}(D^2,\R^n)\cap\ens{\w: \w=\w(0)+\Re\left(\vec{\gamma}\,z^{m}\right)+O(|z|^{m+1})\;\, \text{for some}\;\, \vec{\gamma}\in \C^n},
       \end{align*}
       and $\mathscr{D}'_{m}(D^2,\R^n)$ its topological dual. Then we have
       \begin{align*}
       \Delta\H+4\,\Re\left(\p{\z}\left(|\H|^2\p{z}\phi+2\s{\H}{\H_0}\p{\z}\phi\right)\right)=4\pi \,\Re\left(\bs{\frac{\p{z}^{m}\vec{\delta}_0}{(m-1)!}}{\vec{C}_0}\right)-2\pi\s{\vec{\delta}_0}{\vec{\gamma}_0}\quad \text{in}\;\, \mathscr{D}'_{m}(D^2,\R^n),
       \end{align*}
       where $\vec{\delta}_0=(\delta_0,\cdots,\delta_0)\in \mathscr{D}'(D^2,\R^n)$ is the vectorial Dirac mass at $0$.
       \end{prop}

     	\begin{proof} 
     	Recall that for all Willmore immersion $\vec{\Psi}:\Sigma\rightarrow \R^n$, and all variation $\vec{w}\in W^{2,2}_{\iota}\cap W^{1,\infty}(\Sigma,\R^n)$, we have
     	\begin{align*}
     		DW(\phi)(\vec{w})=\int_{\Sigma}\bigg(\s{\Delta_g^N\vec{w}+\mathscr{A}(\vec{w})}{\H}+|\H|^2\s{d\phi}{d\w}_g\bigg)d\vg,
     	\end{align*}
     	where $\Delta_g^N$ is the normal Laplacian, and $\vec{A}(\w)$ is the Simon's operator
     	\begin{align*}
     		\mathscr{A}(\w)=-\frac{1}{2}\vec{\I}\res_g\left(d\phi\totimes d\w+d\w\totimes d\phi\right)
     	\end{align*}
     	Now, recall that the Willmore equation in divergence form is given on by 
     	\begin{align*}
     		d\,\Im\left(\partial \H+|\H|^2+2\,g^{-1}\otimes\s{\H}{\h_0}\totimes\bar{\partial}\phi\right)=0,\quad \text{on}\;\, D^2\setminus\ens{0}
     	\end{align*}
     	or equivalently 
     	\begin{align*}
     		\Delta\H+4\,\Re\left(\p{\z}\left(|\H|^2\p{z}\phi+2\s{\H}{\H_0}\p{\z}\phi\right)\right)=0\quad \text{on}\;\, D^2\setminus \ens{0}.
     	\end{align*}
     	Now, we want to understand what happens at $0$. At the equation makes sense in distributional sense, we know by a classical theorem of Schwartz that the missing term at $z=0$ in the Euler-Lagrange equation is a linear combination of derivatives of Dirac masses. Now, let $\vec{w}:\C\rightarrow \R^n$ be a compactly supported (in $D^2$) smooth variation of $\phi$, and $\vec{\gamma}_{j,k}\in \C^n$ be such that for all $m\geq 0$
     	\begin{align*}
     		\vec{w}(z)=\Re\left(\sum_{j,k=0}^{m}\vec{\gamma}_{j,k}z^j\z^k\right)+O(|z|^{m+1}).
     	\end{align*}
     	Then we compute
     	\begin{align*}
     		\int_{\C\setminus\bar{B}(0,\epsilon)}\s{\Delta\vec{w}}{\vec{H}}d\mathscr{L}^2&=-\int_{\partial B(0,\epsilon)}\left(\s{\partial_{\nu}\vec{w}}{\H}-\s{\w}{\partial_{\nu}\vec{H}}\right)d\mathscr{H}^1+\int_{\C\setminus\bar{B}(0,\epsilon)}\s{\w}{\Delta \H}d\mathscr{L}^2.
     	\end{align*}
     	Furthermore, observe that for all smooth function $f:\bar{B}(0,\epsilon)\rightarrow \R$, we have
     	\begin{align*}
     	\partial_{\nu}f&=\D f(x)\cdot\frac{x}{|x|}=\frac{1}{|z|}\left(\p{x_1}f\cdot x_1+\p{x_2}f\cdot x_2\right)
     	=\frac{1}{2|z|}\left(\left(z+\z\right)\left(\partial+\bar{\partial}\right)+\frac{\left(z-\z\right)}{i}i\left(\partial-\bar{\partial}\right)\right)f\\
     	&=\frac{1}{|z|}\left(z\cdot \partial+\z\cdot \bar{\partial}\right)f=\frac{2}{|z|}\,\Re\left(z\cdot \p{z}f\right).
     	\end{align*}
     	Now, recall that the volume form on $\partial B(0,\epsilon)$ is
     	\begin{align}
     	\frac{x_1dx_2-x_2dx_1}{|x|}=\frac{1}{4i|z|}\left(\left(z+\z\right) \left(dz-d\z\right)-(z-\z)(dz+d\z)\right)=\frac{1}{2i|z|}\left(\z dz-z d\z\right)=\frac{1}{|z|}\Im\left(\z dz\right).
     	\end{align}
     	Finally, we deduce that for all $f,g\in C^{\infty}(\bar{B}(0,\epsilon))$
     	\begin{align*}
     	\int_{\partial B(0,\epsilon)}g\,\partial_{\nu}f\,d\mathscr{H}^1=\Im\int_{\partial B(0,\epsilon)}g\,\frac{2}{|z|}\,\Re\left(z\cdot \p{z}f\right)\frac{\z}{|z|}dz=2\,\Im\int_{\partial B(0,\epsilon)}g\,\Re\left(z\cdot \p{z}f\right)\frac{dz}{z}.
     	\end{align*}
 Now, we trivially have
 \begin{align*}
    \Re\left(z\cdot \p{z}\H\right)&=-\frac{m}{2}\Re\left(\frac{\vec{C}_0}{z^{m}}\right)+O(|z|^{m+1}),\qquad
 	\Re(z\cdot\p{z}\vec{w})%&=\Re\left(\frac{1}{2}z\sum_{j,k=0}^{m}j\vec{\gamma}_{j,k}z^{j-1}\z^k+\frac{1}{2}z\sum_{j,k=0}^{m}k\bar{\vec{\gamma}_{j,k}}z^{k-1}\z^j)\right)+O(|z|^{m+1})\\
 	=\frac{1}{2}\Re\left(\sum_{j,k}^{m}(j+k)\vec{\gamma}_{j,k}z^j\z^k\right)+O(|z|^{m+1}) 	
 	\end{align*}
	Therefore, we find as for all $j,k\geq 0$
	\begin{align*}
		\int_{\partial B(0,\epsilon)}z^k\z^l\frac{1}{z^{m}}\frac{dz}{z}=2\pi i \epsilon^{2l}\delta_{k-m,l}
	\end{align*} 
	we obtain
	\begin{align*}
		&\int_{\partial B(0,\epsilon)}\s{\Re(z\cdot \p{z}\vec{w})}{\H}\frac{dz}{z}\\
		&=\int_{B(0,\epsilon)}\left(\frac{1}{8}(m+0)\s{\vec{\gamma}_{m,0}}{\vec{C}_0}+\frac{1}{8}(0+m)\s{\bar{\gamma_{0,m}}}{\vec{C}_0}+\frac{1}{8}(m+0)\s{\vec{\gamma}_{0,m}}{\bar{\vec{C}_0}}+\frac{1}{8}(0+m)\s{\bar{\vec{\gamma}_{m,0}}}{\bar{\vec{C}_0}}\right)\frac{dz}{z}+O(\epsilon^2)\\
		&=\frac{\pi i m}{2}\Re\left(\s{\vec{\gamma}_{m,0}+\bar{\vec{\gamma}_{0,m}}}{\vec{C}_0}\right)+F_1(\vec{\gamma}_{j,k})_{j+k<m}+O(\epsilon^2),
	\end{align*}
	where $F_1(\vec{\gamma}_{j,k})_{j+k<m}$ says that the other constant terms only depend on lower order jets of $\w$ (as they involve terms of $\H$ of lower degree).
	Likewise, we have
	\begin{align*}
		\int_{\partial B(0,\epsilon)}\s{\vec{w}}{\Re\left(z\cdot \p{z}\H\right)}\frac{dz}{z}&=-\frac{m}{8}\int_{\partial B(0,\epsilon)}\left(\s{\vec{\gamma}_{m,0}+\bar{\vec{\gamma}_{0,m}}}{\vec{C}_0}+\s{\bar{\vec{\gamma}_{m,0}}+\vec{\gamma}_{0,m}}{\bar{\vec{C}_0}}\right)\frac{dz}{z}+O(\epsilon^2)=\\
		&=-\frac{\pi i m }{2}\Re\left(\s{\vec{\gamma}_{m,0}+\bar{\vec{\gamma}_{0,m}}}{\vec{C}_0}\right)+F_2(\vec{\gamma}_{j,k})_{j+k<m}+O(\epsilon^2).
	\end{align*}
	Therefore,
	\begin{align*}
		-\int_{\partial B(0,\epsilon)}\left(\s{\partial_{\nu}\vec{w}}{\H}-\s{\w}{\partial_{\nu}\vec{H}}\right)d\mathscr{H}^1&=2\,\Im \int_{\partial B(0,\epsilon)}\left(\s{\Re(z\cdot \p{z}\vec{w})}{\H}-\s{\vec{w}}{\Re\left(z\cdot \p{z}\H\right)}\right)\frac{dz}{z}\\
		&=2\pi m\,\Re\left(\s{\vec{\gamma}_{m,0}+\bar{\vec{\gamma}_{0,m}}}{\vec{C}_0}\right)+F(\vec{\gamma}_{j,k})_{j+k<m}+O(\epsilon^2).
	\end{align*}
	Now, notice that
	\begin{align*}
		\partial_z^{m}\vec{w}(0)=\frac{1}{2}m!\left(\vec{\gamma}_{m,0}+\bar{\vec{\gamma}_{0,m}}\right)
	\end{align*}
	so we have
	\begin{align*}
		2\pi m\,\Re\left(\s{\vec{\gamma}_{m,0}+\bar{\vec{\gamma}_{0,m}}}{\vec{C}_0}\right)=4\pi m\,\Re\left(\bs{\frac{\partial_z^{m}\vec{\delta}_0}{m!}}{\vec{C}_0}\right)(\w),
	\end{align*}
	and
	\begin{align}\label{imp}
		\int_{\C\setminus\bar{B}(0,\epsilon)}\s{\Delta\w}{\H}d\mathscr{L}^2=\int_{\C\setminus\bar{B}(0,\epsilon)}\s{\w}{\Delta\H}d\mathscr{L}^2+4\pi m\,\Re\left(\bs{\frac{\p{z}^{m}\vec{\delta}_0}{m!}}{\vec{C}_0}\right)(\vec{\w}).
	\end{align}
	Now, we observe as $\partial\phi=O(|z|^{\theta_0-1})$, $\h_0=O(|z|^{\theta_0-1})$ that
	\begin{align*}
		&|\H|^2\partial \phi=O(|z|^{\theta-2m-1}),\qquad
		2\,g^{-1}\otimes \s{\H}{\h_0}\otimes\bar{\partial}\phi=O(|z|^{-m}).
	\end{align*}
	Now, we compute
	\begin{align*}
		&|\H|^2=\frac{1}{2}\Re\left(\frac{\s{\vec{C}_0}{\vec{C}_0}}{z^{2m}}\right)+\frac{1}{2}\frac{|\vec{C}_0|^2}{|z|^{2m}}+O(|z|^{1-2m})\\
		&|\H|^2\partial\phi=\frac{1}{4}\s{\vec{C}_0}{\vec{C}_0}\vec{A}_0z^{\theta_0-2m-1}dz
		+\frac{1}{4}\bar{\s{\vec{C}_0}{\vec{C}_0}}\vec{A}_0\z^{\theta_0-1}\z^{-2m}dz+\frac{1}{2}|\vec{C}_0|^2\vec{A}_0z^{\theta_0-m-1}\z^{-m}dz+O(|z|^{\theta_0-2m})
	\end{align*}
	and we recall that by Codazzi identity
	\begin{align*}
		\bar{\partial}\h_0&=g\otimes \partial\H+K_g\,g\otimes\partial\phi
		=|z|^{2\theta_0-2}\times\left(-\frac{m}{2}\frac{\vec{C}_0}{z^{m+1}}\right)dz^2\otimes d\z+O(|z|^{\theta_0-1})\\
		&=-\frac{m}{2}\vec{C}_0z^{\theta_0-m-2}{\z^{\theta_0-1}}dz^2\otimes d\z+O(|z|^{\theta_0-1})
	\end{align*}
	Therefore, we find that there exists a holomorphic quadratic differential $\vec{\beta}=\vec{F}(z)dz^2$ such that $\vec{F}=O(|z|^{\theta_0-1})$ and
	\begin{align*}
		\h_0=-\frac{m}{2\theta}\vec{C}_0z^{\theta_0-m-2}\z^{\theta_0}dz^2+\vec{\beta}+O(|z|^{\theta_0})
	\end{align*}
    Now, assume that $m=\theta_0-1$ for simplicity. We find that there exists $\vec{D}_0\in \C^n$ such that
    \begin{align*}
    	\h_0=-\frac{\theta_0-1}{2\theta_0}\frac{\z^{\theta_0}}{z}+\vec{D}_0z^{\theta_0-1}+O(|z|^{\theta_0}).
    \end{align*}
    We compute
    \begin{align*}
    	\int_{\C\setminus\bar{B}(0,\epsilon)}\s{\p{\z}\vec{w}}{|\H|^2\partial_{z}\phi}\frac{d\z\wedge dz}{2i}&=\int_{\C\setminus\bar{B}(0,\epsilon)}d\left(\s{\w}{|\H|^2\partial_z\phi}\frac{dz}{2i}\right)-\int_{\C\setminus\bar{B}(0,\epsilon)}\s{\vec{w}}{\p{\z}\left(|\H|^2\p{z}\phi\right)}\frac{d\z\wedge dz}{2i}\\
    	&=-\frac{1}{2i}\int_{\partial B(0,\epsilon)}\s{\w}{|\H|^2\p{z}\phi}dz-\int_{\C\setminus\bar{B}(0,\epsilon)}\s{\vec{w}}{\p{\z}\left(|\H|^2\p{z}\phi\right)}\frac{d\z\wedge dz}{2i}.
    \end{align*}
    As $\p{z}\phi=\vec{A}_0z^{\theta_0-1}+O(|z|^{\theta})$
    we get
    \begin{align*}
    	|\H|^2\partial\phi&=\frac{1}{4}\s{\vec{C}_0}{\vec{C}_0}\vec{A}_0z^{1-\theta_0}
    	+\frac{1}{4}\bar{\s{\vec{C}_0}{\vec{C}_0}}\vec{A}_0z^{\theta_0-1}\z^{2-2\theta_0}+\frac{1}{2}|\vec{C}_0|^2\vec{A}_0\z^{1-\theta_0}+O(|z|^{\theta_0-2m}),
    \end{align*}
    we have
    \begin{align*}
    	&\int_{\partial B(0,\epsilon)}\s{\w}{|\H|^2\p{z}\phi}dz\\
    	&=\int_{\partial B(0,\epsilon)}\left(\frac{1}{8}\s{\vec{C}_0}{\vec{C}_0}\s{\vec{\gamma}_{\theta_0-2,0}+\bar{\vec{\gamma}_{0,\theta_0-2}}}{\vec{A}_0}+\frac{1}{4}|\vec{C}_0|^2\s{\vec{\gamma}_{0,\theta_0-2}+\bar{\vec{\gamma}_{\theta_0-2,0}}}{\vec{A}_0}\right)\frac{dz}{z}+F_4(\vec{\gamma}_{j,k})_{j+k\leq \theta_0-3}+O(\epsilon^2)\\
    	&=2\pi i\left(\frac{1}{8}\s{\vec{C}_0}{\vec{C}_0}\s{\vec{\gamma}_{\theta_0-2,0}+\bar{\vec{\gamma}_{0,\theta_0-2}}}{\vec{A}_0}+\frac{1}{4}|\vec{C}_0|^2\s{\vec{\gamma}_{0,\theta_0-2}+\bar{\vec{\gamma}_{\theta_0-2,0}}}{\vec{A}_0}\right)+F_4(\vec{\gamma}_{j,k})_{j+k\leq \theta_0-3}+O(\epsilon^2)
    \end{align*}
    and
    \begin{align*}
    	&\Re\left(-\frac{1}{2i}\int_{\partial B(0,\epsilon)}\s{\w}{|\H|^2\p{z}\phi}dz\right)\\
    	&=-\pi\Re\left(\frac{1}{8}\s{\vec{C}_0}{\vec{C}_0}\s{\vec{\gamma}_{\theta_0-2,0}+\bar{\vec{\gamma}_{0,\theta_0-2}}}{\vec{A}_0}+\frac{1}{4}|\vec{C}_0|^2\s{\vec{\gamma}_{0,\theta_0-2}+\bar{\vec{\gamma}_{\theta_0-2,0}}}{\vec{A}_0}\right)+F_4(\vec{\gamma}_{j,k})_{j+k\leq \theta_0-3}+O(\epsilon^2)
    \end{align*}
    Now, we have
    \begin{align*}
    	&\s{\H}{\h_0}=\bs{\frac{1}{2}\frac{\vec{C}_0}{z^{\theta_0-1}}+\frac{1}{2}\frac{\bar{\vec{C}_0}}{\z^{\theta_0-1}}}{-\frac{\theta_0-1}{2\theta_0}\vec{C}_0\frac{\z^{\theta_0}}{z}+\vec{D}_0z^{\theta_0-1}}+O(|z|)\\
    	&=-\frac{\theta_0-1}{4\theta_0}\s{\vec{C}_0}{\vec{C}_0}\left(\frac{\z}{z}\right)^{\theta_0}+\frac{1}{2}\s{\vec{C}_0}{\vec{D}_0}-\frac{\theta_0-1}{4\theta_0}|\vec{C}_0|^2\frac{\z}{z}+\frac{1}{2}\s{\bar{\vec{C}_0}}{\vec{D}_0}\left(\frac{z}{\z}\right)^{\theta_0-1}+O(|z|)
    \end{align*}
    while
    \begin{align*}
    	e^{-2\lambda}\p{\z}\phi=|z|^{2-2\theta_0}\left(1+O(|z|)\right)\times \left(\bar{\vec{A}_0}\z^{\theta_0-1}+O(|z|^{\theta_0})\right)=\bar{\vec{A}_0}\frac{1}{z^{\theta_0-1}}+O(|z|^{2-\theta_0}),
    \end{align*}
    so that
    \begin{align*}
    	g^{-1}\otimes\s{\H}{\h_0}\otimes\bar{\partial}\phi&=-\frac{\theta_0-1}{4\theta_0}\s{\vec{C}_0}{\vec{C}_0}\bar{\vec{A}_0}\frac{\z^{\theta_0}}{z^{2\theta_0-1}}+\frac{1}{2}\s{\vec{C}_0}{\vec{D}_0}\frac{\bar{\vec{A}_0}}{z^{\theta_0-1}}
    	-\frac{\theta_0-1}{4\theta_0}|\vec{C}_0|^2\bar{\vec{A}_0}\frac{\z}{z^{\theta_0}}\\
    	&+\frac{1}{2}\s{\bar{\vec{C}_0}}{\vec{D}_0}\frac{\bar{\vec{A}_0}}{\z^{\theta_0-1}}+O(|z|^{2-\theta_0}).
    \end{align*}
    Therefore, we find
    \begin{align*}
    	\int_{\C\setminus\bar{B}(0,\epsilon)}\s{\p{\z}\w}{\s{\H}{\H_0}\p{\z}\phi}\frac{d\z\wedge dz}{2i}=-\frac{1}{2i}\int_{\partial B(0,\epsilon)}\s{\w}{\s{\H}{\H_0}\p{\z}\phi}dz-\int_{\C\setminus \bar{B}(0,\epsilon)}\s{\w}{\p{\z}\left(\s{\H}{\H_0}\p{\z}\phi\right)}\frac{d\z\wedge dz}{2i}
    \end{align*}
    and
    \begin{align*}
    	&\int_{\partial B(0,\epsilon)}\s{\w}{\s{\H}{\H_0}\p{\z}\phi}dz\\
    	&=\int_{\partial B(0,\epsilon)}\left(\frac{1}{4}\s{\vec{C}_0}{\vec{D}_0}\s{\vec{\gamma}_{\theta_0-2,0}+\bar{\vec{\gamma}_{0,\theta_0-2}}}{\bar{\vec{A}_0}}+\frac{1}{4}\s{\bar{\vec{C}_0}}{\vec{D}_0}\s{\bar{\vec{\gamma}_{\theta_0-2,0}}+\vec{\gamma}_{0,\theta_0-2}}{{\bar{\vec{A}_0}}}\right)\frac{dz}{z}+F_5(\vec{\gamma}_{j,k})_{j+k\leq \theta_0-3}+O(\epsilon^2)\\
    	&=\frac{\pi i}{2}\left(\s{\vec{C}_0}{\vec{D}_0}\s{\vec{\gamma}_{\theta_0-2,0}+\bar{\vec{\gamma}_{0,\theta_0-2}}}{\bar{\vec{A}_0}}+\s{\bar{\vec{C}_0}}{\vec{D}_0}\s{\bar{\vec{\gamma}_{\theta_0-2,0}}+\vec{\gamma}_{0,\theta_0-2}}{{\bar{\vec{A}_0}}}\right)+F_5(\vec{\gamma}_{j,k})_{j+k\leq \theta_0-3}+O(\epsilon^2).
    \end{align*}
    Now, as
    \begin{align*}
    	DW(\phi)(\vec{w})=\int_{\Sigma}\bigg(\s{\Delta_g^N\vec{w}+\mathscr{A}(\vec{w})}{\H}+|\H|^2\s{d\phi}{d\w}_g\bigg)d\vg,
    \end{align*}
    and $|\vec{\I}|_g^2=O(|z|^{2-2\theta_0})$, $g=|z|^{2\theta_0-2}(1+O(|z|))|dz|^2$. Therefore, if $\w$ is a smooth variation of $\phi$ we have in the chart $D^2\subset \C$ near to a branch point $p$ of $\phi$
    and
    \begin{align*}
    \Delta\w=O(|z|^{\theta_0-2}).
    \end{align*}
    we have
    \begin{align*}
    	\sg{\Delta \w}{\H}=\s{\Delta\w}{\Re\left(\frac{\vec{C}_0}{z^{\theta_0-1}}\right)}=O\left(\frac{1}{|z|}\right)\in L^1(D^2)
    \end{align*}
    In particular, if 
    \begin{align*}
    	\w=\w(0)+\Re\left(\vec{\gamma} z^{\theta_0-1}\right)+O(|z|^{\theta_0}),
    \end{align*}
    we have
    \begin{align*}
    	\Delta \w=O(|z|^{\theta_0-2}),
    \end{align*}
    and
    \begin{align*}
    	\D\w=O(|z|^{\theta_0-2}).
    \end{align*}
    Therefore, we have
    \begin{align*}
    	\mathscr{A}(\w)=-\frac{1}{2}\vec{\I}\res_g \left(d\phi\totimes d\w+d\w\totimes d\phi\right)=O(|z|^{3-3\theta_0})\times O(|z|^{\theta_0-2})\times O(|z|^{\theta_0-1})=O(|z|^{-\theta_0})
    \end{align*}
    Furthermore, as $\H=O(|z|^{1-\theta_0})$, $e^{2\lambda}=|z|^{2\theta_0-2}(1+O(|z|))$, we obtain
    \begin{align*}
    	\s{\mathscr{A}(\w)}{\H}d\vg=O(|z|^{-\theta_0})\times O(|z|^{1-\theta_0})\times O(|z|^{2\theta_0-2})=O\left(\frac{1}{|z|}\right)\in L^1(D^2).
    \end{align*}
    Finally, we have $\H e^{\lambda}=O(1)$, and
    \begin{align*}
    	\s{d\phi}{d\w}_g=e^{-2\lambda}\s{\D\phi }{\D\w}=O(|z|^{2-2\theta_0})\times O(|z|^{\theta_0-1})\times O(|z|^{\theta_0-2})=O\left(\frac{1}{|z|}\right)\in L^1(D^2).
    \end{align*}
    Likewise, we have
    \begin{align*}
    	\Re\left(g^{-1}\otimes\left(\bar{\partial}\phi\totimes \bar{\partial}\w\right)\otimes \h_0+\left(\bar{\partial}\phi\totimes \partial\w\right)\H\right)=O(|z|^{\theta_0-2}),
    \end{align*}
    so that
    \begin{align*}
    	\bs{\Re\left(g^{-1}\otimes\left(\bar{\partial}\phi\totimes \bar{\partial}\w\right)\otimes \h_0+\left(\bar{\partial}\phi\totimes \partial\w\right)\H\right)}{\H}=O\left(\frac{1}{|z|}\right)\in L^1(D^2).
    \end{align*}
    As
    \begin{align*}
    	\Delta^{N}\w=\left(\Delta\w\right)^N+4\,\Re\left(g^{-1}\otimes\left(\bar{\partial}\phi\totimes \bar{\partial}\w\right)\otimes \h_0+\left(\bar{\partial}\phi\totimes \partial\w\right)\H\right)
    \end{align*}
    and $\H$ is normal, we deduce that
    \begin{align*}
    	\s{\Delta^N\w}{\H}=O\left(\frac{1}{|z|}\right)\in L^1(D^2).
    \end{align*}
    Finally, the previous computations imply that for all variation $\w$ of $\phi$ such that
    \begin{align*}
    	\w(z)=\w(0)+\Re\left(\vec{\gamma}z^{\theta_0-1}\right)+O(|z|^{\theta_0}),
    \end{align*}
    we have
    \begin{align*}
    	\s{\Delta_g^N\w}{\H}\in L^1(D^2,d\vg),\quad \s{\mathscr{A}(\w)}{\H}\in L^1(D^2,d\vg),\quad |\H|^2\s{d\phi}{d\w}_g\in L^{1}(D^2)
    \end{align*}
    so that the first variation
    \begin{align*}
    	D W(\phi)(\w)=\int_{\Sigma}\left(\s{\Delta_g^{N}\w+\mathscr{A}(\w)}{\H}+|\H|^2\s{d\phi}{d\w}_g\right)d\vg
    \end{align*}
    is well defined as each terms is in $L^1$ (and even $L^p$ for all $p<2$).
    
    Coming back to the previous computations, for all variation $\vec{\w}$ of $\phi$ such that
    \begin{align*}
    	\w(z)=\w(0)+\Re\left(\vec{\gamma}z^{\theta_0-1}\right)+O(|z|^{\theta_0})
    \end{align*}
     we have (as there are no lower order terms except for the first residue $\gamma_0$ of $\phi$)
    \begin{align}\label{dirac}
    	&\int_{\C\setminus\bar{B}(0,\epsilon)}\s{\Delta\w-4\,\Re\left(\s{\p{\z}\w}{\p{z}\phi}\H+2\s{\p{\z}\w}{\p{\z}\phi}\H_0\right)}{\H}d\mathscr{L}^2\nonumber\\
    	&=\int_{\C\setminus\bar{B}(0,\epsilon)}\s{\w}{\Delta\H+4\,\Re\left(|\H|^2\p{z}\phi+2\s{\H}{\H_0}\p{\z}\phi\right)}d\mathscr{L}^2+2\pi(\theta_0-1)\Re\left(\s{\vec{\gamma}}{\vec{C}_0}\right)-2\pi\s{\vec{w}(0)}{\vec{\gamma}_0}\nonumber\\
    	&=2\pi(\theta_0-1)\Re\left(\s{\vec{\gamma}}{\vec{C}_0}\right)-2\pi\s{\vec{w}(0)}{\vec{\gamma}_0},
    \end{align}
    where the expansion involving the residue comes the Taylor expansion (see \cite{beriviere})
    \begin{align*}
    	\H=\Re\left(\frac{\vec{C}_0}{z^{\theta_0-1}}\right)+\cdots-\vec{\gamma}_0\log|z|+O(1)
    \end{align*} 
    Now, taking $\w(0)=-\vec{\gamma}_0\in \R^n$, and $\vec{\gamma}=\bar{\vec{C}_0}$, we get
    \begin{align*}
    	&\int_{\C\setminus\bar{B}(0,\epsilon)}\s{\Delta\w-4\,\Re\left(\s{\p{\z}\w}{\p{z}\phi}\H+2\s{\p{\z}\w}{\p{\z}\phi}\H_0\right)}{\H}d\mathscr{L}^2=2\pi\left((\theta_0-1)|\vec{C}_0|^2+|\vec{\gamma}_0|^2\right)=0
    \end{align*}
    if and only if
    \begin{align*}
    	\Delta \H+4\,\Re\left(\p{\z}\left(|\H|^2\p{z}\phi+2\s{\H}{\H_0}\p{\z}\phi\right)\right)=0\quad \text{in}\;\, \mathscr{D}'_{\theta_0-1}(D^2,\R^n),
    \end{align*}
    where for all $\alpha\geq 0$,  $\mathscr{D}'_{\alpha}(D^2,\R^n)$ is the dual space of the space $\mathscr{D}_{\alpha}(D^2,\R^n)\subset \mathscr{D}(D^2,\R^n)$ of compactly supported functions $\w:D^2\rightarrow \R^n$
    such that
    \begin{align*}
    	\w=\w(0)+\Re\big(\vec{\gamma}\,z^{\alpha}\big)+O(|z|^{\theta_0}).
    \end{align*} 
    In other words, if
    \begin{align*}
    	\H=\Re\left(\frac{\vec{C}_0}{z^{\theta_0-1}}\right)+\cdots-\vec{\gamma}_0\log|z|+O(1)
    \end{align*}
    we have
    \begin{align*}
    	\Delta\H+4\,\Re\left(\p{\z}\left(|\H|^2\p{z}\phi+2\s{\H}{\H_0}\p{\z}\phi\right)\right)=4\pi \,\Re\left(\bs{\frac{\p{z}^{\theta_0-1}\vec{\delta}_0}{(\theta_0-2)!}}{\vec{C}_0}\right)-2\pi\s{\vec{\delta}_0}{\vec{\gamma}_0}\quad \text{in}\;\, \mathscr{D}'_{\theta_0-1}(D^2,\R^n),
    \end{align*}
    where $\vec{\delta}_0=(\delta_0,\cdots,\delta_0)\in \mathscr{D}'(D^2,\R^n)$ is the vectorial Dirac mass at $0$. The general case $m\leq \theta_0-1$ follows the exact same lines.
\end{proof}

    \begin{theorem}\label{distribution}
    	If a branched immersion $\phi:D^2\rightarrow\R^n$ with a unique branch point at $0$ of multiplicity $\theta_0\geq 1$ satisfies the Willmore equation in the distributional sense everywhere on $D^2$ then it is smooth.
    \end{theorem} 
    \begin{proof}
    	We can assume $\theta_0\geq 2$, as the removability for $\theta_0=1$ for established in \cite{beriviere}.
    	Define for all $m\geq 0$
    	\begin{align*}
    		\mathscr{D}_m(D^2,\R^n)=\mathscr{D}(D^2,\R^n)\cap\ens{\w: \w=\w(0)+\Re\left(\vec{\gamma}\,z^{m}\right)+O(|z|^{m+1})\;\, \text{for some}\;\, \vec{\gamma}\in \C^n},
    	\end{align*}
    	and $\mathscr{D}'_{m}(D^2,\R^n)$ its topological dual.
    	By the proof of the previous theorem, we deduce that for all $\w\in \mathscr{D}_{\theta_0-1}(D^2,\R^n)$ the first derivative
    	\begin{align}\label{fd}
    		DW(\phi)(\w)=\int_{\Sigma}\left(\s{\Delta_g^N\w+\mathscr{A}(\w)}{\H}+|\H|^2\s{d\phi}{d\w}_g\right)d\vg
    	\end{align}
    	is well-defined, and vanishes for all $\w\in \mathscr{D}_{\theta_0-1}(D^2,\R^n)$ if and only if the second residue $r(p)\leq \theta_0-1$ satisfies $
    		\vec{\gamma}_0=0\quad \text{and}\;\, r(p)\leq \theta_0-2.
    	$
    	Therefore, this proves the theorem for $\theta_0=2$. Assuming that $\theta_0\geq 3$, we deduce that
    	\begin{align*}
    		\H=\Re\left(\frac{\vec{C}_1}{z^{\theta_0-2}}\right)+O(|z|^{3-\theta_0}).
    	\end{align*}
    	Therefore, \eqref{fd} is well-defined for $\w\in \mathscr{D}_{\theta_0-2}(D^2,\R^n)$, \textit{i.e.} such that
    	\begin{align*}
    		\w=\w(0)+\Re\left(\vec{\gamma}z^{\theta_0-2}\right)+O(|z|^{\theta_0-1})
    	\end{align*}
    	as we have for all $\w\in \mathscr{D}_{\theta_0-2}(D^2,\R^n)$ the estimates $\H=O(|z|^{2-\theta_0})$ and $\Delta^N\w+e^{2\lambda}\mathscr{A}(\w)=O(|z|^{\theta_0-3})$, so that
    	\begin{align*}
    		&\s{\Delta^N\w+e^{2\lambda}\mathscr{A}(\w)}{\H}=O\left(\frac{1}{|z|}\right)\in L^1(D^2)\\
    		&|\H|^2\s{\D\phi}{\D\w}=O(|z|^{4-2\theta_0})\times O(|z|^{\theta_0-1})\times O(|z|^{\theta_0-3})=O(1)\in L^1(D^2).
    	\end{align*}
    	Now, thanks to \eqref{imp}, this implies that 
    	\begin{align*}
    		\Delta\H+4\,\Re\left(\p{\z}\left(|\H|^2\p{z}\phi+2\s{\H}{\H_0}\p{\z}\phi\right)\right)=4\pi \,\Re\left(\bs{\frac{\p{z}^{\theta_0-2}\vec{\delta}_0}{(\theta_0-3)!}}{\vec{C}_1}\right)\quad \text{in}\quad\;\, \mathscr{D}_{\theta_0-2}'(D^2,\R^{n})
    	\end{align*}
    	so we obtain $\vec{C}_1=0$ as $\phi$ must satisfy the Willmore equation in distributional sense. We see that by a direct induction this implies that $\alpha(0)=0$, and the smoothness of the immersion by \cite{beriviere}.
    \end{proof}

\section{Appendix}\label{appendix}

\subsection{Some basic properties of Lorentz spaces}\label{basiclorentz}

Fix a measured space $(X,\mu)$.
Define for all $0<t<\infty$ the measurable function $f_{\ast}$ on $(0,\infty)$ by
\begin{align*}
	f_{\ast}(t)=\inf\ens{\lambda>0: \mu(X\cap\ens{x:|f(x)|>\lambda})\leq t}
\end{align*}
and recall that for all $\lambda>0$
\begin{align*}
	\leb^1\left((0,\infty)\cap\ens{t:f_{\ast}(t)>\lambda}\right)=\mu\left(X\cap\ens{x:|f(x)|>\lambda}\right).
\end{align*}
In particular, using twice the usual slicing formula (valid for an arbitrary measure $\mu$ that need not be $\sigma$-finite), we find
\begin{align*}
	\np{f}{p}{X,\mu}&=p\int_{0}^{\infty}\lambda^{p}\mu\left(X\cap\ens{x:|f(x)|>\lambda}\right)\frac{d\lambda}{\lambda}=p\int_{0}^{\infty}\lambda^p\leb^1\left((0,\infty)\cap\ens{t:f_{\ast}(t)>\lambda}\right)\frac{d\lambda}{\lambda}\\
	&=\int_{0}^{\infty}f_{\ast}^p(t)dt
	=\np{f_{\ast}}{p}{(0,\infty),\leb^1}.
\end{align*}
To simplify notations we will often remove the reference to the measure $\mu$. 
This motivates the introduction of the following quasi-norm for $1<p<\infty$ and $1\leq q<\infty$
\begin{align}\label{norm1}
	\lnp{f}{p,q}{X}=\left(\int_{0}^{\infty}t^{\frac{q}{p}}f_{\ast}^q(t)\frac{dt}{t}\right)^{\frac{1}{q}}.
\end{align}
If we define
$
	f_{\ast\ast}(t)=\dfrac{1}{t}\int_{0}^{t}f_{\ast}(s)ds,
$
then the associated norm to $L^{p,q}$ is
\begin{align}\label{norm2}
	\np{f}{p,q}{X}=\left(\int_{0}^{\infty}t^{\frac{q}{p}}f_{\ast\ast}^q(t)\frac{dt}{t}\right)^{\frac{1}{q}},
\end{align}
and $(L^{p,q}(X,\mu),\np{\,\cdot\,}{p,q}{X})$ is a Banach space for all $1<p<\infty$ and $1\leq q< \infty$. Now, we have by Fubini's theorem for all $f\in L^{p,q}(X,\mu)$
\begin{align*}
	\np{f}{p,1}{X}&=\int_{0}^{\infty}t^{\frac{1}{p}}f_{\ast\ast}(t)\frac{dt}{t}=\int_{0}^{\infty}\int_{0}^{\infty}t^{\frac{1}{p}-2}f_{\ast}(s)\mathbf{1}_{\ens{0<s<t}}dsdt=\int_{0}^{\infty}f_{\ast}(s)\left(\int_{0}^{\infty}t^{\frac{1}{p}-2}\mathbf{1}_{\ens{0<s<t}}dt\right)ds\\
	&=\int_{0}^{\infty}f_{\ast}(s)\left(\int_{s}^{\infty}t^{\frac{1}{p}-2}dt\right)ds=\frac{p}{p-1}\int_{0}^{\infty}s^{\frac{1}{p}-1}f_{\ast}(s)ds=\frac{p}{p-1}\lnp{f}{p,1}{X}.
\end{align*}
Therefore, $\lnp{\,\cdot\,}{p,1}{X}$ is a norm for all $1<p<\infty$. Furthermore, notice that Fubini's theorem also shows (\cite{rivnotes}) that 
\begin{align}\label{norm3}
	\lnp{f}{p,q}{X}=p^{\frac{1}{q}}\left(\int_{0}^{\infty}\lambda^{q}\left(\mu\left(X\cap\ens{x:|f(x)|>\lambda}\right)^{\frac{q}{p}}\frac{d\lambda}{\lambda}\right)\right).
\end{align}
In particular, for $q=1$ each of the quantities \eqref{norm1}, \eqref{norm2} and \eqref{norm3} defines a norm on $L^{p,1}(X,\mu)$.
Finally, for $q=\infty$, we define the quasi-norm
\begin{align*}
	\lnp{f}{p,\infty}{X}=\sup_{\lambda>0}\,t\left(\mu\left(X\cap\ens{x:|f(x)|>\lambda}\right)\right)^{\frac{1}{p}}=\sup_{t>0}\,t^{\frac{1}{p}}f_{\ast}(t)
\end{align*}
and the norm
\begin{align*}
	\np{f}{p,\infty}{X}= \sup_{t>0}\,t^{\frac{1}{p}}f_{\ast\ast}(t)
\end{align*}
makes $(L^{p,\infty}(X),\np{\,\cdot\,}{p,\infty}{X})$ a Banach space (they are the classical Marcinkiewicz weak $L^p$ spaces). Notice however that $L^{1,\infty}$ is \emph{not} a Banach space.
We have the general inequality for all $1<p<\infty$
\begin{align*}
    \lnp{f}{p,\infty}{X}\leq \np{f}{p,\infty}{X}\leq \frac{p}{p-1}\lnp{f}{p,\infty}{X}.
\end{align*}
Finally, recall the inequality
\begin{align*}
	\left|\int_{X}fgd\mu\right|\leq \int_{0}^{\infty}f_{\ast}(t)g_{\ast}(t)dt.
\end{align*}
It implies that for all $1<p<\infty$
\begin{align*}
	\int_{0}^{\infty}f_{\ast}(t)g_{\ast}(t)dt=\int_{0}^{\infty}t^{\frac{1}{p}}f_{\ast}(t)t^{p'}g_{\ast}(t)\frac{dt}{t}\leq \lnp{g}{p',\infty}{X}\int_{0}^{\infty}t^{\frac{1}{p}}f_{\ast}(t)\frac{dt}{t}=\lnp{f}{p,1}{X}\lnp{g}{p',\infty}{X},
\end{align*}
while for all $1<p<\infty$ and $1\leq q< \infty$, we have by H\"{o}lder's inequality (applied to the Haar measure $\nu=\dfrac{dt}{t}$ on $(0,\infty)$)
\begin{align*}
	\int_{0}^{\infty}f_{\ast}(t)g_{\ast}(t)dt=\int_{0}^{\infty}t^{\frac{1}{p}}f_{\ast}(t)t^{\frac{1}{p'}}g_{\ast}(t)\frac{dt}{t}\leq \left(\int_{0}^{\infty}t^{\frac{p}{q}}f_{\ast}^q(t)\frac{dt}{t}\right)^{\frac{1}{q}}\left(\int_{0}^{\infty}t^{\frac{q'}{p'}}g_{\ast}^{q'}(t)dt\right)^{\frac{1}{q'}}=\lnp{f}{p,q}{X}\lnp{g}{p',q'}{X}.
\end{align*}
Therefore, we have for all $1<p<\infty$ and $1\leq q\leq \infty$
\begin{align}\label{id}
	\left|\int_{X}fg\, d\mu\right|\leq \lnp{f}{p,q}{X}\lnp{g}{p',q'}{X}\leq \np{f}{p,q}{X}\np{g}{p',q'}{X}
\end{align}
and one shows that for all $1<p<\infty$ and $1\leq q< \infty$, the dual space of $L^{p,q}(X,\mu)$ is $L^{p',q'}(X,\mu)$. In particular, \eqref{id} implies that  for all $1<p<\infty$ 
\begin{align*}
    \np{f}{p,1}{X}=\frac{p^2}{p-1}\int_{0}^{\infty}\mu\left(X\cap\ens{x:|f(x)|>t}\right)^{\frac{1}{p}}dt
\end{align*}
The main case of interest in this article is the $L^{2,1}$ norm, which now can be defined as
\begin{align*}
	\np{f}{2,1}{X}=4\int_{0}^{\infty}\left(\mu\left(X\cap\ens{x:|f(x)|>t}\right)\right)^{\frac{1}{2}}dt,
\end{align*}
and
\begin{align*}
	\left|\int_{X}fg\,d\mu\right|\leq \frac{1}{2}\np{f}{2,1}{X}\lnp{g}{2,\infty}{X}\leq \np{f}{2,1}{X}\np{g}{2,\infty}{X}.
\end{align*}
As
\begin{align*}
	\frac{1}{n}\np{\frac{1}{|x-y|^{n-1}}}{\frac{n}{n-1},\infty}{\R^n,\leb^n}=\lnp{\frac{1}{|x-y|^{n-1}}}{\frac{n}{n-1},\infty}{\R^n,\leb^n}=\alpha(n)^{\frac{n}{n-1}},
\end{align*}
we have for all open subset $\Omega\subset \R^n$ and $f\in L^{n,1}(\Omega)$, for all $y\in \R^n$
\begin{align}\label{lnweak}
	\int_{\Omega}\frac{|f(x)|}{|x-y|^{n-1}}d\leb^n(x)\leq \frac{n-1}{n^2}\np{f}{2,1}{\Omega}\np{\frac{1}{|x-y|^{n-1}}}{\frac{n}{n-1},\infty}{\R^n}= \frac{n-1}{n}\alpha(n)^{\frac{n-1}{n}}\np{f}{n,1}{\Omega}.
\end{align}
In particular, if $\Omega\subset \R^2$, we have
\begin{align}\label{l2weak}
	\int_{\Omega}\frac{|f(x)|}{|x-y|}d\leb^2(x)\leq \frac{\sqrt{\pi}}{2}\np{f}{2,1}{\Omega}.
\end{align}

\subsection{Extension operators on annuli}

The following result was used in \cite{quanta} and \cite{quantamoduli}.

\begin{theorem}\label{extop}
	Let $n\geq 2$, $\epsilon>0$ and $1+\epsilon<R<\infty$ and $\Omega_{R}=B(0,R)\setminus \bar{B}(0,1)$ be the associated annulus. Then there exists a linear extension operator  
	\begin{align*}
		T:\bigcup_{1\leq p<\infty}W^{1,p}(\Omega_{R})\rightarrow \bigcup_{1\leq p<\infty}W^{1,p}(B(0,R))
	\end{align*}
	such that for all $1\leq p<\infty$, there exists a universal constant $C_1(n,\epsilon)>0$ (independent of $R>1+\epsilon$) such that for all $1\leq p<\infty$
	\begin{align*}
		\wp{T u}{1,p}{B(0,R)}\leq C_1(n,\epsilon)\wp{u}{1,p}{\Omega_{R}}.
	\end{align*}
	Furthermore, for all $1<p<\infty$ and $1\leq q\leq \infty$, $T$ extends as a linear operator $W^{1,(p,q)}(\Omega_R)\rightarrow W^{1,(p,q)}(B(0,R))$ such that  for some universal constant $C_2(n,p,q,\epsilon)$
	\begin{align*}
		\wp{T u}{1,(p,q)}{B(0,R)}\leq C_2(n,p,q,\epsilon)\wp{u}{1,(p,q)}{\Omega_{R}}.
	\end{align*}
\end{theorem}
\begin{proof}
	The second assertion follows directly from the Stein-Weiss interpolation theorem (\cite{helein}, $3.3.3$). For the first part, construct by \cite{brezis}, IX.$7$ a linear extension operator $\tilde{T}$ such that for all $u \in W^{1,p}(B(0,1+\epsilon)\setminus \bar{B}(0,1))$, $\tilde{T} u\in W^{1,p}(B(0,1+\epsilon))$ and such that 
	\begin{align}\label{ext}
		\wp{\tilde{T}u}{1,p}{B(0,1+\epsilon)}\leq C(n,\epsilon)\wp{ u}{1,p}{B(0,1+\epsilon)}.
	\end{align}
	Now, if $u\in W^{1,p}(B(0,R))$, just consider the restriction $\bar{u}|B(0,1+\epsilon)\setminus \bar{B}(0,1)$, and define
	\begin{align*}
		Tu(x)=\left\{\begin{alignedat}{2}
		&u(x)\qquad &&\text{if}\;\, x\in B(0,R)\setminus \bar{B}(0,1+\epsilon)\\
		&\tilde{T}u(x)\qquad &&\text{if}\;\, x\in B(0,1+\epsilon).
		\end{alignedat} \right.
	\end{align*}
	As $\tilde{T}u=u$ on $B(0,1+\epsilon)\setminus \bar{B}(0,1)$, $T$ satisfies the claimed properties by \eqref{ext}. 
\end{proof}
\begin{rem}
	Although the norm of the norm of the operator $T:W^{1,p}(\Omega_R)\rightarrow W^{1,p}(B(0,R))$ does not depend on $1<p<\infty$, the norm of $T:W^{1,(p,q)}(\Omega_R)\rightarrow W^{1,(p,q)}(B(0,R))$ depends \emph{a priori} on $1<p<\infty$ and $1\leq q\leq \infty$, as the constant of the Stein-Weiss interpolation theorem depends on these parameters.
\end{rem}           

\subsection{Proof of Theorem \ref{tb}}\label{proofTB}
	
	Recall that Theorem \ref{tb} permits to restrict the possible bubbles as follows.  
	\begin{theorem}\label{tbprime}
		Let $\phi:S^2\setminus\ens{p_1,\cdots,p_d}\rightarrow \R^3$ be a \emph{non-planar} complete minimal surface with finite total curvature arising as a bubble of a sequence Willmore immersions (or as a conformal image of this bubble). Then  
		\begin{align*}
		\int_{S^2}K_gd\vg\leq -8\pi.
		\end{align*}
		with equality if and only if $\phi$ is either a minimal sphere with exactly one end of multiplicity $5$ whose Gauss map is ramified at its end, or $\phi$ is a minimal sphere with one flat end and one end of multiplicity $3$. 
	\end{theorem}
	\begin{proof}
		The catenoid had a non-zero flux so \cite{classification} implies that its conformal images cannot be bubbles of Willmore immersions. We first show the property mentioned before the Theorem. Assume that a complete minimal surface $\phi:\Sigma\setminus \ens{p_1,\cdots,p_d}\rightarrow \R^n$ with finite total curvature admits the following expansion (notice that there cannot be any $\log$ in the expansion as the flux must vanish by \cite{classification})
		\begin{align}\label{weierstrass}
		\phi(z)=\Re\left(\frac{\vec{A}_0}{z^{m}}+\frac{\vec{A}_1}{z^{m-k}}\right)+O(|z|^{k-m+1}),
		\end{align}
		for some $\vec{A}_0,\vec{A}_1\in \C^n\setminus\ens{0}$ and $m,k\geq 1$.
		As $\phi$ is conformal, we deduce that $\s{\p{z}\phi}{\p{z}\phi}=0$, and this implies that 
		\begin{align*}
		\s{\vec{A}_0}{\vec{A}_0}=0,\qquad \s{\vec{A}_0}{\vec{A}_1}=0.
		\end{align*}
		Therefore, we have
		\begin{align*}
		|\phi(z)|^2&=\frac{1}{2}\frac{|\vec{A}_0|^2}{|z|^{2m}}+\frac{1}{2}\Re\left(\frac{\s{\bar{\vec{A}_0}}{\vec{A}_1}}{z^{m-k}\z^m}\right)+O(|z|^{k-2m+1})
		=\frac{|\vec{A}_0|^2}{2|z|^{2m}}\left(1+\,\Re\left(\frac{\s{\bar{\vec{A}_0}}{\vec{A}_1}}{|\vec{A}_0|^2}z^k\right)+O(|z|^{k+1})\right)\\
		&=\frac{\beta_0}{|z|^{2m}}\left(1+2\,\Re\left(\alpha_0z^k\right)+O(|z|^{k+1})\right).
		\end{align*}
		Without loss of generality, assume that $\beta_0=1$. Then we get
		\begin{align*}
		\vec{\Psi}(z)&=\frac{\phi(z)}{|\phi(z)|^2}=|z|^{2m}\left(1-2\,\Re\left(\alpha_0z^k\right)+O(|z|^{k+1})\right)\times \Re\left(\frac{\vec{A}_0}{z^m}+\frac{\vec{A}_1}{z^{m-k}}+O(|z|^{k-m+1})\right)\\
		&=\Re\left(\left(1-2\,\Re\left(\alpha_0z^k\right)\right)\vec{A}_0\z^m+\vec{A}_1z^k\z^{m}\right)+O(|z|^{m+k+1})\\
		&=\Re\left(\bar{\vec{A}_0}z^m-\bar{\alpha_0}\bar{\vec{A}_0}z^{m+k}+\left(\bar{\vec{A}_1}-\alpha_0\bar{\vec{A_0}}\right)z^m\z^k\right)+O(|z|^{m+k+1}).
		\end{align*}
		Then the conformal factor of $\vec{\Psi}$ satisfies
		\begin{align*}
		e^{2\lambda}=2|\p{z}\vec{\Psi}(z)|^2=\frac{m}{2}|\vec{A}_0|^2|z|^{2m-2}\left(1+O(|z|)\right)=m|z|^{2m-2}\left(1+O(|z|)\right).
		\end{align*}
		Therefore, we deduce that 
		\begin{align*}
		e^{2\lambda}\H=\frac{1}{2}\Delta\vec{\Psi}(z)=2mk\,\Re\left(\left(\bar{\vec{A}_1}-\alpha_0\bar{\vec{A}_0}\right)z^{m-1}\z^{k-1}\right)+O(|z|^{m+k-1}),
		\end{align*}
		so we finally get
		\begin{align}\label{weierstrass2}
		\H=mk\,\Re\left(\frac{\vec{A}_1-\alpha_0\vec{A}_0}{z^{m-k}}\right)+O(|z|^{k-m+1}),
		\end{align}
		which implies that $r\leq m-k$, with $r=m-k$ if and only if $\vec{A}_1-\alpha_0\vec{A}_0\neq 0$. In other words, we have $r=m-k$ if and only if $\mathrm{Span}_{\C}(\vec{A}_0,\vec{A}_1)\simeq \C^2$.

		Recall that  the Enneper surface is conformally equivalent to a sphere minus a point, and admits a unique end of multiplicity $3$. This minimal surface admits the following Taylor expansion as its end that we take at $z=0$ (see \cite{dierkes})
		\begin{align*}
		\phi(z)=\Re\left(\frac{\vec{A}_0}{z^3}+\frac{\vec{A}_1}{z^2}+\frac{\vec{A}_2}{z}\right),\qquad \left\{\begin{alignedat}{1}
		&\vec{A}_0=\frac{1}{3}\left(-1,i,0\right)\\
		&\vec{A}_1=(0,0,1)\neq 0\\
		&\vec{A}_2=(1,i,0).
		\end{alignedat}\right.
		\end{align*}
		As $\vec{A}_0,\vec{A}_1$ are non-zero and linearly independent, \eqref{weierstrass2} implies that $\theta_0(p)=3$ and $r(p)=2=3-1$, so we are done. As the catenoid and the Enneper surface are the only complete minimal surfaces of total curvature $-4\pi$, we can now assume that the minimal surface has total curvature $-8\pi$.
		
		Now recall the Jorge-Meeks formula for a complete minimal surface $\phi:S^2\setminus\ens{p_1,\cdots,p_d}\rightarrow \R^3$ with multiplicities $m_1,\cdots,m_d\geq 1$ at the ends $p_1,\cdots,p_d$
		\begin{align*}
		\int_{S^2}K_gd\vg=-4\pi\,\deg(\n) =-4\pi\left(0-1+\frac{1}{2}\sum_{j=1}^d(m_j+1)\right)
		\end{align*}
		we see that a total curvature equal to $-8\pi$ yields four possibilities (we denoted here by $\n$ the Gauss map $S^2\rightarrow S^2$).
		As $\n$ must have degree $2$, we deduce that 
		\begin{align*}
		2=\deg(\n)=0-1+\frac{1}{2}\sum_{j=1}^{d}(m_j+1)\geq d-1
		\end{align*}
		so that $d\leq 3$. Now, $d=1$ implies that $m_1=5$, $d=2$ implies that either $m_1=m_2=2$, or $m_1=1$ and $m_2=3$. Finally, $d=3$ implies that $m_1=m_2=m_3=1$, so we see that there are indeed exactly four possibilities.
		
		We denote in the following $(g,\omega)$ the Weierstrass data (where $g$ is the stereographic projection of the Gauss map and $\omega$ is a meromorphic $1$-form), such that 
		\begin{align*}
		\phi(z)=\Re\left(\int_{\ast}^z(1-g^2)\,\omega,\int_{\ast}^zi(1+g^2)\,\omega,\int_{\ast}^z2g\,\omega\right).
		\end{align*}
		As $\phi$ has no flux, the three $1$-forms
		$
		\omega,g\,\omega,g^2\omega
		$
		must have vanishing residue (in the usual sense). 
		
		\textbf{Case $\bm{1}$: the minimal immersion has three ends of multiplicity $\bm{1}$}. As there are no complete minimal surface of genus $0$ in $\R^3$ with three planar ends (\cite{bryant}), this minimal immersion must have at least one catenoid end, which is excluded. Actually, these surfaces are classified as the family of trinoids, which have three catenoid ends (\cite{lopez}).
		
		\textbf{Case $\bm{2}$: the minimal immersion has a unique end of multiplicity $\bm{m=5}$.}
		Then the minimal surface is conformally equivalent to $\C$ and there are two distinct Weierstrass data.
		
		\textbf{Sub-case $\bm{1}$}. The Weierstrass data are given here as 
		$
		g(z)=\lambda\left(z+c+\dfrac{1}{z}\right), \omega=-e^{i\theta}z^2dz,
		$
		where $\lambda\in \R\setminus\ens{0}$, $c\in \C$, $\theta\in \R$. In this parametrisation, the end is at $z=\infty$, so we make a change of variable to get it at $z=0$ so that 
		$
		g(z)=\lambda\left(z+c+\dfrac{1}{z}\right), \omega=e^{i\theta}\dfrac{dz}{z^4}.
		$ 
		Therefore, we deduce that 
		\begin{align*}
		\phi(z)=\Re\left(\frac{\vec{A}_0}{z^5}+\frac{\vec{A}_1}{z^4}+\frac{\vec{A}_2}{z^3}+\frac{\vec{A}_1}{z^2}+\frac{\vec{A}_4}{z}\right),\qquad 		\left\{\begin{alignedat}{1}
		&\vec{A}_0=\frac{1}{5}\lambda^2e^{i\theta}\left(1,-i,0\right)\\
		&\vec{A}_1=\frac{1}{2}\lambda^2 e^{i\theta}\left(c,-ic,\lambda^{-1}\right).
		\end{alignedat}\right.
		\end{align*}
		Therefore, $\vec{A}_0,\vec{A}_1\in \C^3\setminus\ens{0}$ are linearly independent and the compactification of $\phi$ has a unique branch point $p$ of order $\theta_0(p)=5$ and second residue $r(p)=4=5-1$, so we are done. Furthermore, we see that one need only compute the first two term of the Taylor expansion to conclude.
		
		\textbf{Sub-case $\bm{2}$.} The Weierstrass data are
		$
		g(z)=\lambda\left(\dfrac{1}{z^2}+c\right), \omega=e^{i\theta}\dfrac{dz}{z^2},
		$
		where $\lambda\in \R\setminus\ens{0}$, $c\in \C$ and $\theta\in \R$. Here, we see directly that $\omega,g\omega$ and $g^2\omega$ have no residue. Furthermore, we have
		\begin{align*}
		\phi(z)&=\Re\left(\frac{\vec{A}_0}{z^5}+\frac{\vec{A}_1}{z^3}+\frac{\vec{A}_2}{z}\right),\qquad \left\{\begin{alignedat}{1}
		&\vec{A}_0=\frac{1}{5}\lambda^2e^{i\theta}\left(1,-i,0\right)\\
		&\vec{A}_1=\frac{2}{3}\lambda e^{i\theta}\left(\lambda c,-i\lambda c,-1\right)
		\end{alignedat}\right.
		\end{align*}
		As $\vec{A}_0,\vec{A}_1\in \C^3\setminus\ens{0}$ are linearly independent, we deduce that 
		so we deduce that $r(0)=3=5-2$, which is admissible. 
		
		From now on, the minimal surfaces considered will have two ends, or equivalently be conformally equivalent to $\C\setminus\ens{0}$.
		
		\textbf{Case 3: the minimal immersion has a exactly two ends of multiplicities $\bm{m_1=m_2=2}$.} 
		
		One more, there are several possibilities for the Weierstrass data.
		
		\textbf{Sub-case 1.} The Weierstrass data are
		$
		g(z)=\lambda \dfrac{z^2+cz+d}{z+1},  \omega=e^{i\theta}\dfrac{(z+1)^2}{z^3}dz,
		$
		where $\lambda\in \R\setminus\ens{0}$, $-1=\lambda^2e^{2i\theta}(c^2+2d)$, $\theta\in \R$, $(1+c)e^{i\theta}\in \R$, $1-c+d\neq 0$.
		Here, we see directly that 
		%%\begin{align*}
		%\omega=e^{i\theta}\left(\frac{1}{z^3}+\frac{2}{z^2}+\frac{1}{z}\right)dz
		%\end{align*}
		$\omega$ has a non-zero residue equal to $e^{i\theta}\neq 0$, so we are done.
		
		\textbf{Sub-case 2.} The Weierstrass data are
		$
		g(z)=\lambda\left(z+c+\dfrac{1}{z}\right), \omega=i\dfrac{dz}{z}
		$
		where $\lambda\in \R\setminus\ens{0}$ and $-1=\lambda^2(c^2+2)$. Trivially, $\omega$ has a non-zero residue, so we are done.
		
		\textbf{Sub-case 3.} The Weierstrass data are $g(z)=\lambda z^2,  \omega=\dfrac{dz}{z^3}$. However, 
		$
		g\,\omega=\lambda\dfrac{dz}{z}
		$
		has a non-zero residue at $z=0$. This finishes the proof of the case $m_1=m_2=2$.
		
		\textbf{Case 4: the minimal immersion  has two ends, one of multiplicity $\mbox{1}$, and the other of multiplicity $\mbox{3}$.} 
		
		There are four different cases here.
		
		\textbf{Sub-case 1.} The Weierstrass data are
		$
		g(z)=\lambda\dfrac{z^2+c}{z+1}, \omega=\dfrac{(z+1)^2}{z^4}dz
		$
		where $\lambda\in \R\setminus\ens{0}$, $c\in \C\setminus\ens{-1}$. As we will be interested in the end of multiplicity $1$, taking $w=1/z$ yields
		$
		g(w)=\dfrac{cw^2+1}{w(w+1)},\omega=-(w+1)^2dw.
		$
		In particular, we have
		$
		g\,\omega=-\dfrac{dw}{w}\left(1+O(|w|)\right)
	    $
		and $g\,\omega$ has a non-zero residue. Notice that this is the minimal surface corresponding to the L\'{o}pez surface VI. 
		
		\textbf{Sub-case 2.} The Weierstrass data are
		$
		g(z)=\lambda\left(z+\dfrac{1}{z}\right), \omega=\dfrac{dz}{z^2}
	    $
		where $\lambda\in \R\setminus\ens{0}$. Then
		$
		g\,\omega=\lambda\left(\dfrac{1}{z^3}+\dfrac{1}{z}\right)dz
		$
		has a non-zero residue.
		
		\textbf{Sub-case 3.} The Weierstrass data are
		$
		g(z)=\lambda\left(z^2+1\right),\omega=\dfrac{dz}{z^4},
		$
		where  $\lambda\in \R\setminus\ens{0}$.
		Therefore, we have
		\begin{align*}
		\phi(z)&=\Re\left(\frac{\vec{A}_0}{z^3}+\frac{\vec{A}_1}{z}+\vec{A}_2z\right),\qquad \left\{\begin{alignedat}{1}
		&\vec{A}_0=-\frac{1}{3}\left((1-\lambda^2),i(1+\lambda^{2}),2\lambda\right)\\
		&\vec{A}_1=2\lambda^2\left(1,-i,-\lambda^{-1}\right)
		\end{alignedat}\right.
		\end{align*}
		As $\vec{A}_0,\vec{A}_1$ are linearly independent, we deduce that $r(0)=1=3-1$ is admissible.

		\textbf{Sub-case 4.} The Weierstrass data are
		$
		g(z)=\lambda\,z^2, \omega=\dfrac{dz}{z^4}
		$
		and we get 
		\begin{align*}
		\phi(z)=\Re\left(\frac{\vec{A}_0}{z^3}+\frac{\vec{A}_1}{z}+\vec{A}_2z\right),\qquad \left\{\begin{alignedat}{1}
		&\vec{A}_0=-\frac{1}{3}\left(1,i,0\right)\\
		&\vec{A}_1=-2\lambda(0,0,1)
		\end{alignedat}\right.
		\end{align*}
		As $\vec{A}_0,\vec{A}_1\in \C^3\setminus\ens{0}$ are linearly independent, we deduce that $r(0)=1=3-2$, and this implies that this surface is admissible.
	\end{proof}

\nocite{}
\bibliographystyle{plain}
\bibliography{biblio}

\end{document}